\documentclass[11pt,a4paper]{amsart}

\usepackage{amssymb, amsmath, amsthm, amscd, mathrsfs, dsfont}
\usepackage{enumerate}
\usepackage{xcolor}
\usepackage{tikz-cd}
\theoremstyle{plain} \newtheorem{theorem}{Theorem}[section]
\theoremstyle{plain} \newtheorem{corollary}[theorem]{Corollary}
\theoremstyle{plain} \newtheorem{proposition}[theorem]{Proposition}
\theoremstyle{plain}\newtheorem{lemma}[theorem]{Lemma}
\theoremstyle{definition} \newtheorem{definition}[theorem]{Definition}
\theoremstyle{definition}\newtheorem{example}[theorem]{Example}
\theoremstyle{remark}\newtheorem{remark}[theorem]{Remark}
\theoremstyle{definition}
\theoremstyle{remark}\newtheorem{remarks}[theorem]{Remarks}
\theoremstyle{definition}
\theoremstyle{definition}
\theoremstyle{definition}\newtheorem{warning}[theorem]{Warning}

\newcommand{\N}{{\mathbb{N}}}
\newcommand{\C}{{\mathbb{C}}}
\newcommand{\R}{{\mathbb{R}}}
\newcommand{\T}{{\mathbb{T}}}
\newcommand{\Z}{{\mathbb{Z}}}
\newcommand{\Pn}{{\mathbb{P}}}
\newcommand{\F}{{\mathscr{F}}}

\newcommand{\TT}{{\mathscr{T}}}
\newcommand{\XX}{{\mathscr{X}}}
\newcommand{\FF}{{\mathcal{F}}}
\newcommand{\Spaces}{{\mathbf{Spaces}}}
\newcommand{\NCSpaces}{{\mathbf{NCSpaces}}}
\newcommand{\CommAlgebras}{{\mathbf{Algebras}}}
\newcommand{\NCAlgebras}{{\mathbf{NCAlgebras}}}
\newcommand{\TTT}{{\mathcal{T}}}
\newcommand{\Stacks}{{\mathbf{Stacks}}}
\newcommand{\Groupoids}{{\mathbf{Groupoids}}}
\newcommand{\TTTT}{{\mathfrak{T}}}
\newcommand{\PP}{{\mathscr{P}}}
\newcommand{\BB}{{\mathscr{B}}}

\def\restriction#1#2{\mathchoice
              {\setbox1\hbox{${\displaystyle #1}_{\scriptstyle #2}$}
              \restrictionaux{#1}{#2}}
              {\setbox1\hbox{${\textstyle #1}_{\scriptstyle #2}$}
              \restrictionaux{#1}{#2}}
              {\setbox1\hbox{${\scriptstyle #1}_{\scriptscriptstyle #2}$}
              \restrictionaux{#1}{#2}}
              {\setbox1\hbox{${\scriptscriptstyle #1}_{\scriptscriptstyle #2}$}
              \restrictionaux{#1}{#2}}}
\def\restrictionaux#1#2{{#1\,\smash{\vrule height .8\ht1 depth .85\dp1}}_{\,#2}}
\numberwithin{equation}{section}

\def\uniondisjointe{\mathop\bigsqcup}

\usepackage{color}
\usepackage{graphicx}
\usepackage{epsfig}
\usepackage{hyperref}
\hypersetup{colorlinks = true, linkcolor=blue}

\usepackage{float}

\newcommand{\Torus}{\ensuremath{{\mathbb T}}} 
\newcommand{\nctorus}[2]{\ensuremath{{\mathscr{T}_{#1,#2}}}} 
\newcommand{\nctorusgerbe}[2]{\ensuremath{{{\mathscr{T}}^{cal}_{#1,#2}}}} 
\newcommand{\nctoric}[3]{\ensuremath{{\mathscr{X}_{#1,#2,#3}}}} 
\newcommand{\nctoricgerbe}[3]{\ensuremath{{\mathscr{X}^{cal}_{#1,#2,#3}}}} 
\newcommand{\modulitorus}[2]{\ensuremath{{\mathscr{M}_{\mathrm{torus}}^{#1,#2}}}} 
\newcommand{\modulitoric}[2]{\ensuremath{{\mathscr{M}_{\mathrm{toric}}^{#1,#2}}}} 
\newcommand{\modulitorusgerbe}[2]{\ensuremath{{\hat{\mathscr{M}}_{\mathrm{torus}}^{#1,#2}}}} 
\newcommand{\modulitoricgerbe}[2]{\ensuremath{{\hat{\mathscr{M}}_{\mathrm{toric}}^{#1,#2}}}} 
\newcommand{\fg}{\ensuremath{{\tilde T\,_m\kern-2pt\times_{\bar E_k}\C^d}}}	
\newcommand{\fgdec}{\ensuremath{{\tilde T^{cal}\,_{m^{cal}}\kern-2pt\times_{\bar E_k}\C^d}}}	
	
	\title{Quantum (Non-commutative) Toric Geometry: Foundations}
\author[L. Katzarkov \and E. Lupercio \and L. Meersseman \and A. Verjovsky]{Ludmil Katzarkov\and Ernesto Lupercio\and Laurent Meersseman\and Alberto Verjovsky}
\address{
	University of Miami\\
	Department of Mathematics\\
	PO Box 249085\\
	Coral Gables, FL 33124-4250}
\address{
	Departamento de Matem\'{a}ticas\\
	Cinvestav\\
	Av. IPN \#2508, Col. San Pedro Zacatenco, D.F., \underline{M\'{e}xico}}
\address{Laboratoire Angevin de Recherche en Math\'ematiques\\Universit\'{e} d'Angers\\ F-49045 Angers Cedex, \underline{France}}
\address{
	Instituto de Matem\'{a}ticas\\
	Universidad Nacional Autonoma de M\'{e}xico
	Cuernavaca, \underline{M\'{e}xico}}
\date{\today}

\begin{document}

\begin{abstract}
	In this paper, we will introduce \emph{Quantum Toric Varieties} which are (non-commutative) generalizations of ordinary toric varieties where all the tori of the classical theory are replaced by quantum tori.  Quantum toric geometry is the non-commutative version of the classical theory; it generalizes non-trivially most of the theorems and properties of toric geometry. By considering quantum toric varieties as (non-algebraic) stacks, we define their category and show that it is equivalent to a category of quantum fans. We develop a Quantum Geometric Invariant Theory (QGIT) type construction of Quantum Toric Varieties.
	
	Unlike classical toric varieties, quantum toric varieties admit moduli and we define their moduli spaces, prove that these spaces are orbifolds and, in favorable cases,  up to homotopy, they admit a complex structure.
\end{abstract}	
	
\maketitle
\tableofcontents

\section{Introduction} \label{intro}

The classical theory of toric geometry has found multiple applications in the resolution of problems in various fields of mathematics going from combinatorics to differential geometry.

The purpose of this paper is to present a wide-ranging generalization of toric geometry: in the same manner in which non-commutative geometry generalizes classical geometry, quantum toric geometry is the non-commutative version of the classical theory; it generalizes non-trivially most of the theorems and properties of toric geometry. 

Tori (both real and complex) are the building blocks of the classical theory; indeed, a classical $n$-complex dimensional compact, projective K\"ahler toric manifold $X$ can be defined as an equivariant, projective compactification of the $n$-complex dimensional torus $\T_\C^d:= \C^* \times \cdots \times \C^*$ (where $\C^*:= \C \setminus \{0\}$).

Real tori also play an important role in the classical theory: the real torus $\T_\R^d=S^1 \times \cdots \times S^1 \subset \C^* \times \cdots \times \C^* = \T_\C^d$ acts holomorphically on the whole of $X \supset \T_\C^d$. Thinking of the K\"ahler manifold $(X,g,J,\omega)$ as a symplectic manifold, the action of the real compact Lie group $\T_\R^d$ on $X$ is Hamiltonian, implying thus the existence of a continuous equivariant moment map $\mu$ with convex image $P$:
$$\mu \colon X \longrightarrow  P \subseteq \R^d \cong \mathrm{Lie}({\T_\R^d})^*.$$ \emph{A priori}, whenever $X$ is compact, $P$ is a compact, convex set but, for a toric variety $X$, $P$ turns out to be a convex, rational, Delzant polytope: that is, the combinatorial dual of $P$ is a triangulation  of the sphere $S^{d-1}$, and all the slopes of all the edges of $P$ are rational.

It is natural to consider $P$ as a stratified space $P = P_0 \amalg P_1 \amalg \cdots \amalg P_d$ where $P_i$ is the disjoint union of all facets of $P$ of dimension $i$; this stratification is inherited by $X$ via the moment map. To wit, the map $\mu|_{P_i} \colon X_i := \mu^{-1}(P_i) \longrightarrow P_i$ is a trivial real-torus bundle over $P_i$ identifying $X_i$ with the product $P_i \times \T_\R^i$, and so, $X = (P_0 \times \T_\R^0) \amalg \ldots \amalg (P_d \times \T_\R^d),$
reconstructing then $X$ as the disjoint union of real Lagrangian tori.

For more general toric varieties, fans (which can be thought in the polytopal case as the cone with vertex at the origin of the dual of the polytope) are used rather than polytopes, but still, an ubiquitous use of complex and real tori (often appearing in the theory in the guise of lattices on vector spaces), and their partial compactifications, are the basis of the classical theory.

The basic idea behind the field of \emph{Quantum Toric Geometry} is to replace all the tori appearing in classical toric geometry by quantum tori (also known as non-commutative tori): just in the same manner in which toric manifolds can be thought of as integrable systems, our quantum toric manifolds can, in turn, be interpreted as quantum integrable systems. And likewise, for the same essential reason that a version of mirror symmetry for toric varieties can be construed as a parametrized version of $T$-duality for tori, analogously, quantum toric manifolds will have a version of mirror symmetry in which the basic component is non-commutative $T$-duality.

From a slightly different point of view, quantum toric geometry can be thought of as a deformation (with deformation parameter $\hbar$) of the whole field of toric geometry (say, as presented in \cite{CoxBook}): while very many results from the classical theory have their counterparts in our quantum generalization, the proofs of such results are not entirely obvious. On the other hand, the flavor of our theory is familiar, for we encounter the usual suspects: quantum fans, quantum lattices, and the like. Furthermore, of course, the classical theory is a particular case of the quantum theory, as it should be.

Thus, the basic building block of our theory is a non-commutative deformation of the classical tori $\T_\R^i$ known as the \textit{quantum torus} $\TT_{\R,\hbar}^i$ (depending on a '\emph{real} deformation parameter $\hbar$): it is one of the most important and basic spaces in the field of non-commutative geometry \cite{Connes}.

 Let us recall what we mean by a non-commutative space. While an ordinary commutative space $X$ has an algebra of smooth complex valued functions that is commutative, a non-commutative space $\XX$ is a gadget (often, it is  not an ordinary topological space) whose algebra of smooth functions is non-commutative. In fact, roughly speaking\footnote{We are obviating various important analytical and categorical issues in this introduction: often, in non-commutative geometry, one speaks about $C^*$-algebras, rather than simply speaking about algebras. Also, in non-commutative geometry, we don't really consider morphisms of algebras as mappings among them, but instead, a morphism of algebras $A\to B$ will be an $A$-$B$-bi-module: the resulting notion of isomorphism delivers the concept of \emph{Morita equivalence} ($\sim_M$) of algebras (cf. \cite{nlab:morita_equivalence}). In this paper, though, Morita equivalence appears in an enriched manner: as the equivalence relation of groupoids that produces stacks \cite{moerdijk2002introduction}; rather than bimodules, in the groupoid case, Hilsum-Skandalis bibundles take their place \cite{hilsum1987morphismes}.}, we have the following diagram relating four categories\footnote{The upper arrow of the diagram is essentially a consequence of commutative Gelfand–Naimark theorem \cite{arveson2012invitation}.}:
\begin{center}
	\begin{tikzcd}
	\Spaces\arrow[r, "\cong", leftrightarrow]\arrow[d, hook]&
	\CommAlgebras\arrow[d, hook] \\
	\NCSpaces\arrow[r, "\cong", leftrightarrow]& \NCAlgebras/{\sim_M}
	\end{tikzcd}
\end{center}

The quantum 2-torus $\TTT_{\R,\hbar}^2\in\NCSpaces\cong\NCAlgebras/{\sim_M}$ (non-commutative algebras \emph{up to Morita equivalence}) is a good starting example, its algebra $A_\hbar$ of smooth functions (in $\NCAlgebras$) has two (periodic) generators $X$, $Y$ that don't quite commute but rather satisfy the relation:
$$XY = e^{2\pi i\hbar} YX.$$
The algebra $A_\hbar$ can be realized as an operator algebra first appearing in quantum mechanics\footnote{In fact, this equation is precisely the classical Born-Heisenberg-Jordan commutation relation \cite{born1925quantenmechanik}, \cite{born1926quantenmechanik} ,\cite{fedak20091925} in Weyl exponential form \cite{weyl1927quantenmechanik}.}. When we specialize the parameter $\hbar$ to be zero, we obtain a commutative algebra and, in fact, $\TTT_{\R,\hbar=0}^2 \cong \T_\R^2$, recovering the usual torus.

There is an important dichotomy for the parameter $\hbar$; the space $\TTT_{\R,\hbar}^2$ is truly non-commutative only when $\hbar$ is irrational; when $\hbar$ is rational, its algebra of functions is Morita equivalent to a commutative algebra.

A. Connes has pointed out a beautiful geometric interpretation for the non-commutative space $\TT_{\R,\hbar}^2$; it can be thought of as the space of leaves of a foliation (see Section 6 of \cite{connes2008walk}). The Kronecker foliation of slope $\hbar$ on $\T_\R^2$ (depicted in Fig. \ref{Kronecker}) consists on taking the foliation of the Euclidean plane $\R^2$ and projecting it up by the translation action of the integral lattice $\Z^2 \subset \R^2$. This is the same as considering the image of $\R^2$ (and its foliation) into $\C^2$ given by the map $E$, defined as:
$$ E:(x,y) \mapsto (\exp(2\pi i x),\exp(2\pi i y)),$$
and it is because of this that the exponential will play a fundamental role in our theory.

\begin{figure}[h]
	\centering
	\includegraphics[width=0.7\linewidth]{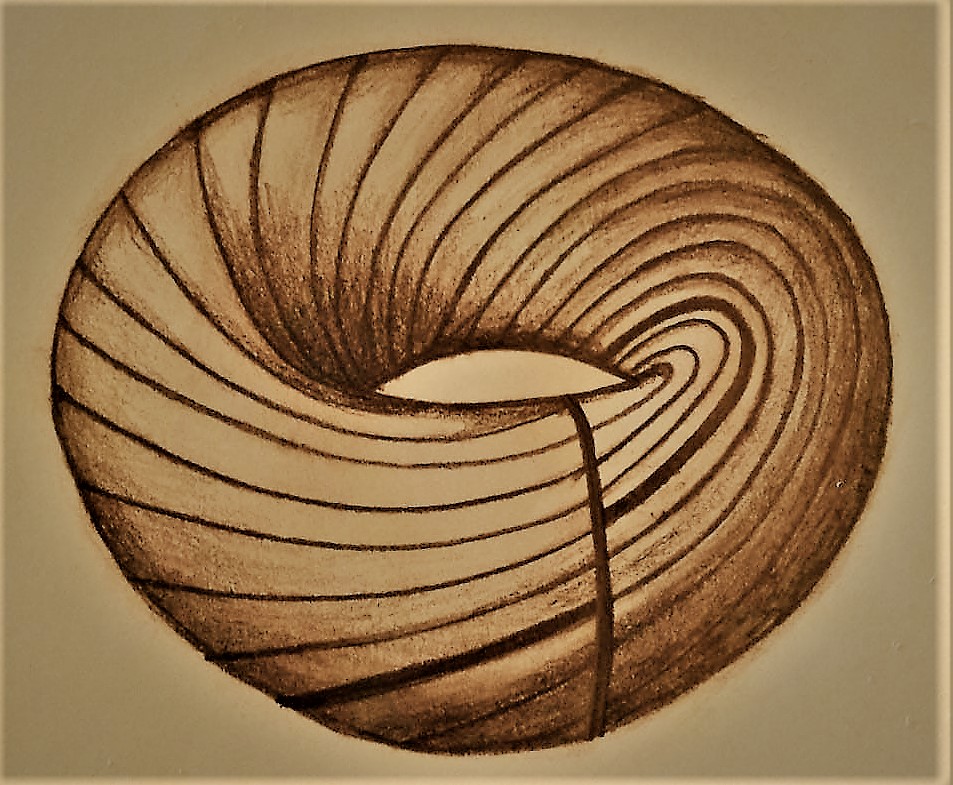}
	\caption[The Kronecker Foliation and its Holonomy]{The Kronecker foliation and its holonomy: the thick line represents an interval of a leaf in the Kronecker foliation going around once and thus defining a rotation from the transversal vertical circle into itself. We will denote the angle of rotation by $\hbar$.}\label{Kronecker}
\end{figure}

Whenever $\hbar=p/q$ is rational, this is a foliation of the real torus $\T_\R^2$ by circles (actually $(p,q)$-torus knots) but, otherwise, each leaf winds densely inside $\T_\R^2$. 

As a first approximation, we think of the leaf space of the Kronecker foliation as the quotient topological space $\TTTT(\hbar):=\T_\R^2/E(\hbar)$ where $E(\hbar):=\{E(x,\hbar x): x\in\R \})$ is the (possibly dense) leaf of the torus passing through the origin; it is also a normal subgroup of $\T_\R^2$, and the quotient is taken in the group sense. We could obtain the same quotient by considering only the transversal circle (the vertical circle in Fig. \ref{Kronecker} above). If $\rho_\hbar\colon S^1 \to S^1$ is the holonomy map that rotates the circle by an angle $\hbar$, and $\langle \rho_\hbar \rangle$  is the discrete group of rotations of the circle it generates (we have an infinite cyclic group $\langle \rho_\hbar \rangle \cong \Z$ whenever $\hbar$ is irrational, and a finite cyclic group otherwise),  then we have:
$$\TTTT(\hbar):=\T_\R^2/E(\hbar) \cong S^1 /\langle \rho_\hbar \rangle,$$
again, a dichotomy ensues: either $\hbar$ is rational and $\TTTT(\hbar)$ is a circle (the quotient of a torus by an embedded torus knot) or $\hbar$ is irrational and $\TTTT(\hbar)$ is a non-Hausdorff topological space.                                                                                                                                                                                                                                                                                                                                                                                                                                                                                                                                                                                                                                                                                                                                                                                                                                                                                                                                                                                          

When, in general, $\TTTT = T/ \sim $ is a non-Hausdorff topological space obtained as the quotient of a manifold $T$ divided by the action of a (possibly non-compact) group (really, any equivalence relation $\sim$ defined by a Lie groupoid action on $T$), there is, at least, two very fertile ways to enrich $\TTTT$ preserving some of the information of the geometric groupoid action on $T$ (and landing in nicer categories than that of possibly non-Hausdorff topological spaces); (1) by using non-commutative algebras (taking the non-commutative quotient as in section 4 of \cite{connes2008walk}), and (2) by using \emph{stacks} (sheafs of groupoids \cite{edidin2003communications}): from $(T,\sim)$ (thought of as a topological groupoid), we can obtain three related objects; (a) a non-Hausdorff topological space, (b) a non-commutative algebra $A_\TTT$, and (c) a stack $\TT$. From these, $\TT$ is the richer, it has more information about the groupoid $(T,\sim)$ than the other two objects; then, by applying the Connes convolution algebra mapping (\cite{cartier2001mad} page 5): 
\begin{center}
	\begin{tikzcd}
	\Groupoids\arrow[r, "C", rightarrow]\arrow[d, twoheadrightarrow]&
	\NCAlgebras\arrow[d, twoheadrightarrow] \\
	\Stacks\arrow[r, "C", rightarrow]& \NCSpaces
	\end{tikzcd}
\end{center}
here, it is useful to remember that $\Stacks\cong\Groupoids/{\sim_M}$ and that $\NCSpaces\cong\NCAlgebras/{\sim_M}$, moreover, the descending arrows consists in both cases in quotienting out Morita equivalences.

Let us consider the example of the quantum torus. Here we have:
\begin{center}
	\begin{tikzcd}
	(\T_\R^2,E(\hbar))\arrow[r, "C", mapsto]\arrow[d, twoheadrightarrow]&
	A_\hbar\arrow[d, twoheadrightarrow] \\
	\TT_{\R,\hbar}^2\arrow[r, "C", mapsto]& \TTT_{\R,\hbar}^2
	\end{tikzcd}
\end{center}
The \emph{dramatis personae} of this commutative diagram are as follows:
\begin{enumerate}[(i)]
	\item The (translation) Lie groupoid $(\T_\R^2,E(\hbar))$ whose manifold of objects is the torus $\T_\R^2$ (which happens to be a Lie group), and whose arrows $(t,s): t \mapsto t\cdot s$ are pairs of elements in $\T^2\times E(\hbar)$.
	\item The non-commutative algebra $A_\hbar$  (whose two generators satisfy $XY = e^{2\pi i\hbar} YX$).
	\item The non-commutative space\footnote{Actually, we should really be taking $\TTT_{\R,\hbar}^2$ to be the dg-category $D^b_{\mathrm{coh}}(A_\hbar-{\mathrm{mod}})$ (after an adequate interpretation of what a coherent sheaf should be); see, for example, \cite{kontsevich2007xi}.} $\TTT_{\R,\hbar}^2$, namely, the Morita equivalence class $[A_\hbar]_{\sim_M}$  of the algebra $A_\hbar$. We will call this \emph{the non-commutative torus} $\TTT_{\R,\hbar}^2$.
	\item The (non-separated, non-algebraic, smooth) stack $\TT_{\R,\hbar}^2$ obtained by stackification of $(\T_\R^2,E(\hbar))$. We will call this the \emph{quantum torus} $\TT_{\R,\hbar}^2$.
\end{enumerate}

Because of the remarkable properties of the category of stacks (for example, the existence of fibered products), in this paper, we will always use stacks\footnote{Another approach would have been to use \emph{diffeological spaces}: indeed, surprisingly, from the diffeological quotient of $\T^2$ by $E(\hbar)$, one can recover the stack $\TT_{\R,\hbar}^2$ (cf. \cite{iglesias2013diffeology, Becerra})} rather than non-commutative spaces: in principle all the non-commutative geometry can be recovered from the stacky geometry although, in practice, this may be not entirely trivial. We will return to this issue in a future work. In any case, it is much simpler to state that, for instance, $\TT_{\C,\hbar}^2$ is a K\"ahler stack, than to try to say the same for its non-commutative avatar $\TTT_{\C,\hbar}^2$.

Our notation for stacky quotients uses brackets so that, for example, we have: 
$$\TT_{\R,\hbar}^2 := [\T_\R^2/E(\hbar)] \cong [S^1 /\langle \rho_\hbar \rangle],$$
and, from now on, we will always use the presentation $\TT_{\R,\hbar}^2 := [S^1 /\langle \rho_\hbar \rangle]$ for the quantum torus. Actually, we will need to pass to the Lie algebra by taking logarithms. Indeed, we will find convenient to use the exponential group homomorphism (with kernel $\Z = \langle 1 \rangle$):
$ E: x \in \R \mapsto E(X):= \exp(2\pi i x) \in S^1,$
which, in turn, induces a map  $$E:[\R /\langle 1,\hbar \rangle] \to [S^1 /\langle \rho_\hbar=E(\hbar) \rangle] = \TT_{\R,\hbar}^2 $$ which is an isomorphism. We will write the additive subgroup $$\Gamma:=\langle 1,\hbar \rangle \subset \R.$$ Sometimes $\Gamma$ is called a \emph{quasi-lattice} but, given our motivation, we will call it a \emph{quantum lattice} or, simply, a \emph{q-lattice}. Clearly, $\Gamma$ behaves quite differently whether $\hbar$ is rational or not: in the former case, $\Gamma$ really is a lattice in $\R$, for it is always the case that $\Gamma \cong\Z$. In any case, $\Gamma$ plays the role of the `Lie algebra' of the rotation group $\langle \rho_\hbar \rangle$:
$$ E: \Gamma \to \langle \rho_\hbar \rangle.$$
With this, we arrive at the \emph{logarithmic representation of the quantum torus}:
$$\TT_{\R,\hbar}^2 \cong [\R /\Gamma].$$
There are two variations to the previous setting that we will need in our theory. First, we will work mostly with complex quantum tori rather than with real quantum tori (although Lagrangian tori will still be real):
$$\TT_{\C,\hbar}^2 := [\T_\C^2/E(\hbar)] \cong [(\C^*) /\langle \rho_\hbar \rangle] \cong [\C /\Gamma].$$
The second important variation arises from the fact that we will need to work with tori of arbitrary integer dimension $d+1$, so that, in general, we define:
$$\TT_{\C,d,\Gamma}:= \TT_{\C,\Gamma}^{d+1} := [\T_\C^{d}/E(\Gamma)]  \cong [\C^d /\Gamma].$$
where $\Gamma$ is a \emph{q-lattice} (namely, a finitely generated additive subgroup of some $\R^d$ spanning it over the real number field). We are to think of $\Gamma$ as the holonomy of a linear foliation on $\T_\C^{d+1}$ analogous to that of Figure \ref{Kronecker} (where $d=1$ and $\Gamma=\langle 1, \hbar \rangle$), of $\T_\C^{d}$ as a transversal to the foliation, and of $\C^d$ as the universal cover to such transversal. 

The simplest example of a quantum toric variety is probably a quantum projective line; just a projective line is an equivariant compactification of a one dimensional complex torus:
$$ \C P^1 = \C^* \cup \{0 \} \cup \{\infty \},$$
the analogous statement is true for a quantum projective line (which is then, in turn, a compactification of a quantum torus):
$$ \C \PP^1_\hbar = \TT_{\C,1,\hbar}  \cup \{0 \} \cup \{\infty \}.$$

This example is constructed in full detail in Examples \ref{P1QFans} and \ref{quantumP1} below. Enough is to say here that we construct $ \C \PP^1_\hbar$ with two charts, both of the form $[\C/\exp (2i\pi \hbar\Z)]$ (which is a partial compactification of $[\C^*/\exp (2i\pi \hbar\Z)]$),
glued by the attaching map:
$$ [z]\in [\C^*/\exp (2i\pi\Gamma)]\longmapsto [z^{-1}]\in [\C^*/\exp (2i\pi (-\Gamma))]. $$

Notice that a quantum projective line $ \C \PP^1_\hbar$ is a compactification of $\TT_{\C,\hbar}^2$. You may want to imaginatively think that $\C$ (resp. $\R$) is both the Lie algebra and the universal covering of $\TT_{\C,\hbar}^2$ (resp. $\TT_{\R,\hbar}^2$), and that $\pi_1(\TT_{\R,\hbar}^2)\cong\Gamma$, but this would be off by one dimension ($2\neq1$) for the case $\hbar=0$.  The dimensions may, at first, look confusing to the reader. To clarify this possible confusion let us mention that:
\begin{itemize}
	\item[i)] The `naive dimension' of $\TT_{\R,\hbar}^2$ seems to be $\dim \T^2 -\dim E(\hbar) = 2-1 = 1$. This is why we shift to the notation $\TT_{1,\R,\hbar}:= \TT_{\R,\hbar}^2$ in the body of the paper.
	\item[ii)]	 The `homotopy type' of $\TT_{\R,\hbar}^2$ is given by the homotopy quotient $\T^2\times_{E(\hbar)} E\R$ which in turn is homotopy equivalent to $\T^2$ (for $E(\hbar)\cong\R$ is contractible), and hence has `homotopic-dimension' two. The same holds for $\TT_{\C,\hbar}^2$. This will be reflected in the periodic cyclic homology of $\TTT_{\R,\hbar}^2$: from the homological point of view, it will look like a two-dimensional space.
	\item[iii)] As mentioned above, a quantum projective line $\C \PP^1_\hbar$ is a compactification of $\TT_{\C,\hbar}^2$, and, indeed, it will also have a `naive complex dimension' of 1 and a `homotopic dimension' of 2. Moreover, we will describe a complex manifold $N_\hbar$ (known as a LVM-manifold cf. Section \ref{lvmtheory} below) together with a foliation $\FF_\hbar$ (defined in Subsection \ref{F}) so that the groupoid $(N_\hbar,\FF_\hbar)$ compactifies the Kronecker groupoid $(\T^2_\C, E(\hbar))$ and the stack $ \C \PP^1_\hbar \cong [N_\hbar/\FF_\hbar]$ equivariantly compactifies the stack $\TT_{\C,\hbar}^2 := [\T_\C^2/E(\hbar)]$. The point here is that the complex dimension of $N_\hbar$ is two (in fact $N_\hbar\cong S^1\times S^3$ is a complex non-symplectic Hopf surface\footnote{Hopf manifolds and the more generally, Calabi-Eckmann manifolds, \cite{calabi1953class} are non-K\"{a}hler manifolds. Topologically they are of the form $S^{2n-1}\times S^{2m-1}$ and they are deformations of an elliptic, holomorphic fibration $E\to S^{2n-1}\times S^{2m-1}\to \Pn^{n-1}\times \Pn^{m-1}$. In this example we are interested in the Hopf case $n=1$, $m=2$. Of course, $\Pn^{n-1}\times \Pn^{m-1}$ is a toric variety. The generalization of a Calabi-Eckmann manifold corresponding to a general toric variety $X$ are the \emph{ LVM-manifolds} as it is proved in \cite{meersseman2004holomorphic}.}), and (in the irrational case) the leaves of the foliation $\F_\hbar$ are all isomorphic to $\C$ and hence, are contractible. In this situation,  the fact that $\C \PP^1_\hbar \cong [N_\hbar/\C]$ explains both the naive and  homotopic dimension countings for this quantum projective line. 
	\item[iv)] The case of rational $\hbar$ (say $\hbar=0$) requires more care, for here the leaves of the foliation $\F_\hbar$ wind up on themselves and rather than being copies of $\C$, they become elliptic curves of the form $S^1\times S^1$ (indeed, the map $N_\hbar \cong S^1\times S^3 \to   [N_0/\FF_0] \cong [N_0/S^1\times S^1] \cong \Pn^1 \cong S^2$ is the trivial constant map crossed with the Hopf fibration, with fiber $S^1\times S^1$), and then the homotopical dimension of $\Pn^1$ and the naive complex geometric dimension coincide and are both equal to 1 (of course). Notice here that the stack $\C \PP^1_0 \cong [N_0/\C] ( \ncong [N_\hbar/\FF_\hbar] \cong \Pn^1)$ still has homotopical dimension equal to two, and we have a fibration $$\C^* \simeq \BB \Z \to \C \PP^1_0 \to \Pn^1,$$
	that is to say \emph{$\C \PP^1_0 $ is a gerbe over $\Pn^1$ with abelian band $\Z$}, which explains the difference of dimensions by 1. We refer to this process as \emph{calibrating} $\Pn^1$ to obtain $\C \PP^1_0 $, the choice of calibration is by no means unique (cf. Definition \ref{decQFdef}, Subsection \ref{NDstandardConedec} and Subsection \ref{CCGcalibrated}. Here we are using a standard calibration as in Example \ref{P1QFans}). In any case, it is easy to recover $\Pn^1$ from $\C \PP^1_0 $ (by forgetting the gerbe) and vice-versa, $\C \PP^1_0 $ can be naturally constructed from is the complex version of the `2-fold homotopy cover' of $\Pn^1$ (the complex 2-fold homotopy cover of $S^2$ is $S^1\times S^3$). 
	\item[v)] When (in Section \ref{Qmoduli}), we want to form a moduli space of quantum lines (more generally, of quantum toric stacks), it is natural to add the calibration (to be thought of as a gerbe degree of freedom) to the rational case\footnote{Here we have a beautiful generalization of what Seiberg and Witten would refer to as `turning on the $B$-field' \cite{seiberg1999string}. In \cite{seiberg1999string}, the authors show that turning on the $B$-field (what we would call `considering a gerbe' over a space) produces an effective action that in turn, can be interpreted in terms of spacetime becoming non-commutative; to wit, string theory in the presence of a constant, non-zero $B$-field, brings about the appearance of non-commutative tori. Thus, in our case, turning on the gerbe, makes $\C \PP^1_0 $ into a slightly non-commutative version of $\Pn^1$. For a motivation in terms of mirror symmetry, see Section 4.8 in \cite{auroux2008mirror}}. The moduli space of calibrated quantum lines can be thought of as a desingularization of the naive moduli space which uses no calibrations (cf. Remark \ref{rkmodulitorus} below and \cite{bursztyn2002survey}). Also, if we want  periodic cyclic homology to form a nice bundle over the moduli space, we will need to use calibrations. 
\end{itemize}

Recall that all the information to reconstruct $ \C P^1$ can be combinatorially encoded by a fan in $\R^1$ with three cones: $\{0\}$, $\R_+$ and $\R_-$ together with the integral lattice $\Z \subset \R^1$. Likewise, all we need to reconstruct $ \C \PP^1_\hbar$ is the \emph{quantum fan} consisting of three cones $\{0\}$, $\R_+$ and $\R_-$ together with the q-lattice $\Gamma \subset \R^1$.

\begin{figure}[h]
	\centering
	\includegraphics[width=0.7\linewidth]{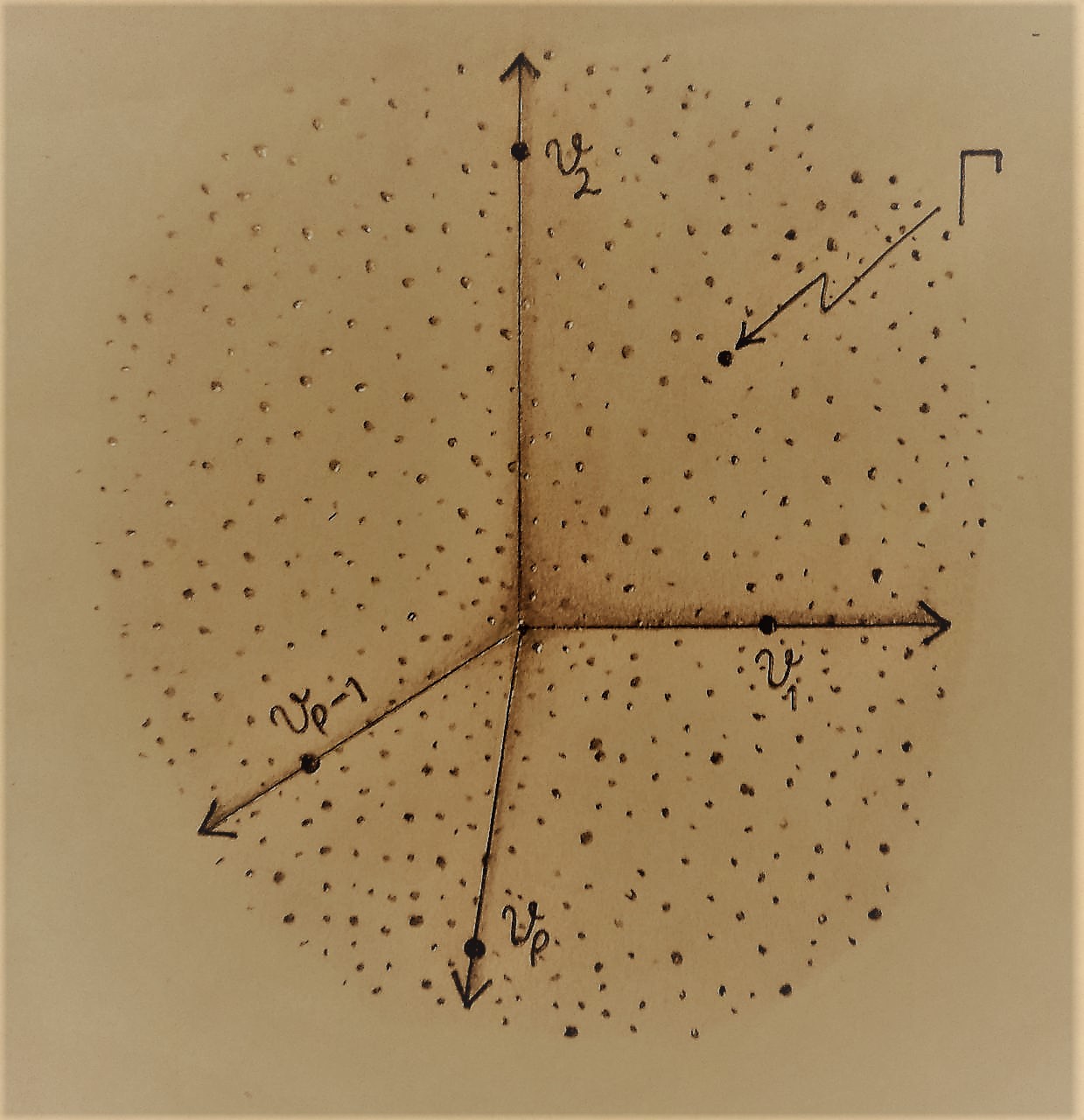}
	\caption[A Quantum Fan]{A quantum fan $(\Delta, v)$ in $\Gamma$ is very similar to a classical fan $\Delta$ in toric geometry, but instead of an integral lattice, it is equipped with a q-lattice $\Gamma$ (cf. Definition \ref{QFdef}). Notice that we must mark the (non-canonical) `primitive vectors' $(v_1,\ldots, v_p)$ (all in $\Gamma$) on every ray of the $1$-skeleton of the fan. The fan $\Delta$ no longer needs to be rational.}\label{AQuantumFan}
\end{figure}

A general quantum toric stack can be constructed starting from a general (not necessarily rational) \emph{quantum fan} (see Figure \ref{AQuantumFan} and Definition \ref{QFdef}): such a q-fan carried a q-lattice $\Gamma \subset \R^d$, and therefore defines a q-torus $\TT_{\C,d,\Gamma} \cong [\C^d /\Gamma]$. The quantum toric stack $\XX_{\Delta,\Gamma,v}$ (cf. Definition \ref{QToricObjects}) is an equivariant compactification of $\TT_{\C,d,\Gamma}$ given by the data of the quantum fan $(\Delta,v)$.

As explained before, we really want to consider the calibrated case (adding gerbe degrees of freedom). At the level of fans, this is achieved by the definition of a \emph{calibrated quantum fan}. The precise description of a calibrated quantum fan (depicted in Figure \ref{ACalibratedQuantumFan}) is found in Definition \ref{decQFdef} below. For now, think of a calibration as a homomorphism $h: \Z^n \to \Gamma$. Given such a calibrated quantum fan, then we define a \emph{Calibrated Quantum Toric Stack} in Definition \ref{decQTVdef} and denoted\footnote{Here $J$ is a certain crucial piece of combinatorial data the we are omitting altogether in the introduction but see Definition \ref{decQFdef} for details.} by $\nctoricgerbe{\Delta}{h}{J}$.

\begin{figure}[h]
	\centering
	\includegraphics[width=0.9\linewidth]{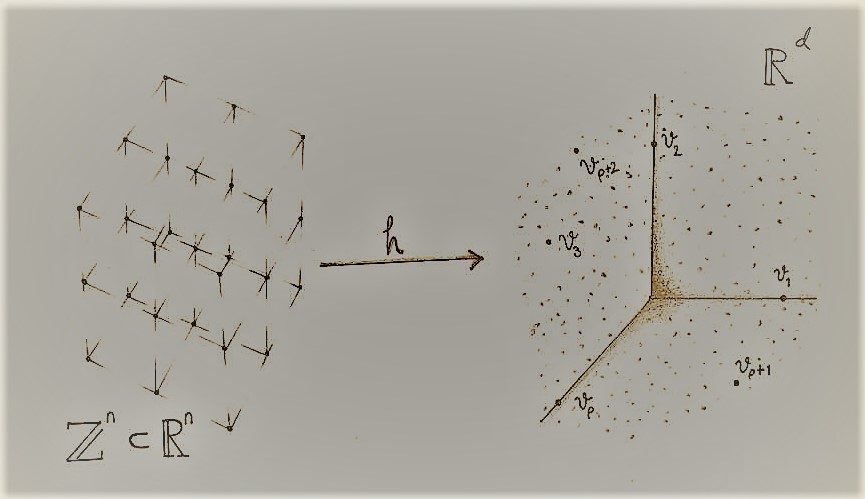}
	\caption[A Calibrated Quantum Fan]{A calibrated quantum fan $(\Delta,h)$ is essentially a quantum fan plus a calibration, namely, a homomorphism $h: \Z^n \to \Gamma$ `determining the various Planck lengths of the quantum system' (cf. Definition \ref{decQFdef}). }\label{ACalibratedQuantumFan}
\end{figure}

The relation between the calibrated Quantum Toric stack $\nctoricgerbe{\Delta}{h}{J}$ and its un-calibrated version  $\nctoric{\Delta}{\Gamma}{v}$ is explained in Proposition \ref{propgerbe}; $\nctoricgerbe{\Delta}{h}{J}$ is a gerbe over $\nctoric{\Delta}{\Gamma}{v}$ with band $\Z^a$ (where $a:=n- \mathrm{rank}_\Z(\Gamma)$): $\nctoricgerbe{\Delta}{h}{J}$ is completely determined by a classifying map $\nctoric{\Delta}{\Gamma}{v} \to \BB \BB \Z$.

One of the main results of this paper is Theorem \ref{prisodecQTV}:\emph{ the category of simplicial calibrated Quantum Toric Stacks $\mathscr Q^{cal}$ is equivalent to the category of quantum toric fans $Q^{cal}$}.

The Quantum Geometric Invariant Theory (QGIT) corresponding to quantum toric geometry is specially interesting: we prove that one can indeed represent $\nctoricgerbe{\Delta}{h}{J}$ as a global quotient of the form (cf. Theorem \ref{TheoremGITdec}):
$$\nctoricgerbe{\Delta}{h}{J}\cong[\mathscr{S}/\mathcal{A}],$$
where $\mathscr{S}$ is the complement in $\C^n$ of a union of coordinate vector subspaces and the classical torus $\Torus^n$ acts multiplicatively on it with a Zariski dense orbit (we denote by $\mathcal{A}$ this action). Moreover, $\mathcal{A}$ actually defines a foliation on $\mathscr{S}$ so that the stackification of the holonomy groupoid of said foliation is isomorphic to the un-calibrated toric  stack $\nctoric{\Delta}{\Gamma}{v}$ (Theorem \ref{theoremholgroupoid}).

From the point of view of QGIT, there is a deep and beautiful relation beetween quantum toric stacks and \emph{LVMB theory} (see section \ref{lvmtheory} for definitions). Such relation occurs only when $n-d$ is even (cf. Definition \ref{defcentered}). The reader may want to think for now of a LVMB-manifold (together with a canonical foliation induced by a holomorphic $C^m$-action) $(N,\FF)$ as a generalization of the Calabi-Eckmann manifolds (and their elliptic foliation induced by a $\C$-action, where $\C$ covers the elliptic curve) so that (cf. Theorem \ref{mainGIT}): 
$$\nctoric{\Delta}{\Gamma}{v}\cong [N/\FF]$$
and
$$\nctoricgerbe{\Delta}{h}{J}\cong[N/\C^m].$$
Furthermore, we will show (Theorem \ref{thmLVMBdec}) that \emph{the category $\mathscr{V}^*_{LVMB}$ of LVMB-manifolds is equivalent to the full subcategory $\mathscr{Q}^{cal}_{even}$ of $\mathscr{Q}^{cal}$ whose objects are calibrated Quantum Toric Stacks associated to an even calibrated Quantum Fan}. To deal with the odd case, we could consider Ishida's theory \cite{IshidaMax}. There, Ishida considers compact complex manifolds $M$ with maximal real-torus action. Complete toric varieties and LVMB manifolds are examples of such manifolds. Just like LVMB-manifolds, these manifolds are also endowed with a canonical holomorphic foliation and the leaf stacks of such foliations are always Quantum toric stacks, landing sometimes in the case when $n-d$ is odd (cf. Remark \ref{Ishida}). In any case, this method provides a way to interpret QGIT for the odd case.  By using Ishida's results \cite{Ishida}, we can also prove that \emph{\nctoric{\Delta}{\Gamma}{v} is K\"{a}hler iff $\Delta$ is polytopal,} for an appropriate definition of Kähler in this context  (Theorem \ref{thkahler}). In the polytopal case that has been just considered, the image of the quantum moment map is described in Subsection \ref{asspol}, and the inverse image of a point in the polytope is a real lagrangian quantum torus.

\begin{figure}[h]
	\centering
	\includegraphics[width=0.7\linewidth]{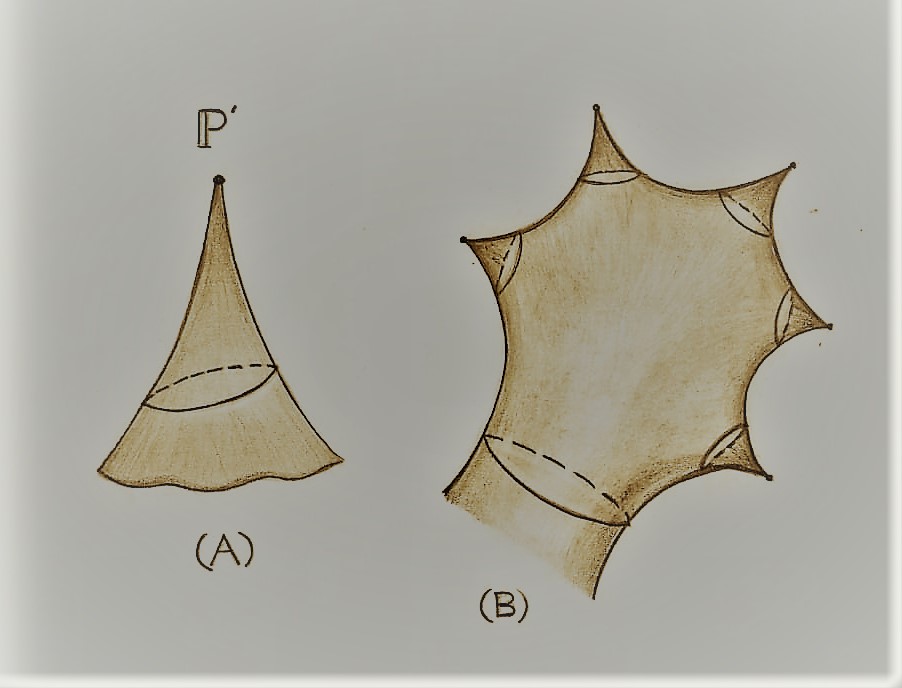}
	\caption[Moduli spaces of quantum toric stacks]{Moduli spaces of quantum toric stacks: in (A) we depict the moduli space of quantum projective lines: it has only one orbifold point with stabilizer $\Z_2$ corresponding to the classical $\Pn^1$ (or equivalently, to its non-commutative avatar $\C \PP^1_0$); for the case of $\Pn^d$, see Subsection \ref{QmoduliPn}. In (B) we depict a more fanciful representation of the moduli space $\modulitoricgerbe{D}{n,d}$ of calibrated Quantum Toric Stacks  with fixed combinatorial type $D$: the orbifold points occur whenever the fan suddenly has more symmetries (cf. Section \ref{Qmoduli}). The classical toric varieties will land on the rational locus of some of these moduli spaces.}\label{ModuliQuantumToricStack}
\end{figure}

Unlike classical toric varieties which are rigid (as equivariant toric spaces), quantum toric stacks admit moduli. In Section \ref{Qmoduli}, the final section of this paper, we study various moduli spaces, specially the moduli space $\modulitoric{D}{d}$ of quantum toric stacks with fixed combinatorial type $D$ (and $\Gamma$-complete), the moduli space $\modulitoricgerbe{D}{n,d}$ of calibrated Quantum Toric Stacks (of maximal length) and fixed combinatorial type $D$ and the moduli space $\mathcal M^{\mathcal S}_{m,n}$ of $G$-biholomorphism classes of LVMB manifolds (see Figure \ref{ModuliQuantumToricStack}). The main theorem of this final section is that that all those moduli spaces  are real finite-dimensional orbifolds (cf. Corollary \ref{cormoduliGC}, Corollary \ref{cormodulidec} and Proposition \ref{propiso}). Under certain numerical condition (that of Theorem \ref{thmorbifold}), more is true: $\mathcal M^{\mathcal S}_{m,n}$ is a complex orbifold, and there is a `twistor bundle complexification mapping':
$$\R^{(n-d)^2/2} \to  {\mathcal M^{\mathcal S}_{m,n}} \to \modulitoricgerbe{D}{n,d},$$
namely, ${\mathcal M^{\mathcal S}_{m,n}}$ is a (sometimes complex) orbibundle of even rank $(n-d)^2/2$ over $\modulitoricgerbe{D}{n,d}$. This immediately implies the homotopy equivalence (a diffeomorphism when $n=d$):
$$ {\mathcal M^{\mathcal S}_{m,n}} \simeq \modulitoricgerbe{D}{n,d}, $$
namely, the moduli space $\modulitoricgerbe{D}{n,d}$ can be promoted (in the same homotopy class) to a complex orbifold ${\mathcal M^{\mathcal S}_{m,n}}$. Let us point out here that there is an interesting analogy with the classical case of the moduli space of curves $M(g,n)$ where $g$ is the genus (combinatorial information) and $n$ is a marking that makes the compactification of the moduli space nicer. Here too, we have combinatorial information $(D,d)$, and marking information $n$, making the moduli space nicer, and even sometimes giving it a complex structure. We will explore this analogy elsewhere.

Let us finish this introduction by adding some very schematic and cursory remarks regarding future work on the relation of this theory to the field of mirror symmetry. It is natural to ask for a generalization of Homological Mirror Symmetry for toric varieties. Classical HMS for ordinary toric varieties has been proved by Abouzaid \cite{abouzaid2009morse} using methods from tropical geometry. It can be considered as a SYZ (Strominger-Yau-Zaslow)  correspondence based on family of Lagrangian tori over a polytope. 
The methods of this paper can be used to obtain a quantization of the corresponding Lagrangian tori and of the fan associated to the polytope.
It is reasonable to expect that HMS generalizes to this situation. This could be understood as a quantum version of the SYZ-correspondance. All the non-commutative projective spaces of the work of Auroux, Katzarkov and Orlov\cite{auroux2008mirror}  on mirror symmetry projective spaces,  appear as quantum toric stacks, and that we expect many of the results of \cite{auroux2008mirror} to generalize to our setting.  There is a specially simple and interesting example of an LVM manifold: the classical Hopf surface  $S^3 \times S^1$. As we have discussed earlier, it has a canonical elliptic foliation that, when deformed, covers all quantum toric projective lines. Using $T$-duality for quantum tori, we can understand `mirror symmetry' for the Hopf surface giving an alternative approach to the one developed by Abigail Ward in her Harvard PhD thesis. From this point of view it is also natural to characterize the moduli space of stability conditions for non-commutative toric varieties considered as triangulated categories (cf. \cite{soibelman2001quantum, baranovsky2003representations, latremoliere2005approximation,latremoliere2017noncommutative,kapovich2017note}): we conjecture that they are closely related to the solenoidal spaces as studied by Sullivan-Verjovsky \cite{SullivanVerjovsky}.

The previous program seems to beg a solution to the question as to the correct formulation of an algebraic geometry (the so called $B$-side) that allows one to define the adequate categories associated to a quantum toric stack. Two puzzles must be met: first, the stacks constructed in this paper are far from being algebraic, and second, even in the K\"{a}hler case, they are constructed starting from non-symplectic manifolds such as the Hopf surface. All the difficulties can be resolved simultaneously by the introduction of a version of algebraic geometry that occurs in a fixed non-standard extension of the complex numbers. We have investigated and introduced such \emph{`chimeric algebraic geometry'} in a series of works \cite{ChLMO, ChKLMV} (cf. \cite{gendron2006algebraic}, \cite{gendron2008geometric},\cite{kossak2004nonstandard}). Then, the theory offered in \cite{ChKLMV} solves the issue of constructing the desired $B$-side for the mirror symmetry program, on the other hand, the symplectic geometry needed for the $A$-side remains an outstading question requiring further analysis. We will return to this issue elsewhere.

Finally, the theory we present here is, as we said before, a wide ranging generalization of toric geometry. There were earlier approaches generalizaing toric geometry in a number of ways, and they are in various ways related to our theory. Let us simply refer the reader to some of the most relevant works in this direction: Battaglia-Prato \cite{BP2}, Battaglia-Zaffran \cite{BZ}, Bressler-Lunts \cite{bressler2005hard}, Firat Pir \cite{firat2018irrational}, Ford \cite{ford1999toroidal}, Postinghel-Sottile \cite{postinghel2015degenerations}, and Ratiu-Zung \cite{ratiu2017presymplectic}.

\ 
\

\textbf{Acknowledgments.} L.K. was supported by the 
Simons collaborative Grant - HMS,
NRU HSE, RF government grant, ag. 14.641.31.000,
Simons Principle Investigator Grant,
CKGA VIHREN grant K$\Pi$-06-$\Pi$B/16.
Much of the research was
conducted while the authors enjoyed the hospitality of IMSA Miami
and Laboratory of Mirror Symmetry HSE Moscow.

E.L. would like to thank FORDECYT (CONACYT), IMATE-UNAM, NRU HSE, RF government grant, ag. 14.641.31.000, the Institute for Mathematical Sciences of the Americas, the Simon's Foundation and the Moshinsky Foundation, the University of Geneva, the QUANTUM project from the University of Angers and the Laboratory of Mirror Symmetry HSE Moscow. 

L.M. would like to thank the kind support of the QUANTUM project from the University of Angers, UMI CNRS 2001 LaSol and the Institute for Mathematical Sciences of the Americas (Simons Foundation and University of Miami).

A.V. would like to thank the support of the University of Angers,
QUANTUM project UMI CNRS 2001 LaSo,  IMSA in Miami and
DGAPA  PAPIIT project  IN108120  UNAM, Mexico.

\section{Conventions on Stacks}
\label{stacks}
In this section, we fix some notations and conventions on the stacks that will be used in the paper. These include all the Quantum Toric Stacks and all calibrated Quantum Toric Stacks; however, the orbifolds of Section \ref{Qmoduli} are not instances of what follows; they are briefly defined in Remark \ref{rkorbifold}.

\begin{remark}
	This section can be skipped by readers that are not familiar with the theory of stacks. Such a reader can think of a stack as a `manifold whose charts are not one-to-one' very similar to orbifolds. All our stacks will have local charts of the form $V \to V/\mathbf{G}$ where the local group $\mathbf{G}$ is discrete and abelian (not quite finite, as in the case of orbifolds). Nevertheless a warning is in order: for most quantum toric varieties, the quotient of a local chart $V/\mathbf{G}$ will \emph{not be Hausdorff} much like in the theory of \emph{quasifolds} \cite{prato2017quasifolds}; still, we use the theory of stacks for we need the whole machinery of morphisms of stacks, fibered products, etc. in our development. Surprisingly, we could have used all the required machinery using the theory of \emph{diffeological spaces} \cite{iglesias2013diffeology} to model our quantum toric manifolds as it is shown on \cite{Becerra} (but the morphisms would still be a bit off). In any case, we prefer the more standard use of stacks presented here.  
\end{remark}

We take as base category the category $\mathfrak A$ of affine toric varieties and toric morphisms. We take for covering of an affine toric variety $T$ a decomposition $T=T_1\cup\cdots \cup T_n$ into toric Zariski open subsets of $T$. With these coverings, $\mathfrak A$ is a site.

We will also need the category $\mathfrak{G}$ of complex analytic spaces $X$ endowed with a holomorphic action of a complex abelian Lie group $G$ with a Zariski open orbit isomorphic to $G$. Morphisms are equivariant holomorphic mappings that restrict to Lie group morphisms. Observe that $\mathfrak{A}$ is a subcategory of $\mathfrak{G}$. 

By a {\itshape cover} of $T$ of $\mathfrak{A}$, we mean an object $\tilde T$ of $\mathfrak{G}$ endowed with a free and proper holomorphic action of a discrete abelian group $H$ whose quotient is $T$. Hence $p : \tilde T\to T$ is an unramified analytic cover of $T$. Moreover, $p$ is equivariant, that is, belongs to $\mathfrak{G}$ since the abelian $H$-action commutes with the abelian $G$-action.

If $H$ is a non-discrete complex abelian Lie group acting freely and properly on $\tilde T$ with quotient $T$, then $p : \tilde T\to T$ is an equivariant principal $H$-bundle and we rather speak of a {\it principal bundle} over $T$.

Quantum Toric Stacks are of the form $[X/H]$ with $X\in\mathfrak{G}$ and $H$ a discrete abelian group or a complex abelian Lie group, or are descent data of such $[X/H]$, cf. Section \ref{Qatlas}. To be more precise, we consider the category $[X/H]$ whose objects are covers $\tilde T$ over a space $T\in\mathfrak A$ with an equivariant holomorphic map $m$
\begin{equation}
\label{stackobjects}
\begin{tikzcd}
\tilde T \arrow[r, "m"]\arrow[d]&X\\
T
\end{tikzcd}
\end{equation} 
Here $m$ is assumed to be equivariant with respect to both the $H$-action and the $G$-action.

The morphisms are
\begin{equation}
\label{stackmorphisms}
\begin{tikzcd}
&            &X\\
\tilde T\arrow[r]\arrow[rru,bend left=20,"m"]\arrow[d]          &\tilde  S \arrow[d]\arrow[ru,bend right=10,"n"']      &                       \\
T        \arrow[r]         &S       &          
\end{tikzcd}
\end{equation} 

The following proposition is enough for our purposes.

\begin{proposition}
	\label{propstacks}
	The functor $[X/H]\to\mathfrak A$ sending $(\tilde T,T,m)$ onto $T$; and $(\tilde T\to \tilde S,T\to S,m,n)$ onto $T\to S$ is a stack over $\mathfrak{A}$.
	\end{proposition}

\begin{proof}
	We must check that pull-backs exist and are unique up to unique isomorphism, that isomorphisms form a sheaf and that descent data are effective.
	
	Given a cartesian diagram
	\begin{equation}
	\label{stackpullback}
	\begin{tikzcd}
	T_f\times_p\tilde S \arrow[r]\arrow[d]&\tilde S\arrow[d,"p"]\\
	T\arrow[r,"f"]&S
	\end{tikzcd}
	\end{equation} 
	its restriction to the Zariski open orbits is
	\begin{equation}
	\label{stackpullbackgroups}
	\begin{tikzcd}
	(\C^*)^a_f\times_p G \arrow[r]\arrow[d]&G\arrow[d,"p"]\\
	(\C^*)^a\arrow[r,"f"]&(\C^*)^b
	\end{tikzcd}
	\end{equation} 
	with $a$, respectively $b$, the dimension of $T$, resp. S. Now, $(\C^*)^a_f\times_p G$ is a complex Lie abelian group, with the following law
	\begin{equation}
	\label{fplaw}
	(w,g)\cdot (w',g')=(w\cdot w',g\cdot g')
	\end{equation}
	which is well defined since $f(w\cdot w')=f(w)\cdot f(w')=p(g)\cdot p(g')=p(g\cdot g')$.
	
	But \eqref{fplaw} extends as an equivariant holomorphic action of $(\C^*)^a_f\times_p G$ on $T_f\times_p\tilde S$. It has obviously a Zariski open orbit isomorphic to $(\C^*)^a_f\times_p G$ and all the arrows of \eqref{stackpullback} are equivariant. Thus pull-back exist and it is straightforward to check that they are unique up to unique isomorphisms of $\mathfrak{G}$.
	
	The last two points (isomorphisms form a sheaf and descent data are effective) are routine checking and we leave their verification to the reader. 
\end{proof}
The choice of $\mathfrak{A}$ as base category reflects the fact that we consider Quantum Toric Stacks as a generalization of classical toric varieties. The choice of equivariant covers reflects the fact that we only deal with equivariant properties of these geometric objects. Of course, other choices - as taking $\mathfrak{G}$ or taking the category of analytic spaces as base category, using equivariant covers or not - may have their interest and should be developed to look for different flavours of the construction.

We finish this section with a warning.

\begin{warning}
	\label{warningstacks}
	Quantum Toric Stacks are not {\itshape algebraic stacks}. Many of them have stabilizers equal to some power of $\Z$. And even when it is not the case, the intensive and fundamental role played by the exponential map $E$ defined in \eqref{E} prevents them from being algebraic.
	
	However, since Quantum Toric Stacks have an atlas which is an analytic space, cf. Corollaries \ref{coratlas} and \ref{coratlasdec}, they are analytic in the sense of \cite[\S 2.4]{Mteich}.
\end{warning}
\section{Quantum Tori}
\label{QTori}
In this section, we define the Quantum Tori which sit inside Quantum Toric Varieties. They are complex analogues of $\TT_{\R,\hbar}^i$ but the precise relationship is not investigated here but in Section \ref{Qmoduli}, see Theorem \ref{thmmoduliQt}.

\subsection{Quantum Torus associated to an additive subgroup of $\R^d$}
\label{QTorus}
In classical toric geometry, there is a single complex torus $\Torus^d:= \Torus^d_\C$ once the dimension $d$ is fixed. In a more intrinsic presentation, for any lattice $\Gamma$ that spans $\R^d$ over the reals, is associated a torus $\text{Hom }(\check{\Gamma},\C^*)$, but all are isomorphic to $\Torus^d$.
	
In Quantum Toric Geometry, things are different since $\Gamma$ is no more a lattice but any finitely generated additive subgroup of some $\R^d$ spanning it over the reals (a \emph{q-lattice}). Then we define,

\begin{definition}
	\label{defQTorus}
	The {\it quantum torus} associated to $\Gamma$ is the quotient stack $[\C^d/\Gamma]$. We denote it by $\nctorus{d}{\Gamma}$.
\end{definition}

\begin{example}
	\label{examplesqrt2prel}
	Let $d=1$ and define $\Gamma$ as the subgroup of $(\R,+)$ generated by $1$ and $\sqrt 2$. Since $\sqrt 2$ is irrational, it is dense in $\R$. It acts freely on $\R$ by translation but this action is not proper and the topological quotient is not Hausdorff. Indeed, for any pair of real numbers $x,y$, since $y-x$ is an accumulation point of $\Gamma$, the images of $x$ and $y$ are not separated.\\
	It follows that the associated quantum torus $[\C/\Gamma]$ is not a manifold and stack language is needed to handle it as a complex "space".
\end{example}
	Hence $\nctorus{d}{\Gamma}=[\C^d/\Gamma]$ has to be understood more formally as a category fibered in groupoids  $\nctorus{d}{\Gamma}\to\mathfrak A$, as in Section \ref{stacks}.
	
	Objects of $\nctorus{d}{\Gamma}$ are $\Gamma$-covers $T$ over a space $A\in\mathfrak A$ with an equivariant holomorphic map $m$ in $\C^d$
	\begin{equation}
	\label{nctorusobjects}
	\begin{tikzcd}
	T \arrow[r, "m"]\arrow[d]&\C^d\\
	A
	\end{tikzcd}
	\end{equation} 
	and morphisms
	\begin{equation}
	\label{nctorusmorphisms}
\begin{tikzcd}
&            &\C^d\\
T\arrow[r]\arrow[rru,bend left=20,"m"]\arrow[d]          &S  \arrow[d]\arrow[ru,bend right=10,"n"']      &                       \\
A        \arrow[r]         &B       &          
\end{tikzcd}
	\end{equation} 
	The functor $\nctorus{d}{\Gamma}\to\mathfrak A$ sends $(T,A,m)$  onto $A$; and $(T\to S,A\to B,m,n)$ onto $A\to B$. \vspace{5pt}\\
If $L$ is a linear map from $\R^d$ to $\R^{d'}$ that sends $\Gamma$ onto $\Gamma'$, then its extension over the complex numbers obviously descends as a stack morphism from $\nctorus{d}{\Gamma}$ onto $\nctorus{d'}{\Gamma'}$.
\begin{equation}
\label{CDdeftorusmorphism}
\begin{tikzcd}[row sep=small]
\R^d\arrow[d,hook]\arrow[r, "L"]&\R^{d'}\arrow[d,hook]\\
\C^d\arrow[d]\arrow[r,"L"]&\C^{d'}\arrow[d]\\
\nctorus{d}{\Gamma}\arrow[r,"\mathscr L"]&\nctorus{d'}{\Gamma'}
\end{tikzcd}
\end{equation}
 We take this as definition of toric morphism. Hence,
\begin{definition}
	\label{deftorusmorphism}
	A {\it torus morphism} $\mathscr L$ from $\nctorus{d}{\Gamma}$ onto $\nctorus{d'}{\Gamma'}$ is a stack morphism such that there exists a linear map  $L$ from $\R^d$ to $\R^{d'}$ sending $\Gamma$ to $\Gamma'$ and satisfying
	{\rm \eqref{CDdeftorusmorphism}}.
\end{definition}
To wit, given $T\to A$ a $\Gamma$-cover, set
\begin{equation}
\label{TLG}
T\times_{L}\Gamma':=\{(t,\gamma')\in T\times \Gamma'\}/\Gamma
\end{equation}
for the following action of $\Gamma$
\begin{equation}
\label{TLGaction}
\gamma\cdot (t,\gamma')=(\gamma\cdot t,\gamma'-L\gamma)
\end{equation}
Let also
\begin{equation}
\label{TLGm}
[(t,\gamma')]\longmapsto m'[(t,\gamma')]:=Lm(t)+\gamma'
\end{equation}
which is well-defined since
\begin{equation*}
\begin{aligned}
Lm(\gamma\cdot t,\gamma'-L\gamma)&=Lm(\gamma\cdot t)+\gamma'-L\gamma\\
&=Lm(t)+L\gamma+\gamma'-L\gamma\\
&=Lm(t)+\gamma'
\end{aligned}
\end{equation*}
Given $f:T\to S$ a morphism of $\Gamma$-cover, over $A\to B$, we define
\begin{equation}
[(t,\gamma')]\in T\times_L\Gamma'\longmapsto f_L([(t,\gamma')]:=[(f(t),\gamma')])\in S\times_L\Gamma'
\end{equation}
and, if $m$, resp. $n$, is the equivariant map associated to $T\to A$, resp. $S\to B$, so that $m=n\circ f$, we have
\begin{equation*}
\begin{aligned}
n'\circ f_L([(t,\gamma')])&=n'([f(t),\gamma'])\\
&=Ln(f(t))+\gamma'\\
&=Lm(t)+\gamma'\\
&=m'([t,\gamma'])
\end{aligned}
\end{equation*}
So we finally have
\begin{equation}
\label{torusmorphismL1}
\mathscr L(T,A,m)=(T\times_L\Gamma',A,m')
\end{equation}
and
\begin{equation}
\label{torusmorphismL2}
\begin{aligned}
\mathscr L&(f : T\to S,A\to B,m,n)=(f_L :T\times_L\Gamma'\to S\times_L\Gamma',\\
& A\to B, m', n')
\end{aligned}
\end{equation}
Since $\text{Vect }\Gamma$ is the whole $\C^d$, there exists a basis of $\C^d$ formed by vectors $(v_1,\hdots,v_d)$ of $\Gamma$. We may assume that  is such a basis. There exists some linear isomorphism $L$ of $\C^d$ sending $(v_1,\hdots,v_d)$ onto the canonical basis $(e_1,\hdots,e_d)$. It defines a torus isomorphism between $\nctorus{d}{\Gamma}$ and $\nctorus{d}{L\Gamma}$. \vspace{5pt}\\
So, defining
\begin{definition}
	\label{defstandardGamma}
	Let $(e_1,\hdots,e_d)$ be the canonical basis of $\R^d$.
	We say that $\Gamma$ is {\it standard} if it contains the standard lattice $\bigoplus_{i=1}^d\Z e_i$.\\
	We say that $\nctorus{d}{\Gamma}$ is {\it standard} if $\Gamma$ is standard.
\end{definition}
\noindent We just proved
\begin{lemma}
	\label{lemmastandard}
	Any quantum torus is isomorphic to a standard quantum torus.
\end{lemma}

If $\Gamma$ is standard, we may decompose it as  
\begin{equation}
\label{Gamma}
\Gamma=\Z^d+\Gamma_0
\end{equation}

Define
\begin{equation}
\label{E}
x\in\R^{d}\subset \C^d\longmapsto E(x):=(e^{2i\pi x_1},\hdots, e^{2i\pi x_{d}})\in (\mathbb S^1)^{d} \subset\Torus^d
\end{equation}
Then we may put a standard $\nctorus{d}{\Gamma}$ in multiplicative form as the quotient stack $[\Torus^d/E(\Gamma)]$ where $E(\Gamma)=E(\Gamma_0)$ acts multiplicatively on $\Torus^d$. We have a commutative diagram
\begin{equation}
\label{CDaddvsmulttorus}
\begin{tikzcd}[column sep=large]
\C^d \arrow[d]\arrow[r,"E"] &\Torus^d\arrow[d]\\
\nctorus{d}{\Gamma}\arrow[r,"\simeq"]&\phantom{d}\kern-5pt[\Torus^d/E(\Gamma)]
\end{tikzcd}
\end{equation} 
In a more functorial point of view, $E$ defines a stack isomorphism $\mathscr E$ from $\nctorus{d}{\Gamma}$ to $[\Torus^d/E(\Gamma_0)]$ such that
\begin{equation*}
\mathscr{E}(T,A,m)=(T/\Z^d,A,Em)
\end{equation*}
with $\Z^d$ acts on $T$ as a subgroup of $\Gamma$ and with
\begin{equation*}
\begin{tikzcd}
T\arrow[d]\arrow[dd,bend right=50]\arrow[r,"m"]&\C^d\arrow[d,"E"]\\
T/\Z^d\arrow[d]\arrow[r,"Em"]&\Torus^d\\
A &
\end{tikzcd}
\end{equation*}
and
\begin{equation*}
\mathscr{E}(f:T\to S,A\to B,m,n)=(f:T/\Z^d\to S/\Z^d,A\to B,Em,En)
\end{equation*}
where $f$ descends as a morphism from $T/\Z^d$ to $S/\Z^d$ since it satisfies by definition $f(p\cdot t)=p\cdot f(t)$ for all $p\in\Gamma$.
\begin{example}
	\label{examplesqrt2}
	Let $d=1$ and $\Gamma$ be generated by $1$ and $\sqrt 2$ as in Example \ref{examplesqrt2prel}. Consider the linear map
	\begin{equation*}
	z\in\C\longmapsto \sqrt{2}\cdot z\in\C
	\end{equation*}
	It sends $1$ onto $\sqrt{2}$ and $\sqrt{2}$ onto $2$ hence it preserves $\Gamma$. So it descends as the torus morphism
	\begin{equation*}
	[z]\in[\C/\Gamma]\longmapsto [\sqrt{2}z]\in[\C/\Gamma]
	\end{equation*}
	or, in multiplicative form,
	\begin{equation}
	\label{multsqrt2}
	[w]\in[\Torus/E(\Gamma)]\longmapsto [w^{\sqrt{2}}]\in[\Torus/E(\Gamma)]
	\end{equation}
	This must be understood as follows. Given $w\in\Torus$, choose $z\in\C$ with $E(z)=w$. Then compute $E(z\sqrt{2})$. Of course, $z$ is unique only up to addition of an integer, and for $p\in\Z^*$, the complex number $E((z+p)\sqrt{2})$ is different from $E(z\sqrt{2})$. Hence there is no well defined mapping $w\to w^{\sqrt{2}}$. However, $E(p\sqrt{2})$ belongs to $E(\Gamma)$ so \eqref{multsqrt2} is well defined.   
\end{example}

\begin{remark}
	\label{rkaddvsmultmorphism}
	It is important to notice that a torus morphism $\mathscr L$ does not always lift as a mapping from $\Torus^d$ to $\Torus^{d'}$, as shown in Example \ref{examplesqrt2}. This will be a source of problem when defining toric morphisms in Section \ref{Qatlas}.
\end{remark}
\subsection{Calibrated Quantum Tori}
\label{decQTorus}
We define now calibrated quantum tori. To do this, we need to fix, in addition to $\Gamma$, a set of generators of $\Gamma$. More precisely, we set

\begin{definition}
	\label{defcalibration}
	A {\it calibration} of $\Gamma$ is given by
	\begin{enumerate}[i)]
		\item An epimorphism $h : \Z^n\to \Gamma$
		\item A subset $J\subset\{1,\hdots,n\}$ such that 
		\begin{equation}
			\label{deccondition}
			\text{Vect}\{h(e_i)\mid i\not \in J\}=\C^d
		\end{equation}
	\end{enumerate}
\end{definition}
The subset $J$ is called the set of {\it virtual generators}. It may be empty as in the following important example.

\begin{example}\textbf{The trivial calibration of the standard torus.}
	\label{trivialdec}
	Let $\Gamma$ be $\Z^n$, so that $\nctorus{n}{\Gamma}$ is just $\Torus^n$. The trivial calibration of $\Torus^n$ is $h=Id$. Note that it forces the set of virtual generators to be empty.
\end{example}

Now we set

\begin{definition}
	\label{defQTorusdec}
	The {\it calibrated quantum torus} associated to $h\ :\ \Z^n\to\Gamma$ is the quotient stack $[\C^d/\Z^n]$ where $\Z^n$ acts through $(z,p)\mapsto z+h(p)$. We denote it by $\nctorusgerbe{h}{J}$.
\end{definition}

Objects of $\nctorusgerbe{h}{J}$ are $\Z^n$-covers $T$ over a space $A\in\mathfrak A$ with an equivariant map $m$ in $\C^d$
\begin{equation}
\label{nctorusobjectsdec}
\begin{tikzcd}
T \arrow[r, "m^{cal}"]\arrow[d]&\C^d\\
A
\end{tikzcd}
\end{equation} 
and morphisms
\begin{equation}
\label{nctorusmorphismsdec}
\begin{tikzcd}
&            &\C^d\\
T\arrow[r]\arrow[rru,bend left=20,"m^{cal}"]\arrow[d]          &S \arrow[d]\arrow[ru,bend right=10,"n^{cal}"']      &                       \\
A        \arrow[r]         &B       &          
\end{tikzcd}
\end{equation} 
The functor $\nctorusgerbe{h}{J}\to\mathfrak A$ sends $(T,A,m^{cal})$  onto $A$; and $(T\to S,A\to B,m^{cal},n^{cal})$ onto $A\to B$. \vspace{5pt}\\
If $L$ is a linear map from $\R^d$ to $\R^{d'}$ that sends $\Gamma$ onto $\Gamma'$ and induces a toric morphism, then its extension over the complex numbers does not define a stack morphism from $\nctorusgerbe{h}{J}$ onto $\nctorusgerbe{h'}{J'}$. We need an extra morphism $H$ from $\Z^n$ to $\Z^{n'}$ such that the following diagram is commutative.
\begin{equation}
\label{CDLH}
\begin{tikzcd}
\Z^n \arrow[r,"H"]\arrow[d,"h"']&\Z^{n'}\arrow[d,"{h'}"]\\
\Gamma\arrow[r,"L"']&\Gamma'
\end{tikzcd}
\end{equation}
Now the map $z\in\C^d\mapsto Lz\in\C^{d'}$ satisfies
\begin{equation}
\label{commut}
L(z+h(p))=Lz+h'(H(p))
\end{equation}
and descends as a stack morphism $\mathscr{L}^{cal}$ from $\nctorusgerbe{h}{J}$ onto $\nctorusgerbe{h'}{J'}$ so we have
\begin{equation}
\label{CDdeftorusmorphismdec}
\begin{tikzcd}[row sep=small]
\R^d\arrow[d,hook]\arrow[r, "L"]&\R^{d'}\arrow[d,hook]\\
\C^d\arrow[d]\arrow[r,"L"]&\C^d\arrow[d]\\
\nctorusgerbe{h}{J}\arrow[r,"\mathscr L^{cal}"]&\nctorusgerbe{h'}{J'}
\end{tikzcd}
\end{equation} 

Now, we will ask $H$ to preserve the virtual generators, that is, we both impose that there exists a map $s$ between the sets of virtual generators $J$ and $J'$ such that $H$ satisfies 
\begin{equation}
\label{Hs}
H(e_i)=e_{s(i)}\qquad\text{ for }\qquad i\in J
\end{equation} 
and that $H$ satisfies
\begin{equation}
\label{Hrestric}
H(e_i)\in \bigoplus_{j\not\in J'}\Z e_j\qquad\text{ for }\qquad i\not\in J
\end{equation}
So finally, we define

\begin{definition}
	\label{deftorusmorphismdec}
	A {\it calibrated torus morphism} $\mathscr L^{cal}$ from $\nctorusgerbe{h}{J}$ onto $\nctorusgerbe{h'}{J'}$ is a stack morphism such that there exists a linear map  $L$ from $\R^d$ to $\R^{d'}$, a linear map $H$ from $\R^n$ to $\R^{n'}$ and a map $s$ from $J$ to $J'$ satisfying {\rm \eqref{CDLH}}, {\rm \eqref{commut}}, {\rm \eqref{CDdeftorusmorphismdec}} as well as
	{\rm \eqref{Hs}} and finally {\rm\eqref{Hrestric}}.\vspace{5pt}\\
	It is a {\it isomorphism} if the three of $L$, $H$ and $s$ are isomorphisms. And it is a {\it marked isomorphism} if, moreover, $J$ is equal to $J'$ and $s$ is the identity.
\end{definition}
To wit, given $T\to A$ a $\Z^n$-cover, set
\begin{equation}
\label{TLGdec}
T\times_{H}\Z^{n'}:=\{(t,q')\in T\times \Z^{n'}\}/\Z^n
\end{equation}
for the following action of $\Z^n$
\begin{equation}
\label{TLGactiondec}
p\cdot (t,q')=(p\cdot t,q'-H(p))
\end{equation}
Let also
\begin{equation}
\label{TLGmdec}
[(t,q')]\longmapsto (m^{cal})'[(t,q')]:=Lm^{cal}(t)+h'(q')
\end{equation}
which is well-defined since
\begin{equation*}
\begin{aligned}
Lm^{cal}(p\cdot t,q'-H(p))&=Lm^{cal}(p\cdot t)+h'(q')-h'H(p)\\
&=Lm^{cal}(t)+Lh(p)+h'(q')-h'H(p)\\
&=Lm^{cal}(t)+h'(q')\quad\text{ since }Lh=h'H
\end{aligned}
\end{equation*}
Given $f:T\to S$ a morphism of $\Z^n$-cover, over $A\to B$, we define
\begin{equation}
[(t,q')]\in T\times_H\Z^{n'}\longmapsto f_H([(t,q')]:=[(f(t),q')])\in S\times_H\Z^{n'}
\end{equation}
and, if $m^{cal}$, resp. $n^{cal}$, is the equivariant map associated to $T\to A$, resp. $S\to B$, so that $m^{cal}=n^{cal}\circ f$, we have
\begin{equation*}
(n^{cal})'\circ f_H([(t,q')])=Ln^{cal}(f(t))+h'(q')=
\begin{aligned}
&Lm^{cal}(t)+h'(q')\\
=&(m^{cal})'([t,q'])
\end{aligned}
\end{equation*}
So we finally have
\begin{equation}
\label{torusmorphismLdec1}
\mathscr{L}^{cal}(T,A,m^{cal})=(T\times_H\Z^{n'},A,(m^{cal})')
\end{equation}
and
\begin{equation}
\label{torusmorphismLdec2}
\begin{aligned}
\mathscr{L}^{cal}&(f : T\to S,A\to B,m^{cal},n^{cal})=(f_H :T\times_H\Z^{n'}\\
& \to S\times_H\Z^{n'},A\to B, (m^{cal})', (n^{cal})')
\end{aligned}
\end{equation}

By \eqref{deccondition}, there exists a basis of $\C^d$ formed by vectors $h(e_i)$ with $i\not\in J$. We may assume that $(h(e_1),\hdots,h(e_d))$ is such a basis. We also may assume that $J$ is given by $\{n-\vert J\vert+1,\hdots,n\}$. Otherwise, there exists a permutation of $\{1,\hdots n\}$ with those properties. Let $H$ denote the matrix corresponding to this permutation and let $s$ be the induced set mapping between $J$ and $\{n-\vert J\vert+1,\hdots,n\}$. Then $(Id, H,s)$ defines a calibrated torus isomorphism between $\nctorusgerbe{h}{J}$ and $\nctorusgerbe{\Gamma}{hH^{-1}}$ and we may replace $h$ with $hH^{-1}$. In the same way, there exists some linear isomorphism $L$ of $\C^d$ sending $hH^{-1}(e_1),\hdots, hH^{-1}(e_d)$ onto the canonical basis $(e_1,\hdots,e_d)$. With $H=Id$ and $s=Id$, it defines a calibrated torus isomorphism between $\nctorusgerbe{\Gamma}{hH^{-1}}$ and $\nctorusgerbe{L\Gamma}{LhH^{-1}}$. \vspace{5pt}\\
So, defining
\begin{definition}
	\label{defstandardh}
	Let $(e_1,\hdots,e_n)$, resp. $(e_1,\hdots,e_d)$ be the canonical basis of $\R^n$, resp. $\R^d$.
	We say that the calibration $h:\Z^n\to\Gamma$ is {\it standard} if 
	\begin{enumerate}[i)]
		\item $h(e_i)=e_i$ for $i$ between $1$ and $d$. 
		\item The set of virtual generators is $\{n-\vert J\vert+1,\hdots,n\}$.
	\end{enumerate}
\end{definition}
\noindent We just proved
\begin{lemma}
	\label{lemmastandarddec}
	Any calibrated quantum torus is isomorphic to one standardly calibrated.
\end{lemma}

If $h$ is standard, we may decompose it as  
\begin{equation}
\label{h}
(p,q)\in\Z^d\times \Z^{n-d}\longmapsto p+\hbar(q)\in\Gamma
\end{equation}

This allows us to describe easily the multiplicative form of $\nctorusgerbe{h}{J}$ as the quotient stack $[\Torus^d/\Z^{n-d}]$ where $p\in\Z^{n-d}$ acts multiplicatively on $w\in\Torus^d$ through the formula $wE(\hbar(p))$ (which means that each coordinate of $w$ is multiplied by the corresponding coordinate of $E(\hbar(p))$).
 We have a commutative diagram
\begin{equation}
\label{CDaddvsmultdectorus}
\begin{tikzcd}[column sep=large]
\C^d \arrow[d]\arrow[r,"E"] &\Torus^d\arrow[d]\\
\nctorusgerbe{h}{J}\arrow[r,"\simeq"]&\phantom{d}\kern-5pt[\Torus^d/\Z^{n-d}]
\end{tikzcd}
\end{equation} 

In a more functorial point of view, $E$ defines a stack isomorphism $\mathscr E^{cal}$ from $\nctorus{\Gamma}{h}$ to $[\Torus^d/\Z^{n-d}]$ such that
\begin{equation*}
\mathscr{E}^{cal}(T,A,m^{cal})=(T/\Z^d,A,Em^{cal})
\end{equation*}
with $p\in\Z^d$ acting on $t$ as $(p,0)\in\Z^d\times\Z^{n-d}=\Z^n$ on $t$ and with
\begin{equation*}
\begin{tikzcd}
T\arrow[d]\arrow[dd,bend right=50]\arrow[r,"m^{cal}"]&\C^d\arrow[d,"E"]\\
T/\Z^d\arrow[d]\arrow[r,"Em^{cal}"]&\Torus^d\\
A &
\end{tikzcd}
\end{equation*}
and
\begin{equation*}
\begin{aligned}
&\mathscr{E}^{cal}(f:T\to S,A\to B,m^{cal},n^{cal})=(f:T/\Z^d\to S/\Z^d,\\
&A\to B,Em^{cal},En^{cal})
\end{aligned}
\end{equation*}
where $f$ descends as a morphism from $T/\Z^d$ to $S/\Z^d$ since it satisfies by definition $f(p\cdot t)=p\cdot f(t)$ for all $p\in\Z^n$.

\begin{example}
	\label{examplesqrt2dec}
	As in Example \ref{examplesqrt2}, let $d=1$ and $\Gamma$ be generated by $1$ and $\sqrt 2$ and let $L$ be the linear map $z\mapsto z\sqrt 2$. Consider the calibration
	\begin{equation*}
	(x,y,z)\in\R^3\longmapsto x+y\sqrt 2 \in \R
	\end{equation*}
	with set of virtual generators $\{3\}$. We have
	\begin{equation*}
	L(z+h(p))=Lz+h(2p_2,p_1,p_3)
	\end{equation*} 
	hence, setting $H(x,y,z)=(2y,x,z)$, we obtain a calibrated torus morphism
	\begin{equation*}
	[z]\in[\C/\Z^3]\longmapsto [\sqrt{2}z]\in[\C/\Z^3]
	\end{equation*}
	or, in multiplicative form,
	\begin{equation*}
	[w]\in[\Torus/\Z^2]\longmapsto [w^{\sqrt{2}}]\in[\Torus/\Z^2]
	\end{equation*}
	However, if we use the calibration 	
	\begin{equation*}
	(x,y,z)\in\R^3\longmapsto x+y\sqrt 2+z \in \R
	\end{equation*}
	an easy computation proves that there is no calibrated torus morphism corresponding to $L$, for the requirement that $H$ fixes the virtual generator gives a contradiction.
\end{example}

\section{The category of Quantum Fans}
\label{Qfans}
In this Section, we describe the Quantum Fans, that is the combinatorial data needed to construct Quantum Toric Stacks. In what follows we will refer to Quantum Toric Stacks shortly as \emph{Quantum Torics}.

\subsection{Quantum Fans}
\label{Qfan}
Let $\Gamma\subset \R^d$ be a finitely generated additive subgroup such that $\text{Vect}_{\R}\Gamma=\R^d$.

\begin{definition}
\label{QFdef}
A {\it Quantum Fan} $(\Delta, v)$ in $\Gamma$ consists of 
\begin{enumerate}
\item[i)] A collection $\Delta$ of strongly convex cones in $\Gamma$ such that every intersection of cones is a cone, every face of a cone is a cone and $0$ is a cone.
\item[ii)] The choice, for every $1$-cone of $\Delta$ of a generating vector $v_i\in\Gamma\setminus\{0\}$.
\end{enumerate}
\end{definition}
We set $v=(v_1,\hdots, v_p)$. If $\sigma$ is the cone generated by $(v_{i_1},\hdots v_{i_k})$, we use the notation $\sigma=\langle i_1\hdots i_k\rangle$.
\begin{remarks}
	\label{rkordinaryfans}
	To compare with known cases, observe that
	\begin{enumerate}
		\item[i)] The case of classical Toric Fans corresponds to the case $\Gamma$ discrete and $v_i$ the unique primitive vector of $\Gamma$ generating the corresponding $1$-cone. Here the data $v$ is completely determined by $\Delta$.
		\item[ii)] The case of Orbifold Toric Fans (stacky Fans of \cite{BCS}) corresponds to the case $\Gamma$ discrete and $\Delta$ simplicial. Here $v_i$ is not assumed to be primitive so choosing $v$ means  choosing a positive multiple of each primitive generator. 
	\end{enumerate}
\end{remarks}

\begin{definition}
	\label{defsfans}
	We say that a Quantum Fan $(\Delta,v)$ in $\Gamma$ is 
	\begin{enumerate}[\rm i)]
		\item {\it irrational} if $\Gamma$ is not discrete in $\R^d$
		\item {\it simplicial} if every cone of $\Delta$ is a cone over a simplex
		\item {\it complete} if the union of all cones of $\Delta$ cover $\R^d$ entirely
		\item {\it $\Gamma$-complete} if $\Gamma=\Z v_1+\hdots +\Z v_p$.
		\item {\it polytopal} if there exists a convex polytope $P$ with vertices in $\Gamma$ such that $\Delta$ is the fan over the faces of $P$.
		\end{enumerate}
\end{definition}
Notice that only points i) and iv) are  specific to Quantum Fans.
We define now morphisms of Quantum Fans.
\begin{definition}
\label{QFmorphismdef}
A {\it morphism of Quantum Fans} between $(\Delta, v)$ in $\Gamma$ and $(\Delta', v')$ in $\Gamma'$ is a linear map
\begin{equation}
\label{QFmdef}
L\ :\ \R^d\longrightarrow \R^{d'}
\end{equation}
such that
\begin{enumerate}
\item[i)] $L(\Gamma)\subset \Gamma'.$
\item[ii)] If $\sigma$ is a cone of $\Delta$ then $L(\sigma)\subseteq\sigma'$ for some cone $\sigma'$ of $\Delta'$.
\item[iii)] For each $i$ and each $\sigma'=\langle j_1\hdots j_k\rangle$ such that $L(v_i)\subseteq\sigma'$; then $L(v_i)$ is a $\N$-linear combination of $(v'_{j_1},\hdots,v'_{j_k})$.
\end{enumerate}
\end{definition}

\begin{remarks}
Observe that
\begin{enumerate}
\item[i)] For ordinary fans, that is $\Gamma$ discrete and $v_i$ primitive generators, then point iii) is automatically satisfied.
\item[ii)] For  stacky fans, definition \ref{QFmorphismdef} coincides with that of \cite[Remark 4.5]{BCS}.
\end{enumerate}
\end{remarks}

The following characterization of isomorphisms will be useful.

\begin{lemma}
\label{isoQFlemma}
Let $L$ be a Quantum Fan morphism between $(\Delta,v)$ and $(\Delta',v')$. The following two statements are equivalent
\begin{enumerate}
\item[i)] $L$ is a Quantum Fan isomorphism.
\item[ii)] $L$ is a linear isomorphism, $L(\Gamma)=\Gamma'$, both Quantum Fans have the same number of $1$-cones, say $p$, and $L$ sends $(v_1,\hdots,v_p)$ onto a permutation of $(v'_1,\hdots,v'_p)$.
\end{enumerate}
\end{lemma}

\begin{proof}
Assume i). From the definition, if $L$ is an isomorphism of Quantum Fans then it is a linear isomorphism, $L(\Gamma)=\Gamma'$ and $L(\sigma)=\sigma'$ in point ii). In particular, $L(v_i)$ is a positive multiple of some $v'_{i'}$. Point iii) in definition \ref{QFmorphismdef}  shows that it is a positive integer multiple, with integer inverse. Hence, for all $i$, we have that $L(v_i)$ is equal to some $v'_{i'}$ proving ii).
The converse is obvious.
\end{proof}

Especially, observe that every quantum fan with $\text{Vect}(v_1,\hdots,v_p)=\R^d$ is isomorphic to one satisfying 
\begin{enumerate}
\item[i)] $\Z^d\subset \Gamma$.
\item[ii)] $(v_1,\hdots, v_d)$ is the canonical basis of $\R^d$.
\end{enumerate}
If $\text{Vect}(v_1,\hdots,v_p)$ has dimension $l<d$, it is isomorphic to a Quantum Fan satisfying 
\begin{enumerate}
\item[i)] $\Z^l\times\{0\}\subset \Gamma$.
\item[ii)] $(v_1,\hdots, v_l)$ is the canonical basis of $\R^l\times\{0\}\subset \R^d$.
\end{enumerate}

\begin{definition}
\label{standarddef}
We call {\it standard} a Quantum Fan satisfying the previous normalization conditions.
\end{definition}
Quantum Fans and their morphisms form a category that we denote by $Q$.
\begin{example}
	\label{exP2}
	Consider the $\Gamma$-complete and complete standard Quantum Fan of $\R^2$ generated by
	\begin{equation*}
	v_1=e_1\qquad v_2=e_2\qquad v_3=ae_1+be_2
	\end{equation*}
	with $a<0$ and $b<0$. For $a=b=-1$, this is exactly the fan of the classical $\mathbb P^2$.
	
	We call such a fan {\it a $\Gamma$-complete quantum deformation} of $\mathbb P^2$'s fan.
\end{example}
\subsection{Calibrated Quantum Fans}
\label{decQfan}
We now introduce the finer notion of calibrated quantum fan.

\begin{definition}
\label{decQFdef}
A {\it calibrated Quantum Fan} $(\Delta, h)$ in $\Gamma$ consists of 
\begin{enumerate}
\item[i)] A collection $\Delta$ of strongly convex cones in $\Gamma$ such that every intersection of cones is a cone, every face of a cone is a cone and $0$ is a cone. 
\item[ii)] A calibration $h : \Z^n\to\Gamma$ with its set $J$ of virtual generators (see Definition  \ref{defcalibration}).
\item[iii)] A set of generators, that is a subset $I=\{i_1,\hdots, i_p\}$ in $\{1,\hdots, n \}$ which is disjoint from $J$ and such that the $1$-cones generated by $h(e_{i_k})$ for $k=1,\hdots, p$ are exactly the $1$-cones of $\Delta$ 
\end{enumerate}
The {\it length} $l$ of the calibrated Quantum Fan $(\Delta,h)$ is defined as the nonnegative quantity $\vert J\vert$. It is {\it maximal} if equal to $n-p$.
\end{definition}
We immediately see that $(\Delta,h)$ determines canonically a Quantum Fan $(\Delta,v)$ by setting $v=(h(e_{i_1}),\hdots, h(e_{i_p}))$. However, it contains strictly more information than a quantum fan as soon as $J$ is not empty, as we shall see in Section \ref{decodef}. 
\\
Indeed, one usually goes the other way round. Starting with a Quantum Fan, we calibrate it by choosing an epimorphism $h$ and a set of virtual generators. The  following two cases are important.

\begin{example}\textbf{The Trivial calibration of a $\Gamma$-complete fan.}
	\label{completecalibration}
	Let $\Delta$ be a $\Gamma$-complete fan in $\Gamma\subset \R^d$ with $1$-cone generators $v_1,\hdots v_p$. Since it is $\Gamma$-complete, we can calibrate it with $h : \Z^p\to \Gamma$ by setting $h(e_i)=v_i$ for $i=1,\hdots,p$. This forces the set of virtual cones to be empty.
\end{example}

A calibrated quantum fan also determines a fan in the standard lattice $\Z^n$ as follows. Let
\begin{equation}
\label{imath}
i_j\in I\longmapsto \imath(j):= j\in \{1,...p\}
\end{equation} 
Given $\sigma=\langle j_1,\hdots, j_k\rangle$ a cone of $\Delta$, we consider the cone $\imath^*\sigma$ in $\Z^n\subset \R^n$ generated by $\imath^{-1}(j_1),\hdots, \imath^{-1}(j_k)$. When $\sigma$ runs over the cones of $\Delta$, it describes a classical fan $(\Delta_h,e_I)$ in the lattice of integer points of $\R^n$. 

\begin{example}\textbf{Canonical calibration of $(\Delta_h,e_I)$}.
		\label{trivialdecDh}
		Since the $1$-cone generators of $(\Delta_h,e_I)$ are vectors of the canonical basis of $\R^n$, it can trivially calibrated by the identity as in Example \ref{trivialdec}. And we can take $J$ as subset of virtual generators. We denote this calibrated fan by $(\Delta_h,Id)$ and call its calibration the {\itshape canonical calibration} of $(\Delta_h,e_I)$.
\end{example} 

\begin{remark}
	\label{rkimath}
	Observe that $\imath^*$ realizes an order preserving bijection from the cones of $\Delta$ to the cones of $\Delta_h$. We denote by $\imath_*$ its inverse, which is also order preserving. 
\end{remark}
Let us now define morphism of calibrated Quantum Fans.
\begin{definition}
\label{decQFmorphismdef}
A {\it morphism of calibrated Quantum Fans} between $(\Delta, h)$ in $\Gamma$ and $(\Delta', h')$ in $\Gamma'$ is a pair $(L,H)$, where
\begin{enumerate}[\rm i)]
\item $L$ is a morphism between the associated Quantum Fans $(\Delta, v)$ and $(\Delta', v')$.
\item $H$ is a morphism between the associated fans $(\Delta_h, e_I)$ and $(\Delta'_{h'}, e_{I'})$. 
\item We have $L\circ h=h'\circ H$.
\item Each $e_i$ for $i\not\in J$ is sent through $H$ onto a $\Z$-linear combination of $(e_i)_{i\not\in J'}$.
\item There exists a map $s$ from $J$ to $J'$ such that $H$ sends every vector $e_i$ with $i\in J$ onto $e_{s(i)}$.
\end{enumerate}
	It is an {\it isomorphism} if the three of $L$, $H$ and $s$ are isomorphisms. And it is a {\it marked isomorphism} if moreover $J$ is equal to $J'$ and $s$ is the identity.
\end{definition}
The last requirement implies that $H$ sends the set of virtual generators of $(\Delta,h)$ into the same set for $(\Delta',h')$. So we have to think of these vectors as marked vectors.
Observe also that the fourth requirement is automatic for $e_i$ being generator of a $1$-cone since $H$ is a toric morphism (and indeed as a $\mathbb N$-linear combination); so it really concerns those $e_i$ with $i\not\in J$ which are not generators of $1$-cones. It agrees with definition \ref{deftorusmorphismdec} and implies that $H$ splits as
\begin{equation*}
H=\begin{pmatrix}
H_1 &0\\
0 &H_2
\end{pmatrix}
\end{equation*}
with $H_1$, respectively $H_2$, a $\vert\bar J\vert$ square matrix, resp. a $\vert J\vert$ square matrix.

Notice also that, with this definition, $H$ acts on the cones of $\Delta_h$ exactly as $L$ acts on the cones of $\Delta$. To wit,
\begin{lemma}
	\label{lemmadecpreserve}
	Let $(L,H)$ be a morphism of calibrated Quantum Fans between $(\Delta, h)$ in $\Gamma$ and $(\Delta', h')$ in $\Gamma'$. Let $\sigma$, resp. $\sigma'$, be a cone of $\Delta$, resp. $\Delta'$. Let $\imath^*\sigma$, resp. $(\imath')^*\sigma'$, be the cone of $\Delta_h$, resp. $\Delta_{h'}$, associated through {\rm \eqref{imath}}. Then,
	$$
	L\sigma\subset \sigma'\iff H\imath^*\sigma\subset (\imath')^*\sigma'
	$$
\end{lemma}

\begin{proof}
	Use
	\begin{equation*}
	L\sigma=Lh\imath^*\sigma=h'H\imath^*\sigma
	\end{equation*}
	If $H\imath^*\sigma\subset (\imath')^*\sigma'$, then, applying $h'$ to the previous relation, we obtain $L\sigma\subset \sigma'$.\\
	Conversely, assume $L\sigma\subset \sigma'$. Still by the previous relation, we obtain 
	$$
	h'H\imath^*\sigma=L\sigma\subset \sigma'=h'(\imath')_*\sigma'
	$$
	By construction, $h'$ realizes a bijection between the cones of $\Delta_h$ and those of $\Delta$. Hence we conclude $H\imath^*\sigma\subset (\imath')^*\sigma'$.
\end{proof}

The next lemma shows that $h$ induces a Quantum Fan morphism between $(\Delta_h,Id)$ and $(\Delta, h)$.
\begin{lemma}
	\label{hdecmorphism}
	Let $(\Delta,h)$ be a calibrated Quantum Fan and let $(\Delta_h,Id)$ be the associated trivially calibrated fan. Then $(h,Id)$ is a Quantum Fan morphism from $(\Delta_h,Id)$ to $(\Delta, h)$.
\end{lemma}

\begin{proof}
	The proof is straightforward. Firstly, $h$ sends $\Delta_h$ onto $\Delta$ by definition. Secondly, the identity is obviously a fan isomorphism of $(\Delta_h,Id)$. Thirdly, the following diagram commutes
	\begin{equation}
	\label{CDhId}
	\begin{tikzcd}
	\Z^n \arrow[r,"Id"]\arrow[d,"Id"']&\Z^{n}\arrow[d,"{h}"]\\
	\Z^n\arrow[r,"h"']&\Gamma
	\end{tikzcd}
	\end{equation}
	Fourthly, the set of virtual vectors of $\Delta_h$ is empty, so there is nothing to check in points iv) and v) of Definition \ref{decQFmorphismdef}.
\end{proof}
For later use, we note that $(L,H)$ induces a homomorphism $K$ between the kernel $\Xi$ of $h$ and the kernel $\Xi'$ of $h'$ such that the following diagram is commutative
\begin{equation}
\begin{tikzcd}
\label{completemorphismCD}
0 \arrow[r] &\Xi \arrow[r]\arrow[d,"K"'] &\Z^n \arrow[r,"h"]\arrow[d,"H"'] &\Gamma \arrow[r]\arrow[d,"L"'] & 0\\
0 \arrow[r] &\Xi' \arrow[r]  &\Z^{n'} \arrow[r,"h'"'] &\Gamma'\arrow[r] &0
\end{tikzcd}
\end{equation}

We also give a characterization of isomorphisms in the maximal length case. 

\begin{lemma}
\label{isoDQFlemma}
The following two statements are equivalent
\begin{enumerate}
\item[i)] $(L,H)$ is an isomorphism between the maximal calibrated Quantum Fans $(\Delta,h)$ and $(\Delta',h')$.
\item[ii)] $L$ is a Quantum Fan isomorphism and $H$ is a block permutation matrix that permutes both $\{e_{i_1},\hdots,e_{i_p}\}$ and $\{e_{j_1},\hdots,e_{j_{n-p}}\}$. 
In particular, $L$ sends the $1$-cones generators $(v_1,\hdots,v_p)$ onto a permutation of $(v'_1,\hdots,v'_p)$ and the virtual generators $(h(e_{j_1}),\hdots, h(e_{j_{n-p}}))$ onto a permutation of $(h'(e_{j'_1}),\hdots, h'(e_{j'_{n-p}}))$.
\end{enumerate}
Moreover, $(L,H)$ is a marked isomorphism if and only if in addition $J=J'$ and the permutation induced by $H$ on $(h(e_{j_1}),\hdots, h(e_{j_{n-p}}))$ is the identity.
\end{lemma}

\begin{proof}
Assume i). From the definition,  $(L,H)$ is an isomorphism implies that $L$ is a Quantum Fan isomorphism and that $H$ is an isomorphism. Hence $s$ is a permutation. Moreover, from Lemma \ref{isoQFlemma}, we have $L$ induces a permutation of $(v_1,\hdots,v_p)$ so we are done. The converse is obvious, as well as the marked case.
\end{proof}

As in the proof of Lemma \ref{lemmastandarddec}, by a permutation $H$ in $\R^n$ and a linear change of coordinates $L$ in $\R^d$, we may assume that $J$ is given by the last $\vert J\vert$ elements of $\{1,\hdots, n\}$ and that $h$ is in standard form $\eqref{h}$. To be more precise, this linear change of coordinates $L$ straighten a basis $B$ of $\R^d$ formed by vectors $h(e_i)$ with $i\not \in J$ onto the canonical basis. Let $k$ be the dimension of the vector subspace of $\R^d$ generated by $v_1,\hdots, v_p$. Say for simplicity in the notations that $v_1,\hdots,v_k$ form a basis of this subspace. We may take as first $k$ vectors of $B$ the $1$-cones generators $v_1,\hdots,v_k$. Hence, through  $(L,H)$, we put both $h$ and $\Delta$ standard. Finally, performing another permutation of $\R^n$ if necessary, we may assume that $I$ is sent to $\{1,\hdots,k,d+1,\hdots,d+p-k\}$.

Summing up, we just show that $(\Delta, h)$ is isomorphic to a calibrated Quantum Fan satisfying
\begin{enumerate}
	\item[i)] The calibration $h$ is standard.
	\item[ii)] The induced fan $(\Delta,v)$ is standard.
	\item[iii)] The set of generators $I$ is $\{1,\hdots,k,d+1,\hdots,d+p-k\}$.
\end{enumerate}
\noindent and we define
\begin{definition}
	\label{defstandarddecfan}
	We call {\it standard} a calibrated fan satisfying the previous normalization conditions.
\end{definition}

We say that $(\Delta,h)$ is simplicial (resp. irrational, ... see Definition \ref{defsfans}) if the underlying $(\Delta,v)$ is simplicial (resp. irrational, ...). Calibrated Quantum Fans and their morphisms form a category that we denote by $Q^{cal}$.

We finish with some examples of (calibrated) Quantum Fans and morphisms.

\begin{example}
\label{P1QFans}
Let 
\begin{equation}
 \label{quantumP1v12}
v_1=1\qquad v_2=-1
\end{equation}
and 
\begin{equation}
\label{quantumP1Delta}
\Delta=\{0,\R^+\cdot v_1,\R^+\cdot v_2\}
\end{equation}
Let $a\in\R$ and let finally
\begin{equation}
 \label{quantumP1Gamma}
\Gamma=\Z+a\Z
\end{equation}
This is a standard Quantum Fan. We can calibrate it by defining
\begin{equation*}
(x,y,z)\in\Z^3\longmapsto h_a(x,y,z):=x-y+az\in\Gamma
\end{equation*}
and $J=\{3\}$. Note that $I=\{1,2\}$, thus the fan is maximal. There are only two isomorphisms of $(\Delta,v)$: the identity $Id$ and $-Id$. However $-Id$ is not always an isomorphism of $(\Delta, h_a)$. Setting
\begin{equation}
\label{HP1inversion}
H(e_1)=e_2\qquad H(e_2)=e_1\qquad H(e_3)=e_3
\end{equation}
(there is no other possible choice following Lemma \ref{isoDQFlemma}), we need the commutation property $h_aH=Lh_a$. This only occurs if $a$ is zero. In that case, $(-Id, H)$ is an isomorphism of $(\Delta, h_0)$. 

More generally, $(-Id, H)$ with $H$ as in \eqref{HP1inversion} is an isomorphism from $(\Delta, h_a)$ to $(\Delta, h_{-a})$.
\end{example}

\begin{example}
	\label{exP2dec}
	Consider a $\Gamma$-complete quantum deformation of the $\mathbb P^2$'s fan as in Example \ref{exP2}. Its trivial calibration is given by
	\begin{equation}
	h(x,y,z)=(x+az,y+bz)
	\end{equation}
	Note that the trivial calibration is in standard form because we calibrated a standard fan.
\end{example}
\begin{example}
\label{P1P2morphisms}
Consider the calibrated Quantum Fan in $\R^2$ with canonical basis $(e_1,e_2)$
\begin{equation*}
v_1=e_1\qquad v_2=e_2\qquad v_3=-e_1-e_2\qquad w=xe_1+ye_2
\end{equation*}
with $x$ and $y$ arbitrary real numbers. Here we denote by $\Delta_2$ the complete fan generated by $v_1,v_2,v_3$, $\Gamma$ is generated by $v_1,v_2,v_3$ and $w$, and $h: \Z^4\to \Gamma$ is defined through $h(e_i)=v_i$ and $h(e_4)=w$. We assume it is maximal.
\\
We look for morphisms $(L, H)$ from the calibrated fan $(\Delta, h_a)$ of example \ref{P1QFans} to $(\Delta_2,h)$. Starting with $L$ being a matrix $^t(\alpha,\beta)$, we see from point iii) of Definition \ref{QFmorphismdef} that $\alpha$ and $\beta$ are integers.  Then, $H$ must send the virtual generator $e_3$ onto the virtual generator $e_4$. Through the equality $hH=Lh_a$, this yields $x=a\alpha$ and $y=a\beta$. Since $\alpha$ and $\beta$ are integers, we obtain that there exists such a morphism $(L,H)$ if and only if $x\in\Z a$ and $y\in \Z a$.\\
Assume this condition is fulfilled. Then there are two cases for $L$. Either $a$ and thus $x$ and $y$ are zero and $L$ may take any integer entries. Observe that this is the classical case, i.e. both fans are classical fans. Or $a$ is not zero, so $\alpha$ and $\beta$ are fixed by the previous condition.\\
We have now to determine $H$.  Assume $\alpha$ and $\beta$ are nonnegative and $\alpha\geq \beta$. {Then $L$ sends $\R^+\cdot e_1$ into the cone $\langle 1,2\rangle$ and $\R^+\cdot (-e_1)$  into the cone $\langle 2,3\rangle$, hence} $H(e_1)$, respectively $H(e_2)$ is a linear combination of $e_1$, $e_2$, resp. $e_2$ and $e_3$, with nonnegative integer coefficients. From this data, a straightforward computation shows that the unique admissible $H$ have the form
\begin{equation*}
H=\begin{pmatrix}
&\alpha &0 &0\cr
&\beta &\alpha-\beta &0\cr
&0 &\alpha &0\cr
&0 &0 &1
\end{pmatrix}
\end{equation*}
So, we finally get that there exists a morphism $(L, H)$ from $(\Delta, h_a)$ to $(\Delta_2,h)$ if and only if $x=a\alpha$ and $y=a\beta$. Moreover, if in addition $a=0$, there exists a unique morphism $(L,H)$ for each choice of $(\alpha,\beta)$ in $\Z^2$; and if $a$ is not zero, there exists a unique morphism $(L,H)$. The other cases ($\beta\geq \alpha\geq 0$ and so on) are treated similarly.
\end{example}

\section{The definition (atlas) of Quantum Toric Varieties}
\label{Qatlas}
In this section, we give an explicit description of a Quantum Toric Variety associated to a simplicial Quantum Fan as a stack over affine toric varieties.

Geometrically a quantum toric variety has local charts modelled onto the quotient of $\C^d$ by a discrete group, and with gluings given by monomials with possibly irrational exponents, that is, is a quasifold, as introduced in \cite{BP1} and \cite{BP2}.

However, to work with such an object, it is necessary to give a more functorial description of it. In the same way that orbifolds are functorially described as Deligne-Mumford stacks, we replace the quasifold type description with a stack construction.

 Recall Section \ref{stacks}. Every (classical) toric variety is a stack over $\mathfrak A$ and is indeed characterized by a descent data of affine toric varieties. For example, the standard $\mathbb P^1$ is obtained by gluing two copies of the affine toric variety $\C$ over $\C^*$ through the map $z\to 1/z$. As a stack over $\mathfrak A$, it may be presented as the following descent data of affine toric varieties. 

An object over $T\in\mathfrak A$ is a pair $(\begin{tikzcd}
T_1\arrow[r,"m_1"] &\C
\end{tikzcd}, \begin{tikzcd}
T_2\arrow[r,"m_2"] &\C
\end{tikzcd})$ such that
\begin{enumerate}
	\item[i)] $T=T_1\cup T_2$ is a covering of $T$.
	\item[ii)] We have $m_1(T_1\cap T_2)=m_2(T_1\cap T_2)=\C^*$.
	\item[iii)] The following diagram commutes
	$$\begin{tikzcd}
	T_1\cap T_2 \arrow[r,"m_1"] \arrow[dr, "m_2"']&\C^*\arrow[d,"z\mapsto 1/z"]\\
	&\C^*
	\end{tikzcd}
	$$
\end{enumerate}
And a morphism above $\begin{tikzcd}
T\arrow[r,"f"] &S
\end{tikzcd}$
between $(\begin{tikzcd}
T_1\arrow[r,"m_1"] &\C
\end{tikzcd}, \begin{tikzcd}
T_2\arrow[r,"m_2"] &\C
\end{tikzcd})$ and 
$(\begin{tikzcd}
S_1\arrow[r,"n_1"] &\C
\end{tikzcd}, \begin{tikzcd}
S_2\arrow[r,"n_2"] &\C
\end{tikzcd})$ must satisfy $n_i\circ f_i=m_i$ for $i=1,2$.
Of course, this is isomorphic to the stack $\underline{\mathbb P^1}$ and is a complicated way of describing it. But the point here is that, once given the category of affine toric varieties, general toric varieties can be defined directly and functorially through this descent data procedure.

In the Quantum case, we will first define affine simplicial quantum toric varieties as discrete quotient stacks; and then general simplicial quantum toric varieties through descent.
 
\subsection{Standard affine Quantum Toric Varieties and their morphisms}
\label{NDstandardCone}
Let $C_{k,d}$ be the standard simplicial cone generated by $(v_1=e_1,\hdots,v_k=e_k)$ in $\R^d$. Let $\Gamma=\Z^d+\Gamma_0$ be in standard form also.

Observe that the group $E(\Gamma)$ acts freely on $\Torus^d$.
The Standard affine Quantum Toric Variety associated to $C_{k,d}$ and $\Gamma$ is the quotient stack
\begin{equation}
\label{QstaffineTV}
Q_{k,d,\Gamma}=\left [
\C^k\times \Torus^{d-k}/E(\Gamma)
\right ]
\end{equation}
Hence, an object over the affine toric variety $T$ is simply a $E(\Gamma)$-cover $\tilde T$ with an equivariant map $m$ 
\begin{equation}
\label{affQTVobject}
\begin{tikzcd}
\tilde T \arrow[r,"m"]\arrow[d] &\C^k\times \Torus^{d-k}\\
T & &
\end{tikzcd}
\end{equation}
and a morphism over $T\to S$ is given by
\begin{equation}
\label{affQTVmorphism}
\begin{tikzcd}
         &            &\C^k\times \Torus^{d-k}\\
\tilde T \arrow[r,"f"']\arrow[rru,bend left=20,"m"]\arrow[d]          &\tilde S  \arrow[d]\arrow[ru,bend right=10,"n"']      &                       \\
T        \arrow[r]         &S       &          
\end{tikzcd}
\end{equation}
Taking $k=0$ in \eqref{affQTVobject} and \eqref{affQTVmorphism} gives the description of the Quantum Torus $\nctorus{d}{\Gamma}$ of Definition \ref{defQTorus} (in its multiplicative form).

Inclusion of a standard $l$-affine quantum toric variety into a standard $k$-cone quantum toric variety is given by inclusion of $\C^l\times \Torus^{d-l}$ into $\C^k\times \Torus^{d-k}$ that is
\begin{equation}
\label{affQTVinclusion}
\begin{tikzcd}
\tilde T \arrow[r,"m"]\arrow[d] &\C^l\times \Torus^{d-l}\arrow[r,hook]&\C^k\times\Torus^{d-k}\\
T & & &
\end{tikzcd}
\end{equation}
and
\begin{equation}
\label{affQTVmorphisminclusion}
\begin{tikzcd}
&            &\C^l\times \Torus^{d-l}\arrow[r,hook]&\C^k\times\Torus^{d-k}\\
\tilde T \arrow[r]\arrow[rru,bend left=20, "m"]\arrow[d]          &\tilde S  \arrow[d]\arrow[ru,bend right=10,"n"']      &  &                     \\
T        \arrow[r]         &S       &   &       
\end{tikzcd}
\end{equation}
In particular, taking $l=0$, we have an inclusion of the Quantum Torus $\nctorus{d}{\Gamma}$ into $Q_{k,d,\Gamma}$.\vspace{5pt}\\
Let us now describe morphisms from a standard $k$-cone onto a standard $k'$-cone, which is more delicate. Recall Definition \ref{deftorusmorphism}. We set

\begin{definition}
	\label{defafftoricmorphism}
	A toric morphism $\mathscr L$ from $Q_{k,d,\Gamma}$ onto $Q_{k',d',\Gamma'}$ is a stack morphism which restricts to a torus morphism from $\nctorus{d}{\Gamma}$ to $\nctorus{d'}{\Gamma'}$.
\end{definition}

Hence we may associate to $\mathscr L$ a linear map $L$ from $\R^d$ to $\R^{d'}$ sending $\Gamma$ onto $\Gamma'$.

Let us show now how to construct such morphisms from a Quantum Fan morphism between the fan constituted by the standard cone $C_{k,d}$ in $\Gamma$ with generators $(e_1,\hdots,e_k)$ of $1$-cones and the fan constituted by the standard cone $C_{k',d'}$ in $\Gamma'$ with generators $(e'_1,\hdots,e'_{k'})$ of $1$-cones. So, according to Definition \ref{QFmorphismdef}, we assume that the linear map $L :\R^d\to \R^{d'}$ has the additional property of sending each $e_i$ for $i$ between $1$ and $k$ onto a $\N$-linear combination of $(e'_1,\hdots,e'_{k'})$.

For $s\leq d$ a positive integer, define
\begin{equation}
\label{expk}
z\in \C^d\mapsto E_s(z):=(e^{2i\pi z_1},\hdots,e^{2i\pi z_s},z_{s+1},\hdots, z_d)\in \Torus^{s}\times\C^{d-s}
\end{equation}
and
\begin{equation}
\label{expbark}
z\in \C^d\mapsto \bar E_s(z):=(z_{1},\hdots, z_s,e^{2i\pi z_{s+1}},\hdots,e^{2i\pi z_d})\in \C^{s}\times\Torus^{d-s}
\end{equation}
We note that $E_s\circ \bar E_s=\bar E_s\circ E_s=E$. In the same way, we define $E'$, $E'_{s'}$ and $\bar E'_{s'}$ for $s'\leq d'$.

We first consider the diagram
\begin{equation}
\label{Lextension}
\begin{tikzcd}
\C^d=\C^k\times \C^{d-k} \arrow[r,"L"]\arrow[d,"E_k"']&\C^{d'}\arrow[d,"E'_{k'}"]\\
\Torus^{k}\times\C^{d-k}\arrow[r,"\overline{{L}}"]\arrow[d,hook]&\Torus^{k'}\times\C^{d'-k'}\arrow[d,hook]\\
\C^d\arrow[r,"\overline{{L}}"]\arrow[d,"\bar E_k"']&\C^{d'}\arrow[d,"\bar E'_{k'}"]\\
\C^k\times\Torus^{d-k}\arrow[r,dashed]&\C^{k'}\times\Torus^{d'-k'}
\end{tikzcd}
\end{equation}
In \eqref{Lextension}, the map $L$ descends to $\Torus^{k}\times\C^{d-k}$ and then extends to $\C^d$ because it sends the first $k$ vectors of the canonical basis of $\R^d$ onto a $\N$-linear combination of $(v'_1,\hdots, v'_{k'})$. However, {\it there is no reason for it to descend to $\C^k\times\Torus^{d-k}$}. The dash arrow means that, in general, there is no well-defined arrow at the bottom. This important fact is a source of trouble to turn $L$ into a morphism of the standard Quantum Toric Variety.

Let $(\tilde T,m)$ be an object of this stack above $T$. It therefore satisfies \eqref{affQTVobject}. Consider the fiber product
\begin{equation}
\label{fp}
\hat T:=\fg=\{(\tilde t,z)\mid m(\tilde t)=\bar E_k(z)\}
\end{equation}
By definition, we have a cartesian diagram
\begin{equation}
\label{cartesian}
\begin{tikzcd}
\hat T\arrow[d,"pr_1"']\arrow[r,"pr_2"]\arrow[dr,phantom,"\scriptstyle{\square}",very near start, shift right=0.5ex]&\C^d\arrow[d,"\bar E_k"]\\
\tilde T\arrow[r,"m"]&\C^k\times\Torus^{d-k}
\end{tikzcd}
\end{equation}
where $pr_i$ means projection onto the $i$-th factor.
Observe that $E_k(\Gamma)$ acts on $\hat T$ as follows. Set
\begin{equation}
\label{Ekaction}
E_k(\gamma)\cdot z=(E(\gamma_1)z_1,\hdots,E(\gamma_k)z_k,z_{k+1}+\gamma_{k+1},\hdots,z_d+\gamma_d)
\end{equation}
then
\begin{equation}
\label{fpaction}
E_k(\gamma)\cdot (\tilde t,z)=(\bar E_kE_k(\gamma)\cdot \tilde t=E(\gamma)\cdot \tilde t,E_k(\gamma)\cdot z)
\end{equation}
is well defined since
\begin{equation}
\begin{aligned}
m(E(\gamma)\cdot \tilde t)=&E(\gamma)\cdot m(\tilde t)\\
=&\bar E_kE_k(\gamma)\cdot\bar E_k(z)\\
=&\bar E_k(E_k(\gamma)\cdot z)
\end{aligned}
\end{equation}
and we have
\begin{lemma}
	\label{fpcovering}
	Let $p$ be the composition $$\begin{tikzcd}
	\hat T=\fg\arrow[r,"pr_1"] &\tilde T\arrow[r] &T
	\end{tikzcd}
	$$
	Then $\begin{tikzcd}
	\hat T=\fg\arrow[r,"p"] &T
	\end{tikzcd}
	$  is a $E_k(\Gamma)$-cover.
\end{lemma}

\begin{proof}
We first prove that action {\rm \eqref{fpaction}} is free and proper. If $\gamma\in E_k(\Gamma)$ fixes $(\tilde t,z)$, then we have
\begin{equation*}
\bar E_k(\gamma)\cdot\tilde t=\tilde t\qquad\text{ and }\qquad \gamma\cdot z=z
\end{equation*}
Now the first equation says that $\bar E_k(\gamma)\in E(\gamma)$ acts as a deck transformation of $\tilde T\to T$ fixing point $\tilde t$, hence is equal to identity. So, in coordinates, 
\begin{equation*}
\gamma=(1,p)\in(\C^*)^k\times \Z^{d-k}
\end{equation*}
and the second equation gives
\begin{equation*}
\gamma\cdot z=(z_1,\hdots,z_k,z_{k+1}+p_1,\hdots,z_d+p_{d-k})=z
\end{equation*}
so all $p_i$ are zero and $\gamma$ is the identity of $E(\Gamma)$, proving freeness of action \eqref{fpaction}.

Moreover, if $K$ is a compact of $\hat T$, then $pr_1(K)$, resp. $pr_2(K)$ are compact of $\tilde T$, resp. $\C^d$. Let $\gamma$ such that $\gamma\cdot K\cap K$ is not empty. Since the action of $E(\Gamma)$ on $\tilde T$ is proper, then $\bar E_k(\gamma)\cdot(pr_1(K))$ meets $pr_1(K)$ only for a finite number of $\bar E_k(\gamma)$. In other words, the first $k$-coordinates of $\gamma$ belong to a finite set. Since the action of $E_k(\Gamma)$ on $\C^d$ is by translations on the last $(d-k)$-coordinates, $\gamma\cdot pr_2(K)$ meets $pr_2(K)$ only for a set of $\gamma$ whose last $(d-k)$-coordinates belong to a finite set. Putting altogether, this proves that $\gamma$ lives in a finite set, and action \eqref{fpaction} is proper.

Finally, it is easy to check that the map
\begin{equation*}
\begin{tikzcd}
(\tilde t,z)\in\hat T\arrow[r] &t\in T,
\end{tikzcd}
\end{equation*} 
where $t$ is the image of $\tilde t$ through the cover projection $\tilde T\to T$, is invariant through action \eqref{fpaction}, and, using $m(\tilde t)=\bar E_k(z)$, we conclude as well that $(\tilde t,z) \to t$ has fiber above $t$ equal to the orbit of $(\tilde t,z)$. Thus it descends as an isomorphism between the quotient of $\hat T$ by \eqref{fpaction} and $T$. 
\end{proof}
Now define $\tilde T'$ as the quotient of $\hat T\times E'(\Gamma')$ by the free $E_k(\Gamma)$-action
\begin{equation}
E_k\gamma\cdot (\tilde t,z,E'\gamma')=(E\gamma\cdot \tilde t,E_k\gamma\cdot z,\bar E'_{k'}\overline{{L}}(E_k\gamma)^{-1}\cdot E'\gamma')
\end{equation}
We will use the notations of bundles associated to a principal bundle. That is, given a principal bundle $E\to B$ with fiber and structural group an abelian group $G$, a manifold $F$ and a morphism $\rho$ from $G$ to the automorphism group of $F$, one can construct the associate cover $E\times_{\rho}F\to B$ with fiber $F$ as
\begin{equation}
\label{asscov}
E\times_{\rho}F:=(E\times F)/G
\end{equation}
where $G$ acts by deck transformations on $E$ and via $\rho^{-1}$ on $F$\footnote{Here since our group is abelian, there is no need to distinguish between left and right actions of $G$.}. Here this is exactly what we are doing and $\tilde T'$ is nothing else than
\begin{equation}
(\fg)\times_{\bar E'_{k'}\overline{{L}}}E'(\Gamma')
\end{equation} 
that is the $E'(\Gamma')$-cover associated to the $E_k(\Gamma)$-cover $\fg$ through the morphism  $\bar E'_{k'}\overline{{L}}$.

Finally, set
\begin{equation}
\label{fp2ndaction}
(\tilde t,z,E'(\gamma'))\longmapsto m'(\tilde t,z,E'\gamma'):=\bar E'_{k'}\overline{{L}}(z)\cdot E'(\gamma')\in \C^k\times \Torus^{d-k}
\end{equation}
Straightforward computations show that $m'$ descends as a $E'(\Gamma')$-equivariant map from $\tilde T'$ to $\C^k\times \Torus^{d-k}$. 

Hence 
\begin{equation}
\label{affQTVobjectprime}
\begin{tikzcd}
\tilde T' \arrow[r,"m'"]\arrow[d] &\C^k\times \Torus^{d-k}\\
T & &
\end{tikzcd}
\end{equation}
is an object of $Q_{k,d,\Gamma}$.

Observe that all these spaces fit into the following commutative diagrams
\begin{lemma}
	\label{assoclemma}
	We have
	\begin{equation*}
\begin{tikzcd}
\hat T\arrow[d]\arrow[dr,phantom,"\scriptstyle{\square}", very near start]\arrow[r]&\C^d\arrow[r,"\overline{{L}}"]\arrow[d,"\bar E_k"']&\C^{d'}\arrow[d,"\bar E'_{k'}"]&\hat T\times_{\overline{{L}}}E'_{k'}(\Gamma')\arrow[l, "B"']\arrow[dl,phantom,"\kern25pt \scriptstyle{\square}", very near start]\arrow[d,"A"]\\
\tilde T\arrow[r,"m"]\arrow[dr, bend right=20]&\C^k\times\Torus^{d-k}\arrow[r,dashed]&\C^{k'}\times\Torus^{d'-k'}&\tilde T'\arrow[l, "m'"']\arrow[lld, bend left=10]\\
&T &&\\
\end{tikzcd}
\end{equation*}
where
\begin{equation}
\label{A}
A(\tilde t,z,E'_{k'}\gamma')=(\tilde t,z,E'\gamma')
\end{equation}
and
\begin{equation}
\label{B}
B(\tilde t,z,E'_{k'}\gamma')=\bar L(z)\cdot E'_{k'}\gamma'
\end{equation}
\end{lemma}
\begin{proof}
	Observe that \eqref{A} and \eqref{B} define $A$ and $B$ as functions from $\hat T\times E'_{k'}(\Gamma')$, but, since
	\begin{equation*}
	\begin{aligned}
	A(g\cdot (\tilde t,z,E'_{k'}\gamma'))=&A(\bar E_k(g)\cdot \tilde t,g\cdot z,(\overline{{L}}g^{-1})\cdot E'_{k'}\gamma')\\
	=&(\bar E_k(g)\cdot \tilde t,g\cdot z,\bar E'_{k'}\overline{{L}}g^{-1}\cdot E\gamma')\\
	=&g\cdot (\tilde t,z,E\gamma')
	\end{aligned}
	\end{equation*}
and
\begin{equation*}
\begin{aligned}
B(g\cdot (\tilde t,z,E'_{k'}\gamma'))=&\overline{{L}}(g\cdot z)\cdot\overline{{L}}g^{-1}\cdot E'_{k'}\gamma'\\
=&\overline{{L}}(g)\cdot\overline{{L}}(z)\cdot(\overline{{L}}(g))^{-1}E'_{k'}\gamma'\\
=&B(\tilde t,z,E'_{k'}\gamma')
\end{aligned}
\end{equation*}
they descend as function from $\hat T\times_{\overline{{L}}}E'_{k'}(\Gamma')$ onto $\tilde T'$ (for $A$) and onto $\C^{d'}$ (for $B$).

Then, we have only to check that the right-hand diagram is commutative and cartesian. But
\begin{equation*}
\bar E'_{k'}B(\tilde t,z,E'_{k'}(\gamma'))=\bar E'_{k'}\overline{{L}}(z) E'(\gamma')=m'(\tilde t,z,E'(\gamma'))
\end{equation*}
hence it is commutative. Now, the map
\begin{equation*}
(\tilde t,z,E'_{k'}\gamma')\in \hat T\times E'_{k'}(\Gamma')\longmapsto(\tilde t, z, E'\gamma',\overline{{L}}(z)\cdot E'_{k'}\gamma')\in \hat T\times E'(\Gamma')\times \C^{d'}
\end{equation*}
descends as an isomorphism $I$ between $\hat T\times_{\overline{{L}}}E'_{k'}(\Gamma')$ and $\tilde T'\,_{m'}\kern-2pt\times_{\bar E'_{k'}}\C^{d'}$ which makes the following diagram commutative
\begin{equation}
\label{superfp}
\begin{tikzcd}
&&\tilde T'\,_{m'}\kern-2pt\times_{\bar E'_{k'}}\C^{d'}\arrow[dll,"pr_2", bend right=10]\arrow[ddl,"pr_1",bend left=10]\\
\C^{d'}\arrow[d,"\bar E'_{k'}"]&T\times_{\overline{{L}}}E'_{k'}(\Gamma')\arrow[l, "B"]\arrow[d,"A"]\arrow[ur,"I"]&\\
\C^{k'}\times \Torus^{d'-k'}&\tilde T'\arrow[l, "m'"]&\\
\end{tikzcd}
\end{equation}
and the desired diagram is cartesian.
\end{proof}
 Let $\mathscr{L}$ be the map from $Q_{k,d,\Gamma}$ to $Q_{k',d',\Gamma'}$ sending $(\tilde T,m)$ to $(\tilde T',m')$ and \eqref{affQTVmorphism}
 to
\begin{equation}
\label{affQTVmorphismprime}
\begin{tikzcd}
&            &\C^k\times \Torus^{d-k}\\
\tilde T' \arrow[r,"f'"']\arrow[rru,bend left=20, "m'"]\arrow[d]          &\tilde S'  \arrow[d]\arrow[ru, bend right=10, "n'"']      &                       \\
T        \arrow[r]         &S       &          
\end{tikzcd}
\end{equation}
with
\begin{equation}
\label{fprime}
f'([\tilde t,z,E'(\gamma')])=[f(\tilde t),z,E'(\gamma')]
\end{equation}
The map $\mathscr{L}$ is a stack morphism from $Q_{k,d,\Gamma}$ to $Q_{k',d',\Gamma'}$. \vspace{3pt}\\
Moreover, we note the following property.
\begin{lemma}
	\label{stackmorphismlemma}
	Given 
	\begin{equation}
	\label{composition}
	\begin{tikzcd}
	C_{k,d}\subset\C^d\arrow[r,"L"]&C_{k',d'}\subset\C^{d'}\arrow[r,"L'"]&C_{k'',d''}\subset\C^{d''}
	\end{tikzcd}
	\end{equation}
	and setting $L''=L'\circ L$, we have
	\begin{equation}
	\label{assoc}
	\mathscr{L''}=\mathscr{L'}\circ{\mathscr{L}}.
	\end{equation}
\end{lemma}
	
	\begin{proof}
		We just prove property \eqref{assoc} on the objects. The rest of the properties can be verified easily. Starting from \eqref{composition} and from
		 an object $(\tilde T,m)$ over $T$, we form $\tilde T'=\mathscr{L}(T)$ and $\tilde T''=\mathscr{L'}(T')$ and we have to compare it with $\mathscr{L''}(T)$.
		 
		 Now, it is a standard fact that 
		 \begin{equation*}
		 \left (\fg\times_{{\overline{L}}}E'_{k'}(\Gamma')\right )\times_{{\overline{L'}}}E''_{k''}(\Gamma'')
		 \end{equation*} 
		 is isomorphic to
		 \begin{equation*}
		 \left (\fg\right )\times_{{\overline{L'}}{\overline{L}}}E''_{k''}(\Gamma'')
		 \end{equation*}
		 which implies 
		 \begin{equation*}
		 \begin{aligned}
		 \tilde T''=&\left (\tilde T'\,_{m'}\kern-2pt\times_{\bar E'_{k'}}\C^{d'}\right )\times_{\bar E''_{k''}{\overline{L'}}} E''(\Gamma'')\\
		 	=&\left (((\fg)\times_{\bar E'_{k'}{\overline{L}}}E'(\Gamma'))\,_{m'}\kern-2pt\times_{\bar E'_{k'}}\C^{d'}\right )\times_{\bar E''_{k''}{\overline{L'}}}E''(\Gamma'')\\
		 	=&\left ((\fg)\times_{{\overline{L}}}E'_{k'}(\Gamma')\right )\times_{\bar E''_{k''}{\overline{L'}}}E''(\Gamma'')
		 \end{aligned}
		 \end{equation*}
		 but, using Lemma \ref{assoclemma} and \eqref{superfp}, we obtain
		 \begin{equation*}
		 \tilde T''=(\fg)\times_{\bar E''_{k''}{\overline{L'}}{\overline{L}}}E''(\Gamma'')
		 \end{equation*}
		 and we are done.
	\end{proof}

Such morphisms are indeed toric morphisms. And we claim that we obtain in this way all affine toric morphisms, that is

\begin{theorem}
	\label{theoremtoricmorphismcondition}\ \\
	\vspace{-15pt}
	\begin{enumerate}
		\item [\rm i)] Let $\mathscr L$ be a stack morphism as above. Then $\mathscr L$ is a toric morphism.
		\item[\rm ii)] Let $\mathscr L$ be a torus morphism from $\nctorus{d}{\Gamma}$ to $\nctorus{d'}{\Gamma'}$ and let $L$ be the induced linear mapping $L$ from $\R^d$ to $\R^{d'}$ sending $\Gamma$ onto $\Gamma'$. 
		Then, $\mathscr L$ extends as a toric morphism from $Q_{k,d,\Gamma}$ onto $Q_{k',d',\Gamma'}$ if and only if $L$ is a morphism of Quantum Fans from the cone $(C_{k,d},e_1,\hdots,e_k)$ onto the cone $(C_{k',d'},e'_1,\hdots e'_{k'})$.
	\end{enumerate}
\end{theorem} 

\begin{proof}
	To prove i), we just have to check that the restriction of $\mathscr L$ to the Quantum torus $\nctorus{d}{\Gamma}$ coincides with the torus morphism induced by $L$. Consider the object
	\begin{equation}
	\begin{tikzcd}
	\tilde T\arrow[r,"m"] \arrow[d]&\Torus^d\arrow[r,hook]&\C^k\times\Torus^{d-k}\\
	T&&
	\end{tikzcd}
	\end{equation} 
	Then $(\tilde S,m'):=(\mathscr L(\tilde T),\mathscr{L}(m))$ fits into the diagram of Lemma \ref{assoclemma}
		\begin{equation}
	\begin{tikzcd}[column sep=large]
	\hat T\arrow[d]\arrow[dr,phantom,"\kern -5pt\scriptstyle{\square}", very near start]\arrow[r]&\Torus^k\times\C^{d-k}\arrow[r,"\overline{{L}}"]\arrow[d,"\bar E_k"']&\Torus^{k'}\times\C^{d'-k'}\arrow[d,"\bar E'_{k'}"]&\hat S\arrow[l]\arrow[dl,phantom,"\kern 5pt \scriptstyle{\square}", very near start]\arrow[d]\\
	\tilde T\arrow[r,"m"]\arrow[dr, bend right=20]&\Torus^{d}\arrow[r,dashed]&\Torus^{d'}&\tilde S\arrow[l, "m'"']\arrow[lld, bend left=10]\\
	&T&&\\
	\end{tikzcd}
	\end{equation}
	with
	\begin{equation}
	\left\{
	\begin{aligned}
	&S=\hat T\times_{\bar E'_{k'}{\overline{L}}}E'(\Gamma')\\
	&\hat S=S_{m'}\kern-5pt\times_{\bar E'_{k'}}\Torus^{k'}\times\C^{d'-k'}\\
	&m'(\tilde t,z,E'(\gamma'))=\bar E'_{k'}\overline{L}(z)E'(\gamma')
	\end{aligned}
	\right .
	\end{equation}
	and, using \eqref{Lextension}, the previous diagram may be augmented as
	\begin{equation}
	\label{augCD}
	\begin{tikzcd}[column sep=large]
	\bar T\arrow[rrr,bend left=25, "\restriction{\mathscr{L}}{\nctorus{d}{\Gamma}}"]\arrow[dr,phantom,"\kern -5pt\scriptstyle{\square}", very near start]\arrow[d]\arrow[r]&\C^d\arrow[r,"L"]\arrow[d,"E_k"']&\C^{d'}\arrow[d,"E'_{k'}"]&\bar S \arrow[l,"a"']\arrow[dl,phantom,"\kern 5pt \scriptstyle{\square}", very near start]\arrow[d,"b"]\\
	\hat T\arrow[d]\arrow[dr,phantom,"\kern -5pt\scriptstyle{\square}", very near start]\arrow[r]&\Torus^k\times\C^{d-k}\arrow[r,"\overline{{L}}"]\arrow[d,"\bar E_k"']&\Torus^{k'}\times\C^{d'-k'}\arrow[d,"\bar E'_{k'}"]&\hat S\arrow[l]\arrow[dl,phantom,"\kern 5pt \scriptstyle{\square}", very near start]\arrow[d]\\
	\tilde T\arrow[r,"m"]\arrow[dr, bend right=20]&\Torus^{d}\arrow[r,dashed]&\Torus^{d'}&S\arrow[l, "m'"']\arrow[lld, bend left=10]\\
	&T&&\\
	\end{tikzcd}
	\end{equation}
	\vspace{-15pt}\\for	
	\begin{equation}
	\left\{
	\begin{aligned}
	&\bar T=\tilde T\,_m\kern-2pt\times_{E}\C^d\quad\text{and }\\
	&\bar S=\bar T\times_{L}\Gamma'
	\end{aligned}
	\right .
	\end{equation}
	and finally
	\begin{equation}
	\left\{
	\begin{aligned}
	&a[\tilde t,z,\gamma']=Lz+\gamma'\quad\text{and }\\
	&b[\tilde t,z,\gamma']=([\tilde t,E_k(z),E'(\gamma')],\overline{L}E_k(z)E'_{k'}(\gamma'))
	\end{aligned}
	\right .
	\end{equation}
	 Comparing the first line with the multiplicative form of the torus morphism associated to $L$  yields the result. \vspace{3pt}\\
	\indent Let us prove ii). Let $\mathscr L$ be a torus morphism induced by a linear map $L$. If $L$ is a morphism of Quantum Fans from the fan $(C_{k,d},e_1,\hdots,e_k)$ onto the fan $(C_{k',d'},e'_1,\hdots e'_{k'})$ then it sends the first $k$ vectors of the canonical basis of $\R^d$ onto a $\N$-linear combination of $(v'_1,\hdots, v'_{k'})$. Hence diagram \eqref{Lextension} is well defined, we may apply the above construction so by Lemma \ref{stackmorphismlemma} and point i), $\mathscr L$ extends as a toric morphism.
	
	Conversely, assume $\mathscr L$ extends. Looking at the inclusion $\nctorus{d}{\Gamma}\subset Q_{k,d,\Gamma}$ on the atlas \eqref{QstaffineTV}, we obtain the diagram
	\begin{equation*}
	\begin{tikzcd}
	&            &\Torus^{d}\arrow[r,hook]&\C^k\times\Torus^{d-k}\\
	\Torus^d\times E(\Gamma) \arrow[r,hook]\arrow[rru,bend left=10, "m"]\arrow[d]          &\C^k\times\Torus^{d-k}\times E(\Gamma)\arrow[d]\arrow[rru,bend right=10,"m"']      &  &                     \\
	\Torus^d        \arrow[r,hook]         &\C^k\times\Torus^{d-k}       &   &       
	\end{tikzcd}
	\end{equation*}
	and its image through $\mathscr L$
	\begin{equation*}
	\begin{tikzcd}
	&            &\Torus^{d}\arrow[r,hook]&\C^k\times\Torus^{d-k}\\
	\tilde S \arrow[r,hook]\arrow[rru,bend left=10, start anchor=30, "m'"]\arrow[d]          &\mathscr L(\C^k\times\Torus^{d-k}\times E(\Gamma),m')\arrow[d]\arrow[rru,bend right=10,"m'"']      &  &                     \\
	\Torus^{d}        \arrow[r,hook]         &\C^{k}\times\Torus^{d-k}       &   &       
	\end{tikzcd}
	\end{equation*}
	where $\tilde S$ and $m'$ are defined as above. But this means that, if a sequence of points $(E(z_n), Id)$ in $\Torus^d\times E(\Gamma)$ converging to some limit in $\C^k\times\Torus^{d-k}\times E(\Gamma)$, then $E'L(z_n)$ must converge in $\C^{k'}\times \Torus^{d'-k'}$ (use the augmented diagram \eqref{augCD}). So $\bar L$ is well defined and we are done.  
	\end{proof}

From \eqref{QstaffineTV}, Definition \ref{defafftoricmorphism} and Lemma \ref{stackmorphismlemma}, we immediatly obtain that standard affine quantum toric varieties and their morphisms form a category. Moreover, we may sum up this Section in the following

\begin{corollary}
	\label{affisocat}
	Let $\mathscr Aff$ be the category whose objects are standard cones $C_{k,d}$ and morphisms are linear maps from $\C^d$ to $\C^{d'}$ sending $\Gamma$ in $\Gamma'$ and $C_{k,d}$ and $C'_{k',d'}$. Then $\mathscr Aff$ is equivalent to the category of standard affine quantum toric varieties.
\end{corollary}

\subsection{Simplicial Quantum Toric Varieties}
\label{Cgluings}
Let $(\Delta,v)$ be a simplicial standard Quantum Fan in $\Gamma$. Since $\Delta$ is simplicial, every cone is the cone over a simplex. For $I\subset \{1,\hdots, d\}$, we denote by $v_I$ the cone over the simplex with vertices $(v_i)_{i\in I}$. Denote by $\vert I\vert$ the cardinal of $I$ and set $I=\{i_1,\hdots,i_{\vert I\vert}\}$. 

For each maximal cone $v_I$, choose some invertible linear map
\begin{equation}
\label{AIg}
\begin{tikzcd}
\C^d\arrow[r,"A_I"]&\C^d
\end{tikzcd}
\end{equation}
which sends $v_I$ onto the standard $\vert I\vert$-cone $C_{\vert I\vert, d}$. Notice that it sends $\Gamma$ onto $A_I\Gamma$.

We associate to $v_I$ the affine quantum toric variety
\begin{equation}
\label{QaffineTV}
Q_I=Q_{\vert I\vert,d}=\left [
\C^{\vert I\vert} \times \Torus^{d-\vert I\vert}/E({A_I}\Gamma)
\right ]
\end{equation}
If $\vert I\vert =d$, observe that we may set
\begin{equation}
\label{AI}
A_I^{-1}=\left ( v_{i_1},\hdots,v_{i_d}\right )
\end{equation}
as a square $d\times d$ matrix. Otherwise, we put
\begin{equation}
\label{AIbis}
A_I^{-1}=\left ( v_{i_1},\hdots,v_{i_{\vert I\vert}}, e_{j_1},\hdots, e_{j_{d-\vert I\vert}}\right )
\end{equation}
where $(e_i)$ is the canonical basis of $\R^d$ and where the indices $j_k$ are chosen so that the previous $d\times d$ matrix is invertible.

We note the obvious fact
\begin{lemma}
	\label{indepAI}
	Let $A_I$ and $A'_I$ be two linear maps as in {\rm \eqref{AIg}}. Let $Q_I$ and $Q'_I$ be the associated stacks following {\rm \eqref{QaffineTV}}. Then the linear map $A'_I\circ A_I^{-1}$ induces an isomorphism between $Q_I$ and $Q'_I$.
\end{lemma}

\begin{proof}
	The linear map $A'_I\circ A_I^{-1}$ is an isomorphism of the category $\mathscr Aff$ of standard cones. Use Corollary \ref{affisocat}.
\end{proof}

The same type of description works for non-maximal cones. Starting with $J=i_1\cdots i_l$ labelling a non-maximal cone $v_J$, let $I\supset J$ such that $v_I$ is a maximal cone. Then we associate to $v_J$ the stack 
\begin{equation}
\label{QaffineTVnotmax}
Q_{I\supset J}=\left [
\C^{\vert J\vert }\times \Torus^{d-\vert J\vert}/E({A_I}\Gamma)
\right ]
\end{equation}

Following \eqref{affQTVinclusion} and \eqref{affQTVmorphisminclusion}, inclusion $J\subset I$ gives a stack inclusion $Q_{I\supset J}\subset Q_I$.

Besides, let $v_I$ and $v_{I'}$ be two maximal cones with non-empty intersection and let $J=I\cap I'$. Then the linear isomorphism $A_{IJ}:=A_JA_I^{-1}$ defines an isomorphism from $Q_{I\supset J}$ to $Q_{I'\supset J}$, which indicates how to glue the affine quantum toric varieties $Q_I$ and $Q_{I'}$ along $Q_J$. 

We finish with a collection of affine quantum toric varieties 
$\left (Q_I\right )_{I \text{ maximal}}$
and a collection of isomorphisms 
\begin{equation*}\left (\mathscr A_{IJ}\right )_{I,J \text{ maximal},\ I\cap J\not = \emptyset}
\end{equation*}
describing the gluings between these varieties. Because of Lemma \ref{stackmorphismlemma} and because the $(A_JA_I^{-1})$ form obviously a cocycle, the collection of isomorphisms is a descent datum. We use it to {\it define} the Quantum Toric Variety associated to the fan $(\Delta,v)$, cf. the example of $\mathbb P^1$ at the beginning of Section \ref{Qatlas}. We use the notation $\nctoric{\Delta}{\Gamma}{v}$ for this Quantum Toric Variety.

\begin{definition}
	\label{QToricObjects}
	Let $T\in\mathfrak{A}$. An {\it object of $\nctoric{\Delta}{\Gamma}{v}$ over $T$} is a covering $(T_I)_{I\in {\mathscr I_{max}}}$ of $T$ indexed over the set of maximal cones $\mathscr I_{max}$ together with an object 
	\begin{equation*}
	\begin{tikzcd}
	\tilde T_I \arrow[r,"m_I"]\arrow[d] &\C^{\vert I\vert}\times \Torus^{d-\vert I\vert}\\
	T_I & &
	\end{tikzcd}
	\end{equation*}
	of $Q_I$ for any $I\in\mathscr I_{max}$ satisfying the descent datum condition
	\begin{equation*}
	\mathscr{A}_{II'}\left ( 
	\begin{tikzcd}
	\tilde T_{I\supset J} \arrow[r,"m_I"]\arrow[d] &\C^{\vert J\vert}\times \Torus^{d-\vert J\vert}\\
	T_I & &
	\end{tikzcd}\kern-25pt
	\right )=
	\begin{tikzcd}
		\tilde T_{I'\supset J} \arrow[r,"m_{I'}"]\arrow[d] &\C^{\vert J\vert}\times \Torus^{d-\vert J\vert}\\
	T_I & &      
	\end{tikzcd}
	\end{equation*} 
	above any couple $(I,I')$ of maximal cones with non-empty intersection $J$, and setting $\tilde T_{I\supset J}:=m_I^{-1}(\C^{\vert J\vert}\times \Torus^{d-\vert J\vert})$ and $\tilde T_{I'\supset J}:=m_{I'}^{-1}(\C^{\vert J\vert}\times \Torus^{d-\vert J\vert})$.
\end{definition}  
	
\begin{definition}
	\label{QToricsMorphisms}
	Let $T\to S$ be a morphism from $T\in\mathfrak{A}$ and $S\in\mathfrak{A}$. A {\it morphism of $\nctoric{\Delta}{\Gamma}{v}$ over $T\to S$} is a collection of morphisms
	\begin{equation*}
	\begin{tikzcd}
	&            &\C^{\vert I\vert }\times \Torus^{d-\vert I\vert}\\
	\tilde T_I \arrow[r, "F_I"]\arrow[rru,bend left=20,start anchor =60, end anchor=real west,"m_I"]\arrow[d]          &\tilde S_I  \arrow[d]\arrow[ru,bend right=10,"n_I"']      &                       \\
	T_I        \arrow[r,"f"]         &S_I       &          
	\end{tikzcd}
	\end{equation*}
	satisfying the compatibility conditions: 
	\begin{equation*}
	\mathscr{A}_{II'}\left (
	\begin{tikzcd}
&            &\C^{\vert J\vert }\times \Torus^{d-\vert J\vert}\\
\tilde T_{I\supset J} \arrow[r, "F_I"]\arrow[rru,bend left=20,start anchor =60, end anchor=real west,"m_I"]\arrow[d]          &\tilde S_{I\supset J}  \arrow[d]\arrow[ru,bend right=10,"n_I"']      &                       \\
T_J        \arrow[r,"f"]         &S_J       &          
\end{tikzcd}
\right )
\end{equation*}
is equal to
\begin{equation*}
\begin{tikzcd}
&            &\C^{\vert J\vert }\times \Torus^{d-\vert J\vert}\\
\tilde T_{I'\supset J} \arrow[r, "F_{I'}"]\arrow[rru,bend left=20,start anchor =60, end anchor=real west,"m_{I'}"]\arrow[d]          &\tilde S_{I'\supset J}  \arrow[d]\arrow[ru,bend right=10,"n_{I'}"']      &                       \\
T_J        \arrow[r,"f"]         &S_J       &          
\end{tikzcd}
\end{equation*}
\end{definition}
 We may describe $\nctoric{\Delta}{\Gamma}{v}$ more concretely using gluings of quasifolds. 
Let $v_I$ and $v_{I'}$ be two maximal cones with non-empty intersection. Then, using the charts given in \ref{QaffineTV}, we glue them on the intersection using the map
\begin{equation}
\label{gluingmonomial}
[z]\in Q_{I\supset J}\longmapsto [z^{{A_{I'}}{A_I}^{-1}}]\in Q_{I'\supset J}
\end{equation}

In \eqref{gluingmonomial}, given $z\in\C^d$ and $A=(a_{ij})_{i,j=1}^d$ be a square $d\times d$ matrix, we use the notation 
\begin{equation}
\label{zA}
z^A\qquad\text{ for }\qquad (z_1^{a_{11}}\cdots z_d^{a_{d1}},\hdots, z_1^{a_{1d}}\cdots z_d^{a_{dd}})
\end{equation}

Observe that the exponents in \eqref{gluingmonomial} can be irrational; nevertheless, the corresponding monomials are well defined as morphisms from the stack $Q_{I\supset J}$ to the stack $Q_{I'\supset J}$. To be more precise, the gluing \eqref{gluingmonomial} is defined on $[\Torus^d/E({A_I}\Gamma)]\subset Q_I$ through the diagram
\begin{equation*}
\begin{tikzcd}
0\arrow[r]&A_I\Gamma\arrow[r]\arrow[d,"A_{I'}A_I^{-1}"]&\C^d\arrow[r]\arrow[d,"A_{I'}A_I^{-1}"]&\C^d/A_I\Gamma\arrow[r]\arrow[d,"A_{I'}A_I^{-1}"]&0\\
0\arrow[r]&A_{I'}\Gamma\arrow[r]&\C^d\arrow[r]&\C^d/A_{I'}\Gamma\arrow[r]&0
\end{tikzcd}
\end{equation*}
and the stack isomorphism between $[\C^d/A_I\Gamma]$ (respectively $[\C^d/A_{I'}\Gamma]$) and $[\Torus^d/E({A_I}\Gamma)]$ (respectively $[\Torus^d/E({A_{I'}}\Gamma)]$) given by $E$. 

Then it is extended to $Q_{I\supset I\cap J}$ by noting that the first $\vert I\cap J\vert$ rows of the matrix $A_J^{-1}A_I$ correspond to the first $\vert I\cap J\vert$ rows of the identity.
\medskip\\
Gluings \eqref{gluingmonomial} obviously satisfy the cocycle condition. Hence we obtain the alternative more geometric definition of $\nctoric{\Delta}{\Gamma}{v}$.
\begin{definition}
\label{QTVdef}
We call Quantum Toric Variety associated to the simplicial Quantum Fan $(\Delta, {\Gamma},v)$ the stack obtained by gluings all the $Q_I$ associated to maximal cones $V_I$ along their intersections through \eqref{gluingmonomial}. We denote it by $\nctoric{\Delta}{\Gamma}{v}$.
\end{definition}

\begin{remarks}
\label{rkordinaryTV}
If $(\Delta, v)$ is an ordinary fan (that is $\Gamma$ is discrete and each $v_i$ is the primitive vector of a $1$-cone, cf. Remark \ref{rkordinaryfans}), then this is the definition of the (usual) toric variety associated to it. If $(\Delta, v)$ is a stacky fan, then we obtain a toric orbifold.
\end{remarks}

\begin{remark}
	\label{rkAIdep}
	Definitions \ref{QToricObjects}, \ref{QToricsMorphisms} and \ref{QTVdef}
	depend a priori on the choice of matrices $A_I$ but we shall see in Section \ref{QTmorphismsection} that two different choices of matrices lead to isomorphic Quantum Toric Varieties. If we start with a {\it complete} fan, then \eqref{AI} gives a canonical choice. In the non-complete case, \eqref{AIbis} is not uniquely defined and cannot play this role.
\end{remark} 
We shall see in Section \ref{QGIT} that a Quantum Toric Variety can be realized as the leaf space (holonomy groupoid) of a non-K\"ahler foliated complex manifold.

\begin{example}
\label{quantumP1}
As an example, we deal with the case of a quantum projective line. Here, we use the Quantum Fan of Example \ref{P1QFans}, so we assume \eqref{quantumP1v12}, \eqref{quantumP1Delta} and \eqref{quantumP1Gamma}. Note that, if $a=0$, this is the fan of the classical $\mathbb P^1$.

We have two charts. Both are modelled onto $Q_I=Q_J=[\C/\exp (2i\pi a\Z)]$.
The gluing is the mapping
\begin{equation}
 \label{quantumP1gluing}
 [z]\in \Torus/\exp (2i\pi\Gamma)\longmapsto [z^{-1}]\in \Torus/\exp (2i\pi (-\Gamma))
\end{equation}
If $a\in\Z$, this is just the classical $\mathbb P^1$. If $a=p/q$ with $p$ and $q$ irreducible, then $\Gamma$ is the lattice $1/q\Z$, so $v_1$ and $v_2$ are not primitive and we obtain {a toric orbifold
	with $0$ and $\infty$ having stabilizer $\Z_q$}. Finally, if $a$ is irrational, then we obtain a stack with two points having a stabilizer equal to $\Z$. This is not an orbifold.
\end{example}

\begin{definition} 
\label{QProjectivespace}
Let $\Gamma\in\R^d$ such that $\Gamma$ contains the lattice of integer points. We call Quantum projective space associated to $\Gamma$ the Quantum toric variety associated to the simplicial fan in $\Delta$ generated by
\begin{equation}
v_1=e_1\qquad\hdots\qquad v_{d}=e_d\qquad v_{d+1}=-e_1-\hdots -e_d.
\end{equation}
\end{definition}

\begin{remarks}
Example \ref{quantumP1} describe all the quantum projective lines for $\Gamma$ as in \eqref{quantumP1Gamma}, that is generated by one or two generators. In all dimensions, observe that
\begin{enumerate}
\item[i)] There is exactly one quantum $\mathbb P^d$ associated to each $\Gamma$.
\item[ii)] If $\Gamma$ is the lattice of integer points, this is exactly the standard $\mathbb P^d$.
\item[iii)] If $\Gamma$ is discrete but is not the lattice of integer points, we obtain a weighted projective space.
\item[iv)] If $\Gamma$ is irrational, some points have infinite discrete stabilizers.
\end{enumerate}
\end{remarks}

Not all compact examples in dimension one are Quantum projective lines as shown by the following example.
  
\begin{example}
\label{QuantumP1bis}
Let us slightly complicate example \ref{quantumP1}. We still assume that $d$ is equal to $1$ and that the fan contains only $\{0\}$ and the two cones $\R^{\geq 0}$ and $\R^{\leq 0}$. But we take as generators for the $1$-cones 
\begin{equation}
 \label{NCP1v12}
 v_1=1\qquad v_2=a\text{ with }a<0
\end{equation}
and we set
\begin{equation}
 \label{NCP1Gamma}
 \Gamma=\Z+\Z a+\Z b\qquad\text{ with }b\in\R
\end{equation}
Let $\Pn^1(a,b)$ be the resulting Quantum toric variety. We see that we still have two charts but this time one is modelled onto $\C/\exp (2i\pi\Gamma)$ and the other onto $\C/\exp (2i\pi a^{-1}\Gamma)$.
The gluing is the mapping
\begin{equation}
 \label{NCP1gluing}
 [z]\in \C/\exp (2i\pi\Gamma)\longmapsto [z^{a^{-1}}]\in \C/\exp (2i\pi a^{-1}\Gamma)
\end{equation}
Let us describe certain of these Quantum Toric Varieties more precisely following the values of $a$ and $b$.
Assume firstly that $a$ and $b$ are rational. Let $q$ be the LCM of the denominators of $a$ and $b$ in their reduced fractional form.
It is easy to check that $\Gamma$ is generated by $1/q$. Hence the first chart is an orbifold chart $\C/\Z_{q}$, and the second one is an orbifold chart $\C/\Z_{q\vert a\vert}$. {If $q$ and $qa$ are coprime}, we obtain the (orbifold) weighted projective line with weights $q$ and $q\vert a\vert$. Observe that the GCD of the weights is not always $1$. Indeed, taking $-a$ equal to one and $b$ equal to $1/q$, we have two orbifold singularities of degree $q$ {and this is not a weighted projective line}.
Assume secondly that $a$ and $b$ are both irrational and are not rationally dependent. Then 
\begin{equation}
 \label{NCP1cplx3}
 \Gamma=\Z\oplus\Z a\oplus\Z b
\end{equation}
so the two charts are modelled onto 
\begin{equation}
 \label{NCP1cplx3charts}
 \C/\exp (2i\pi (\Z a\oplus \Z b))\qquad\text{ and }\qquad  \C/\exp (2i\pi (\Z a^{-1}\oplus \Z a^{-1}b))
\end{equation}
i.e. $\C$ divided by a $\Z^2$ acting freely outside zero but fixing $0$. This is not an orbifold.

The cases $a$ and $b$ are irrational but rationally dependent, or $a$ is rational and $b$ not or $b$ is rational and $a$ not, stay also outside the world of orbifolds but with charts modeled onto $\C$ divided by a $\Z$-action. We left the computations to the reader.
\end{example}

\begin{example}
	\label{exqdP2}
	We construct now the $\Gamma$-complete quantum deformations of $\mathbb P^2$ whose fan is given in Example \ref{exP2}.
	
	The three maximal cones $\langle 1,2\rangle$, $\langle 2,3\rangle$ and $\langle 3,1\rangle$ give rise to the three matrices
	\begin{equation}
	\label{AIP2}
	A_{12}=\begin{pmatrix} 1 &0\\ 0&1\end{pmatrix},\qquad
	A_{23}=\dfrac{1}{a}\begin{pmatrix}
	-b &a\\
	1 &0
	\end{pmatrix},\qquad
	A_{31}=\dfrac{1}{b}\begin{pmatrix}
	0 &1\\
	b &-a
	\end{pmatrix}
	\end{equation}
	and the corresponding three charts
\begin{equation}
\label{chartsQDP2}
Q_{12}=\left [\mathbb C^2\bigg /E\left (\mathbb Z\begin{pmatrix}
	a\\
	b
	\end{pmatrix} \right )\right ],\qquad
	Q_{23}=\left [\mathbb C^2\bigg /E\left (\mathbb Z\begin{pmatrix}
	-b/a\\
	1/a
	\end{pmatrix} \right )\right ]
\end{equation}
and
\begin{equation}
	Q_{31}=\left [\mathbb C^2\bigg /E\left (\mathbb Z\begin{pmatrix}
	1/b\\
	-a/b
	\end{pmatrix} \right )\right ]
\end{equation}
	Gluings are done using matrices \eqref{AIP2}. For example, the gluing between $Q_{23}$ and $Q_{12}$ is given by $A_{23}A_{12}^{-1}$, that is
	\begin{equation}
	[z,w]\in Q^*_{23}\longmapsto [z^{-b/a}w,z^{1/a}]\in Q^*_{12}
	\end{equation}
	where the $*$ means that the common coordinate, here that corresponding to the cone $2$, is not zero. The other two gluings are
	\begin{equation}
	\begin{aligned}
	&[z,w]\in Q^*_{31}\longmapsto [w^{1/b},zw^{-a/b}]\in Q_{12}^*\\
	\text{ and }&[z,w]\in Q^*_{31}\longmapsto [z^{1/b}w,z^{-a/b}]\in Q_{23}^*
	\end{aligned}
	\end{equation}
	So if $a=b=-1$, we have three charts modelled on $\mathbb C^2$ and we recover the classical $\mathbb P^2$.
	
	If $a$ and $b$ are rational, we have three charts modelled on the quotient of $\mathbb C^2$ by a finite group acting with a fixed point. We obtain a toric orbifold, indeed a weighted projective space. Straightforward but lengthy calculations show that, setting $\vert a\vert =p/q$ and $\vert b\vert=r/s$, then we obtain the weighted projective space $\mathbb P(\alpha,\beta,\gamma)$ with
	\begin{equation}
	\label{weights}
	\left\{
	\begin{aligned}
	&\alpha=\text{LCM}(q,s)\\
	&\beta=\text{LCM}(sp/\text{GCD}(sp,qr),p)\\
	&\gamma=\text{LCM}(qr/\text{GCD}(sp,qr),r)
	\end{aligned}
	\right .
	\end{equation}
	{where the weights \eqref{weights} have no common divisor}. This reduces, for $a$ and $b$ integers, to $\mathbb P(1,-a,-b)$. In particular, only $a=b=-1$ gives $\mathbb P^2$. It is also worth noticing that {the actions defining the charts \eqref{chartsQDP2} are always effective so we cannot obtain weighted projective spaces with non-coprime weights}.
		
	Finally, if at least $a$ or $b$ is irrational, we have three charts modelled on the quotient of $\mathbb C^2$ by $\mathbb Z$ acting with a fixed point. We obtain a quasifold and not an orbifold. 
\end{example}
\subsection{Toric morphisms of Simplicial Quantum Toric Varieties.}
\label{QTmorphismsection}
A toric morphism between two simplicial Quantum Toric Varieties $\nctoric{\Delta}{\Gamma}{v}$ and $\nctoric{\Delta'}{\Gamma'}{v'}$ is simply a collection of affine toric morphisms $\mathscr{L}_{II'}$ between the affine cones $Q_I$ and $Q'_{I'}$ that are compatible with the gluings. 
Compatibility means that $\mathscr{A'}_{I'J'}\circ \mathscr{L}_{II'}$ and $\mathscr{L}_{JJ'}\circ \mathscr{A}_{IJ}$ are equal on $Q_{J\supset{I}}$.
\vspace{5pt}\\
It follows immediatly from Lemma \ref{stackmorphismlemma} that this compatibility boils down to
\begin{equation*}
A'_{I'J'}L_{II'}=L_{JJ'}A_{IJ}
\end{equation*}
In other words the linear maps $A'_{I'}L_{II'}A_{I}$ glue together in a Quantum Fan morphism $L$.

Conversely, it is straightforward to check that a morphism of Quantum Fans induces a morphism of the corresponding Quantum Toric Varieties. From this, we obtain a proof of remark \ref{rkAIdep} as well as the following statement (compare with Corollary \ref{affisocat})

\begin{theorem}
	\label{prisoQTV}\textcolor{white}{gh}
	\begin{enumerate}[\rm i)]
		\item A stack morphism between Quantum Toric Varieties is a toric morphism if and only if its restriction to the Quantum tori is a torus morphism.
		\item Let $\mathscr Q$ be the category of simplicial Quantum Toric Varieties. Then $\mathscr Q$ is isomorphic to $Q$.
	\end{enumerate}
\end{theorem}

\section{The definition (atlas) of calibrated Quantum Toric Varieties.}
\label{decodef}
We treat now the calibrated case. We follow the same process as in Section \ref{Qatlas}, that is we first define standard affine calibrated toric varieties and their morphisms, and then general ones using descent.
 
 \subsection{Standard affine calibrated Quantum Toric Varieties and their morphisms}
 \label{NDstandardConedec}
 Let $\Gamma=\Z^d+\Gamma_0$ be in standard form and let $(C_{k,d},h)$ be a standard calibrated cone. Hence $h : \Z^n\to \Gamma$ is a surjective morphism satisfying $h(e_i)=e_i$ for $1\leq i\leq d$, the set of generators of $1$-cones is $\{1,\hdots,k\}$ and the set of virtual generators is $J=\{n-l+1,\hdots, n\}$ for some $l$ between $0$ and $n-d$.\\ 
 The Standard affine calibrated Quantum Toric Variety associated to $C_{k,d}$ and $h$ is the quotient stack
 \begin{equation}
 \label{decQstaffineTV}
 Q^{cal}_{k,d,h}=\left [
 \C^k\times \Torus^{d-k}/\Z^{n-d}
 \right ]
 \end{equation}
 where $\Z^{n-d}$ acts on $\C^k\times \Torus^{d-k}$ as follows
 \begin{equation}
 \label{Zaction}
 (p,z)\longmapsto E(\hbar(p))\cdot z
 \end{equation}
 Hence, an object over the affine toric variety $T$ is simply a $\Z^{n-d}$-cover $\tilde T^{cal}$ with an equivariant map $m^{cal}$ 
 \begin{equation}
 \label{affQTVobjectdec}
 \begin{tikzcd}
 \tilde T^{cal} \arrow[r,"m^{cal}"]\arrow[d] &\C^k\times \Torus^{d-k}\\
 T & &
 \end{tikzcd}
 \end{equation}
 and a morphism over $T\to S$ is given by
 \begin{equation}
 \label{affQTVmorphismdec}
 \begin{tikzcd}
 &            &\C^k\times \Torus^{d-k}\\
 \tilde T^{cal} \arrow[r,"g"]\arrow[rru,bend left=20,"m^{cal}"]\arrow[d]          &\tilde S^{cal}  \arrow[d]\arrow[ru,bend right=10,"n^{cal}"']      &                       \\
 T        \arrow[r]         &S       &          
 \end{tikzcd}
 \end{equation}
 Inclusion of standard cones $(C_{k,d},h)$ in $(C_{k',d},h)$ with $k\leq k'$ is obtained by adding the inclusion $\C^k\times \Torus^{d-k}\hookrightarrow$ $\C^{k'}\times \Torus^{d-k'}$ in diagrams \eqref{affQTVobjectdec} and \eqref{affQTVmorphismdec}, cf. \eqref{affQTVinclusion} and \eqref{affQTVmorphisminclusion}.
 
 We copy now Definition \ref{defafftoricmorphism}. 
 \begin{definition}
 	\label{defmorphismstdec}
 	A (toric) morphism $\mathscr L^{cal}$ between $Q^{cal}_{k,d,h}$ and $Q^{cal}_{k',d',h'}$ is a stack morphism which restricts to a calibrated torus morphism from $\nctorusgerbe{h}{J}$ to $\nctorusgerbe{h'}{J'}$.
 \end{definition}

Hence, we may associate to $\mathscr L^{cal}$ a couple of linear maps $(L,H)$ satisfying \eqref{CDLH}.

Given $(L,H)$ a morphism of calibrated Quantum Fans between $(C_{k,d},h)$ and $(C_{k',d'},h')$, we define a stack morphism $\mathscr{L}^{cal}$ between $Q^{cal}_{k,d,h}$ and $Q^{cal}_{k',d',h'}$ as follows. Observe that 
\begin{equation*}
L=\begin{pmatrix}
L_1 &0\\
0 &L_2
\end{pmatrix}
\qquad\text{ and }\qquad
H=\begin{pmatrix}
H_1 &0\\
0 &H_2
\end{pmatrix}
\end{equation*}
with $L_2$, resp. $H_2$, a $(d-k)\times (d'-k')$ matrix, resp. $(n-k)\times (n'-k')$ matrix. Starting with an object \eqref{affQTVobjectdec}, we first form the fibered product
\begin{equation}
\label{fpdec}
\begin{tikzcd}
\hat T^{cal}:=\fgdec=\{(\tilde t,z)\mid m^{cal}(\tilde t)=\bar E_k(z)\}
\end{tikzcd}
\end{equation}
so that we have a cartesian diagram
\begin{equation}
\label{cartesian2}
\begin{tikzcd}
\hat T^{cal}\arrow[d,"pr_1"']\arrow[r,"pr_2"]\arrow[dr,phantom,"\scriptstyle{\square}",very near start, shift right=0.5ex]&\C^d\arrow[d,"\bar E_k"]\\
\tilde T^{cal}\arrow[r,"m^{cal}"]&\C^k\times\Torus^{d-k}
\end{tikzcd}
\end{equation}
Here $\Z^{n-k}$ acts on $\hat T^{cal}$ via
\begin{equation}
\label{actionhatTdec}
p\cdot (\tilde t,z)=(pr_k(p)\cdot \tilde t,E_k(h(p))\cdot z)
\end{equation}
where $pr_k$ is the projection of $\Z^{n-k}=\Z^{d-k}\times\Z^{n-d}$ onto $\Z^{n-d}$ and where in the expression $h(p)$ we identified $\Z^{n-k}$ with $\{0\}\times \Z^{n-k}$ in $\Z^n$.

Then 
\begin{equation}
\label{Ldec}
\mathscr{L}^{cal}(\tilde T^{cal})=\hat T^{cal}\times_{pr_{k'}\circ H_2}\Z^{n'-d'}
\end{equation}
with associated equivariant map
\begin{equation}
\label{mtildeprimedec}
\begin{tikzcd}[column sep=huge]
(\tilde t,z,q)\arrow[r,mapsto,"\mathscr{L}^{cal}(m^{cal})"]&\bar E'_{k'}\overline{L}(z)\cdot \left (E'\hbar'(q)\right )
\end{tikzcd}
\end{equation}
As for the morphisms, $\mathscr{L}^{cal}$ sends a morphism \eqref{affQTVmorphismdec} onto
\begin{equation}
\label{affQTVmorphismdecL}
\begin{tikzcd}
&            &\C^k\times \Torus^{d-k}\\
\mathscr{L}^{cal}(\tilde T^{cal}) \arrow[r,"g'"]\arrow[rru,bend left=20,"\mathscr{L}^{cal}(m^{cal})"]\arrow[d]          &\mathscr{L}^{cal}(\tilde S^{cal})  \arrow[d]\arrow[ru,bend right=10,"\mathscr{L}^{cal}(n^{cal})"']      &                       \\
T        \arrow[r]         &S       &          
\end{tikzcd}
\end{equation}
with
\begin{equation*}
g'(\tilde t,z,q)=(g(\tilde t),z,q)
\end{equation*}

That $\mathscr{L}^{cal}$ is a stack morphism is proved by adapting the proof of Lemma \ref{stackmorphismlemma} to the calibrated case. We left the details to the reader. As in Section \ref{NDstandardCone}, we then characterize a calibrated toric morphism.

\begin{theorem}
	\label{theoremdectoricmorphismcondition}\ \\
	\vspace{-15pt}
	\begin{enumerate}[\rm i)]
		\item  Let $\mathscr{L}^{cal}$ be a stack morphism as above. Then $\mathscr{L}^{cal}$  is a toric morphism.
		\item Let $\mathscr{L}^{cal}$  be a calibrated torus morphism from $\nctorusgerbe{h}{J}$ to $\nctorus{\Gamma'}{h'}$ and let $(L,H)$ be the induced linear mappings. 
		Then $\mathscr{L}^{cal}$ extends as a toric morphism from $Q^{cal}_{k,d,h}$ to $Q^{cal}_{k',d',h'}$ if and only if $(L,H)$ is a morphism of calibrated Quantum Fans from the cone $(C_{k,d},h)$ onto the cone $(C_{k',d'},h')$.
	\end{enumerate}
\end{theorem} 

\begin{proof}
	This is similar to the proof of Theorem \ref{theoremtoricmorphismcondition}. We content ourselves with giving following two commutative diagrams which describe a calibrated torus morphism $\mathscr{L}^{cal}$ for the first one and a toric morphism for the second. Comparing them and following the same line of arguments as in the proof of Theorem \ref{theoremtoricmorphismcondition} yield the result
	\begin{equation*}
	\begin{tikzcd}[column sep=small]
	\Torus^k\times\C^{d-k}\times \Z^{n-k}\arrow[d,"p"']\arrow[dr,phantom,"\kern -25pt\scriptstyle{\square}", very near start]\arrow[r]&\Torus^k\times\C^{d-k}\arrow[r,"\overline{{L}}"]\arrow[d,"\bar E_k"']&\Torus^{k'}\times\C^{d'-k'}\arrow[d,"\bar E'_{k'}"]&\Torus^{k'}\times\C^{d'-k'}\times \Z^{n'-k'}\arrow[l]\arrow[dl,phantom,"\kern 5pt \scriptstyle{\square}", very near start]\arrow[d]\\
	\Torus^d\times \Z^{n-d}\arrow[r,"m"]\arrow[dr, bend right=20]&\Torus^{d}\arrow[r,dashed]&\Torus^{d'}&\Torus^{d'}\times \Z^{n-d}\arrow[l, "m'"']\arrow[lld, bend left=10]\\
	&\Torus^d &&\\
	\end{tikzcd}
	\end{equation*}
	and
	\begin{equation*}
	\begin{tikzcd}[column sep=small]
	\C^{d}\times \Z^{n-k}\arrow[d,"p"']\arrow[dr,phantom,"\kern -25pt\scriptstyle{\square}", very near start]\arrow[r]&\C^{d}\arrow[r,"\overline{{L}}"]\arrow[d,"\bar E_k"']&\C^{d'}\arrow[d,"\bar E'_{k'}"]&\C^{d'}\times \Z^{n'-k'}\arrow[l]\arrow[dl,phantom,"\kern 5pt \scriptstyle{\square}", very near start]\arrow[d]\\
	\C^k\times\Torus^{d-k}\times \Z^{n-d}\arrow[r,"m"]\arrow[dr, bend right=20]&\C^k\times\Torus^{d-k}\arrow[r,dashed]&\C^{k'}\times\Torus^{d'-k'}&\Torus^{d'}\times \Z^{n'-d'}\arrow[l, "m'"']\arrow[lld, bend left=10]\\
	& \C^{k'}\times\Torus^{d'-k'}&&\\
	\end{tikzcd}
	\end{equation*}
\end{proof}

Consider now the calibrated cone $(C_{k,n},Id)$ in $\R^n$ equipped with the lattice $\Z^n$ and with $J$ as set of virtual generators. This is exactly $(\Delta_h,Id)$ and the associated toric variety is simply $\C^k\times\Torus^{n-k}$. Note that the associated cone $C_{k,n}$ also corresponds to $\C^k\times\Torus^{n-k}$, that is there is no difference here between the calibrated and the non-calibrated case. We have

\begin{corollary}
	\label{corhdecmorphism}
	The calibration $h$ induces a toric morphism $\text{\slshape{h}}^{cal}$ from $\C^k\times\Torus^{n-k}$ onto $Q^{cal}_{k,d,h}$.
\end{corollary}

\begin{proof}
	Combine Theorem \ref{theoremdectoricmorphismcondition} and Lemma \ref{hdecmorphism}.
\end{proof}

We also infer from Theorem \ref{theoremdectoricmorphismcondition} the following corollaries.

\begin{corollary}
	\label{corLdecetL}
	Let $\mathscr{L}^{cal}$ be a morphism from $Q^{cal}_{k,d,h}$ to $Q^{cal}_{k',d',h'}$. Then it induces a morphism $\mathscr{L}$ from $Q_{k,d,\Gamma}$ to $Q_{k',d',\Gamma'}$
\end{corollary}

\begin{proof}
	Associate to $\mathscr{L}^{cal}$ the couple $(L,H)$. The linear mapping $L$ induces a a morphism $\mathscr{L}$ from $Q_{k,d,\Gamma}$ to $Q_{k',d',\Gamma'}$ by Theorem \ref{theoremtoricmorphismcondition}.
\end{proof}
and in particular, since the non-calibrated $C_{k,n}$ also corresponds to $\C^k\times\Torus^{n-k}$,
\begin{corollary}
	\label{corhmorphism}
		The calibration $h$ induces a toric morphism $\text{\slshape{h}}$ from $\C^k\times\Torus^{n-k}$ onto $Q_{k,d,\Gamma}$.	
\end{corollary}

Moreover $H$ induces a torus morphism $\mathscr H$ between $\C^k\times \Torus^{n-k}$ and $\C^{k'}\times\Torus^{n'-k'}$. This morphism can be seen either as a calibrated or a non-calibrated morphism. We have

\begin{lemma}
	\label{lemmaLhhH}
	The diagram of morphisms 
	\begin{equation}
	\label{CDLHstack}
	\begin{tikzcd}
	\C^k\times \Torus^{n-k} \arrow[r,"\mathscr H"]\arrow[d,"\text{\sl h}"']&\C^{k'}\times\Torus^{n'-k'}\arrow[d,"(\text{\sl h'})"]\\
	Q_{k,d,\Gamma}\arrow[r, "\mathscr{L}"']&Q_{k',d',\Gamma'}
	\end{tikzcd}
	\end{equation}
	is commutative
\end{lemma}

\begin{proof}
	By Lemma \ref{stackmorphismlemma}, the composition $\text{\slshape{h'}}\circ\mathscr{H}$, resp. $\mathscr{L}\circ \text{\slshape{h}}$, corresponds to the linear mapping $h'\circ H$, resp. $L\circ h$. But these mappings are equal by \eqref{CDLH}.
\end{proof}

The calibrated version of this lemma runs as follows. 

\begin{lemma}
	\label{lemmaLhhHdec}
	The diagram of morphisms 
	\begin{equation}
	\label{CDLHdecstack}
	\begin{tikzcd}
	\C^k\times \Torus^{n-k} \arrow[r,"\mathscr H"]\arrow[d,"\text{\sl (h)}^{cal}"']&\C^{k'}\times\Torus^{n'-k'}\arrow[d,"(\text{\sl h'})^{cal}"]\\
	Q^{cal}_{k,d,h}\arrow[r, "\mathscr{L}^{cal}"']&Q^{cal}_{k',d',h'}
	\end{tikzcd}
	\end{equation}
	is commutative
\end{lemma}

\begin{proof}
	The composition $\text{(\slshape{h'})}^{cal}\circ\mathscr{H}$, resp. $\mathscr{L}\circ \text{\slshape{h}}^{cal}$, restricts to a calibrated torus morphism, hence is a toric morphism. It corresponds to the linear mappings $(h'\circ H,H)$, resp. $(L\circ h,H)$. But these mappings are equal by \eqref{CDLH}.
\end{proof}
We finish this subsection with a characterization of toric morphisms.
\begin{theorem}
	\label{theoremmorphismaffdec}
	There is a $1:1$ correspondence between morphisms $\mathscr{L}^{cal}$ from $Q^{cal}_{k,d,h}$ to $Q^{cal}_{k',d',h'}$ and pairs $(\mathscr{L},\mathscr{H})$ such that
	\begin{enumerate}
		\item [\rm i)] $\mathscr L$ is a morphism of affine Quantum varieties from $Q_{k,d,\Gamma}$ to $Q_{k',d',\Gamma'}$.
		\item[\rm ii)] $\mathscr H$ is a morphism of classical toric varieties from $\C^k\times \Torus^{n-k}$ to $\C^{k'}\times\Torus^{n'-k'}$.
		\item [\rm iii)] The following diagram is commutative
		\begin{equation}
		\label{cddectoricmorphism}
		\begin{tikzcd}
		\C^k\times \Torus^{n-k}\arrow[d,"{\text{\sl h}}"']\arrow[r,"\mathscr H"]&\C^{k'}\times \Torus^{n'-k'}\arrow[d,"{\text{\sl h}}'"]\\
		Q_{k,d,\Gamma}\arrow[r,"\mathscr L"]&Q_{k',d',\Gamma'}
		\end{tikzcd}
		\end{equation}
		\item[\rm iv)] For $j\not \in J'$, the $j$-th component $\mathscr{H}_j$ is a monomial with integer coefficients in the variables $(w_i)_{i\not\in J}$.
		\item[\rm v)] There exists a map $s$ from $J$ to $J'$ such that, for $j\in J'$, the $j$-th component $\mathscr{H}_{j}$ satisfies 
		\begin{equation*}
		\mathscr{H}_{j}(w)=\left \{
		\begin{aligned}
		\prod_{i\in s^{-1}(j)}&w_i\qquad\text{ if }s^{-1}(j)\not =\emptyset\cr
		&\ 1 \qquad \text{ if }s^{-1}(j)=\emptyset
		\end{aligned}
		\right .
		\end{equation*}
	\end{enumerate}
\end{theorem} 

\begin{proof}
	To $\mathscr{L}^{cal}$, we associate $(L,H)$ and $(\mathscr{L},\mathscr{H})$ as above. Point iii) is Lemma \ref{lemmaLhhH} and points iv) and v) follow from points iv) and v) of Definition \ref{decQFmorphismdef}. Conversely, $(\mathscr{L},\mathscr{H})$ induces $(L,H)$ satisfying Definition \ref{decQFmorphismdef} and $\mathscr{L}^{cal}$ through Theorem \ref{theoremdectoricmorphismcondition}.
\end{proof}
Finally, we note the obvious following fact.
\begin{corollary}
	\label{corcaliso}
	The following points are equivalent.
	\begin{enumerate}[i)]
		\item The stack isomorphism $\mathscr{L}^{cal}$ is an toric isomorphism from $Q^{cal}_{k,d,h}$ to $Q^{cal}_{k',d',h'}$ (that is both $\mathscr{L}^{cal}$ and its inverse are toric morphisms).
		\item The associated Quantum Fan morphism $(L,H)$ is an isomorphism.
		\item The restriction of the stack isomorphism $\mathscr{L}^{cal}$ to the Quantum tori is an isomorphism.
	\end{enumerate}
\end{corollary}
From this Corollary, we define
\begin{definition}
	\label{defisomarked}
	A stack isomorphism $\mathscr{L}^{cal}$ is a {\it marked isomorphism} if the associated Quantum Fan morphism $(L,H)$ is a marked isomorphism, or, equivalently, if its restriction to the Quantum tori is a marked isomorphism.
\end{definition}
\subsection{Forgetting Calibrations}
\label{subsecf}
Given any calibrated standard affine  Quantum Toric Variety $Q^{cal}_{k,d,h}$ we can canonically associate to it its standard non calibrated form $Q_{k,d,\Gamma}$. More precisely, there is a natural forgetful functor $\text{\sl f}$ from $Q^{cal}_{k,d,h}$ onto $Q_{k,d,\Gamma}$. To write it down explicitly we make use of the kernel of $h$. Consider as in \eqref{completemorphismCD}
\begin{equation}
\label{Kernelexactseq}
\begin{tikzcd}
0 \arrow[r] &\Xi \arrow[r,"i"] &\Z^n\arrow[r,"h"] &\Gamma \arrow[r] &0
\end{tikzcd}
\end{equation}
 Observe that 
 \begin{equation*}
 h\circ i(k)=h(i_1(k),i_2(k))=i_1(k)+\hbar(i_2(k))=0
 \end{equation*}
 so 
 \begin{equation}
 \label{inZ}
 \hbar(i_2(k))\in\Z^d.
 \end{equation}
 
 Given an object $(\tilde T^{cal},m^{cal})$ over $T$, then $\Xi$ injects via $i$ in its deck transformation group, hence acts freely and properly on it. So we send it to $(\tilde T,m)$ where
 $\tilde T=\tilde T^{cal}/\Xi$
 and
 \begin{equation*}
 m(\tilde t)=m^{cal}(\tilde t^{cal})\qquad\text{ with }\qquad \tilde t^ {cal}\in\tilde T^{cal}\mapsto \tilde t\in\tilde T
 \end{equation*}
 which is well defined since, using \eqref{inZ}, 
 \begin{equation*}
 m^{cal}(i(k)\cdot\tilde t^{cal})=E(\hbar\circ i_2)\cdot m^{cal}(\tilde t^{cal})=m^{cal}(\tilde t^{cal})
 \end{equation*}
 So we have for the objects
 \begin{equation}
 \begin{tikzcd}[row sep=small]
  &  \C^k\times \Torus^{d-k}&&\\      
 \tilde T^{cal}\arrow[rr,"\text{\sl f}"]\arrow[rd, bend right=15, end anchor=real west]\arrow[ru,end anchor=west,"m^{cal}", bend left=20]&&\tilde T\kern15pt\arrow[ld, bend left=15, start anchor=210, end anchor=5] \arrow[lu,bend right=20,start anchor=130, end anchor=0, "m"']\\
 &T&
 \end{tikzcd}
 \end{equation}
 and for the morphisms,
 \begin{equation}
 \begin{tikzcd}
 &&&&\C^k\times\Torus^{d-k}\arrow[from=llllddd,bend right=25, start anchor=10, end anchor=205, "m" near start]\\
 \tilde T^{cal}\arrow[rrrru,bend left=15,start anchor=north east, end anchor=west,"m^{cal}"]\arrow[rrr, crossing over]\arrow[rd]\arrow[dd] &&&\tilde S^{cal}\arrow[ur,start anchor=north east, end anchor=south west,"n^{cal}"]\arrow[dd, crossing over]\arrow[dl]&\\
 &T\arrow[r]&S&&\\
 \tilde T\arrow[ur]\arrow[rrr]&&&\tilde S\arrow[lu, crossing over]\arrow[ruuu,"n"']
 \end{tikzcd}
 \end{equation}
  
Given a calibrated toric morphism $\mathscr{L}^{cal}$, it induces by Corollary \ref{corLdecetL} a toric morphism $\mathscr{L}$. They are related through $\text{\slshape f}$. Indeed,

 \begin{lemma}
	\label{lemmadecoublicom}
	We have a commutative diagram
	\begin{equation}
	\label{cddecoublicom}
	\begin{tikzcd}
	Q^{cal}_{k,d,h}\arrow[d,"\text{\sl f }"']\arrow[r,"\mathscr{L}^{cal}"]&Q^{cal}_{k',d',\Gamma'}\arrow[d,"\text{\sl f}"]\\
	Q_{k,d,\Gamma}\arrow[r,"\mathscr{L}"]&Q_{k',d',\Gamma'}
	\end{tikzcd}
	\end{equation}
\end{lemma}

\begin{proof}
	We just check it on the objects and left the verification on the morphisms to the reader. Start with an object \eqref{affQTVobjectdec} and set $\text{\sl f }(\tilde T^{cal})=\tilde T$. Then we have
	\begin{equation*}
	\text{\sl f }(\hat T^{cal})=\hat T
	\end{equation*}
	and
	\begin{equation*}
	\begin{aligned}
	\text{\sl f }(\mathscr{L}^{cal}(\tilde T^{cal}))=&\text{\sl f }(\hat T^{cal}\times_{s\circ pr_k}\Z^{n'-d'})\\
	=&\hat T\times_{\bar E'_{k'}{\overline{L}}}E'(\Gamma')\\
	=&\mathscr{L}(\tilde T)\\
	=&\mathscr{L}(\text{\sl f }(\tilde T^{cal})).
	\end{aligned}
	\end{equation*}
	and finally
	\begin{equation*}
	\begin{aligned}
	\text{\sl f }(\mathscr{L}^{cal}(m^{cal}))(\tilde t,z,E'\hbar'(q))=&\mathscr{L}^{cal}(m^{cal})(\tilde t^{cal},z,q)\\
	=&\bar E'_{k'}\overline{L}(z)\cdot \left (E'\hbar'(q)\right )\\
	=&\mathscr{L}(m)(\tilde t,z,E'\hbar'(q))\\
	=&\mathscr{L}(\text{\sl f }(m^{cal}))(\tilde t,z,E'\hbar'(q))
	\end{aligned}
	\end{equation*}
	proving the lemma.
\end{proof}

In particular, $\text{\slshape{h}}$ is the composition with $\text{\sl f}$ of $\text{\slshape h}^{cal}$. In other words, a calibrated affine quantum toric variety is indeed a triple
\begin{equation}
\label{cdtriple}
\begin{tikzcd}[row sep=tiny]
\C^k\times\Torus^{n-k}\arrow[dr,bend left=10,"\text{\sl h}"]\arrow[rr,bend left=20,"\text{\sl h}^{cal}"]&&Q^{cal}_{k,d,h}\arrow[ld,bend right=10,"\text{\sl f}"']\\
&Q_{k,d,\Gamma}&\\
\end{tikzcd}
\end{equation} 
We can give a much concrete description of $\text{\slshape{h}}$ and $\text{\slshape h}^{cal}$. The morphism $\text{\sl h}$ fits into the diagram, cf. \eqref{Lextension}
\begin{equation}
\label{hmorphismCD}
\begin{tikzcd}
\C^d\times\C^{n-d} \arrow[r,"h"]\arrow[d,"E_k"']&\C^{d}\arrow[d,"E_{k}"]\\
\Torus^{k}\times\C^{n-k}\arrow[r,"\overline{{h}}"]\arrow[d,hook]&\Torus^{k}\times\C^{d-k}\arrow[d,hook]\\
\C^d\times\C^{n-d}\arrow[d]\arrow[r,"\overline{{h}}"]&\C^d\arrow[d]\\
(\C^k\times\Torus^{d-k})\times \C^{n-d}\arrow[d,"\bar E_d"]\arrow[r,"\hat h"]&\C^k\times\Torus^{d-k}\arrow[d]\\
\C^k\times \Torus^{n-d}\arrow[r,"\text{\sl h}"]&Q_{k,d,\Gamma}
\end{tikzcd}
\end{equation}
Taking into account \eqref{h}, we may write down $\overline{h}$ as
\begin{equation}
\label{hbar}
\begin{aligned}
(z,w)\in\C^d\times\C^{n-d}\longmapsto \Big (E(&\hbar_1(w))\cdot z_1,\hdots,\\ E(\hbar_k(w))\cdot z_k,\
&z_{k+1}+\hbar_{k+1}(w), \hdots, z_d+\hbar_d(w)\Big )
\end{aligned}
\end{equation}
and $\hat h$ as
\begin{equation}
\label{hhat}
\begin{aligned}
(z,w)\in(\C^k\times\Torus^{d-k})\times\C^{n-d}\longmapsto E(\hbar(w))\cdot z
\end{aligned}
\end{equation}
Hence, we see that 
\begin{equation}
\label{decoupasdec}
\hat h(z,w+p)=E(\hbar(w+p))\cdot z=E(\hbar(p))\cdot \hat h(z,w) 
\end{equation}
for $p\in\Z^{n-d}$, so $\hat h$ descends as a morphism from $(\C^k\times\Torus^{d-k})\times\Torus^{n-d}$ to $Q_{k,d,h}^{cal}$. This is $\text{\sl h}^{cal}$.

We also note the following fact. We see that {\it $\overline{h}$, resp. $\hat h$ and $h$, defines an action of the group $(\C^{n-d},+)$ on $\C^d$, resp. $(\C^k\times\Torus^{d-k})\times\C^{n-d}$} so in \eqref{hmorphismCD} we interpret $\text{\sl h}$ as describing also an action of $(\C^{n-d},+)$. 
\vspace{3pt}

We may sum up the previous commutation properties in a single commutative diagram. Recall that $Q_{k,n}$ is simply $\C^k\times \Torus^{n-k}$.
\begin{lemma}
	\label{lemmaLHhetdec}
	The following diagram is commutative
	\begin{equation}
	\label{cdLHhetdec}
	\begin{tikzcd}[row sep=scriptsize, column sep=scriptsize]
& Q_{k,n} \arrow[dl, "Id"'] \arrow[rr,"\mathscr{H}"] \arrow[dd,"\text{\sl h}^{cal}"' near start] & & Q_{k',n'} \arrow[dl,"Id"'] \arrow[dd,"(\text{\sl h'})^{cal}"] \\
Q_{k,n} \arrow[rr, crossing over,"\mathscr{H}" near end] \arrow[dd,"{\text{\sl h}}"'] & & Q_{k',n'}\\
& Q^{cal}_{k,d,h} \arrow[dl,"{\text{\sl f}}"'] \arrow[rr,"\mathscr{L}^{cal}" near start] & & Q^{cal}_{k',d',h'} \arrow[dl,"{\text{\sl f}}"'] \\
Q_{k,d,\Gamma}\arrow[rr,"\mathscr{L}"] & & Q_{k',d',\Gamma'} \arrow[from=uu, crossing over,"{\text{\sl h'}}" near start]\\
\end{tikzcd}
	\end{equation}
\end{lemma}

\begin{proof}
	This is just a superposition of \eqref{CDLHdecstack}, \eqref{cddectoricmorphism}, and \eqref{cddecoublicom} applied to $\mathscr{L}^{cal}$, $\text{\sl h}^{cal}$ and $(\text{\sl h'})^{cal}$.
\end{proof}

Let now $\mathscr Aff^{cal}$ be the category whose objects are the standard calibrated cones $(C_{k,d},h)$ and whose morphisms are the morphisms of calibrated fans $(L,H)$ between two such $(C_{k,d},h)$. Similarly to Corollary \ref{affisocat}, we have
\begin{proposition}
	\label{propisocatdec}
	The category of calibrated affine Quantum Toric Varieties and their morphisms is equivalent to the category of $\mathscr Aff^{cal}$.
\end{proposition}

\begin{proof}
	The proof is straightforward using Corollary \ref{affisocat} and the isomorphism between the category of cones in a lattice and the category of affine toric varieties in the classical case. An affine Quantum calibrated toric variety $Q^{cal}_{k,d,h}$ is by definition associted to a cone $C_{k,d,\Gamma}$. Then, starting with a morphism $(\mathscr L,\mathscr H)$ of calibrated affine Quantum Toric Varieties, we associate to it $(L,H)$ with $L$ morphism of Quantum Fans sending $(C_{k,d},e_1,\hdots,e_k)$ into $(C_{k',d'},e'_1,\hdots,e'_{k'})$ and $H$ morphism of classical fans. Now it follows from Theorem \ref{theoremmorphismaffdec} that such a linear map $H$ must send each $e_i$ corresponding to a virtual $1$-cone $e'_{s(i)}$. Hence $(L,H)$ satisfies the requirements of Definition \ref{decQFmorphismdef}. Conversely, given $(L,H)$ a morphism of Quantum Fans, we build $\mathscr{L}$ and $\mathscr{H}$ satisfying \eqref{cddectoricmorphism}. Finally condition iv) of Definition \ref{decQFmorphismdef} immediatly implies condition iv) of Theorem \ref{theoremmorphismaffdec}.
\end{proof}

\subsection{Calibrated Simplicial Quantum Toric Varieties}
\label{CCGcalibrated}
Let $(\Delta,h)$ be a simplicial standard calibrated Quantum Fan in $\Gamma$. Let $v_I$ be a maximal cone of $(\Delta,h)$. Define $A_I$ as in \eqref{AI}. Let
$H_I$ be the permutation of $\Z^n$ that sends $(e_1,\hdots, e_{\vert I\vert})$ onto the $(e_I)$. Let $h_I$ be the calibration which makes the following diagram commutative
\begin{equation}
\label{CDcalibratedCone}
\begin{tikzcd}
\Gamma \arrow[r,"A_I"] &A_I\Gamma\\
\Z^n \arrow[u,"h"']\arrow[r,"H_I"] &\Z^n\arrow[u,"h_I"]
\end{tikzcd}
\end{equation}

As usual, we decompose $h_I$ as $(Id, \hbar_I)$ with $\hbar_I$ being a $((n-d)\times d)$-matrix
and let $\Z^{n-d}\ni x$ act onto
$\C^{\vert I\vert} \times \Torus^{d-\vert I\vert}\ni z$ by
\begin{equation}
\label{calibratedConeAction}
L_I(x,z):=E(\hbar_I(x))\cdot z
\end{equation}
as in \eqref{Zaction}. 

We associate to $v_I$ the calibrated affine quantum toric variety
\begin{equation}
\label{QaffineTVdec}
Q_I^{cal}=\left [
\C^{\vert I\vert} \times \Torus^{d-\vert I\vert}/\langle L_I\rangle
\right ]
\end{equation}

It is isomorphic to $Q^{cal}_{k,d,h}$ by \eqref{CDcalibratedCone}. We note the obvious fact
\begin{lemma}
	\label{indepAIbis}
	Let $(A_I,H_I)$ and $(A'_I,H'_I)$ be two linear maps as in {\rm \eqref{AIg}} and {\rm \eqref{CDcalibratedCone}}. Let $Q_I^{cal}$ and $(Q')_I^{cal}$ be the associated calibrated affine quantum toric varieties. Then the linear maps $(A'_I\circ A_I^{-1}, H'_I\circ H_I^{-1})$ induces an isomorphism between $Q_I^{cal}$ and $(Q')_I^{cal}$.
\end{lemma}

\begin{proof}
	The couple $(A'_I\circ A_I^{-1}, H'_I\circ H_I^{-1})$ is an isomorphism of the category $\mathscr{A}ff^{cal}$ of standard cones. Use Proposition \ref{propisocatdec}.
\end{proof}

The same type of description works for non-maximal cones. Starting with $J=i_1\cdots i_l$ labelling a non-maximal cone $v_J$, let $I\supset J$ such that $v_I$ is a maximal cone. Then we associate to $v_J$ the stack 
\begin{equation}
\label{QaffineTVdecnotmax}
Q_{I\supset J}^{cal}=\left [
\C^{\vert J\vert }\times \Torus^{d-\vert J\vert}/\langle L_I\rangle
\right ]
\end{equation}

Inclusion $J\subset I$ gives a calibrated toric inclusion $Q^{cal}_{I\supset J}\subset Q_I^{cal}$.
Besides, let $v_I$ and $v_{I'}$ be two maximal cones with non-empty intersection and let $J=I\cap I'$. Then the linear isomorphisms $A_{II'}:=A_JA_I^{-1}$ and $H_{II'}:=H_JH_I^{-1}$ define an isomorphism $\mathscr{A}_{II'}^{cal}$ from $Q_{I\supset J}^{cal}$ to $Q_{I'\supset J}^{cal}$, which indicates how to glue the calibrated affine quantum toric varieties $Q_I^{cal}$ and $Q_{I'}^{cal}$ along $Q_J^{cal}$. 

As in Section \ref{Cgluings}, we may thus define the calibrated Quantum Toric Variety $\nctoricgerbe{\Delta}{h}{J}$ by descent.

\begin{definition}
	\label{QToricdecObjects}
	Let $T\in\mathfrak{A}$. An {\it object of $\nctoricgerbe{\Delta}{h}{J}$ over $T$} is a covering $(T_I)_{I\in {\mathscr I_{max}}}$ of $T$ indexed over the set of maximal cones $\mathscr I_{max}$ together with an object 
	\begin{equation*}
	\begin{tikzcd}
	\tilde T_I^{cal} \arrow[r,"m_I^{cal}"]\arrow[d] &\C^{\vert I\vert}\times \Torus^{d-\vert I\vert}\\
	T_I & &
	\end{tikzcd}
	\end{equation*}
	of $Q_I^{cal}$ for any $I\in\mathscr I_{max}$ satisfying the descent datum condition
	\begin{equation*}
	\mathscr{A}^{cal}_{II'}\left ( 
	\begin{tikzcd}
	\tilde T_{I\supset J}^{cal} \arrow[r,"m_I^{cal}"]\arrow[d] &\C^{\vert J\vert}\times \Torus^{d-\vert J\vert}\\
	T_I & &
	\end{tikzcd}\kern-25pt
	\right )=
	\begin{tikzcd}
	\tilde T_{I'\supset J}^{cal} \arrow[r,"m_{I'}^{cal}"]\arrow[d] &\C^{\vert J\vert}\times \Torus^{d-\vert J\vert}\\
	T_I & &      
	\end{tikzcd}
	\end{equation*} 
	above any couple $(I,I')$ of maximal cones with non-empty intersection $J$, and setting $\tilde T_{I\supset J}^{cal}:=(m_I^{cal})^{-1}(\C^{\vert J\vert}\times \Torus^{d-\vert J\vert})$ and $\tilde T_{I'\supset J}^{cal}:=(m_{I'}^{cal})^{-1}(\C^{\vert J\vert}\times \Torus^{d-\vert J\vert})$.
\end{definition}  

\begin{definition}
	\label{QToricsdecMorphisms}
	Let $T\to S$ be a morphism from $T\in\mathfrak{A}$ and $S\in\mathfrak{A}$. A {\it morphism of $\nctoricgerbe{\Delta}{h}{J}$ over $T\to S$} is a collection of morphisms
	\begin{equation*}
	\begin{tikzcd}
	&            &\C^{\vert I\vert }\times \Torus^{d-\vert I\vert}\\
	\tilde T_I^{cal} \arrow[r, "F_I"]\arrow[rru,bend left=20,start anchor =60, end anchor=real west,"m_I^{cal}"]\arrow[d]          &\tilde S_I^{cal}  \arrow[d]\arrow[ru,bend right=10,"n_I^{cal}"']      &                       \\
	T_I        \arrow[r,"f"]         &S_I       &          
	\end{tikzcd}
	\end{equation*}
	satisfying the compatibility conditions: 
	\begin{equation*}
	\mathscr{A}^{cal}_{II'}\left (
	\begin{tikzcd}
	&            &\C^{\vert J\vert }\times \Torus^{d-\vert J\vert}\\
	\tilde T_{I\supset J}^{cal} \arrow[r, "F_I"]\arrow[rru,bend left=20,start anchor =60, end anchor=real west,"m_I^{cal}"]\arrow[d]          &\tilde S_{I\supset J}^{cal}  \arrow[d]\arrow[ru,bend right=10,"n_I^{cal}"']      &                       \\
	T_J        \arrow[r,"f"]         &S_J       &          
	\end{tikzcd}
	\right )
	\end{equation*}
	is equal to
	\begin{equation*}
	\begin{tikzcd}
	&            &\C^{\vert J\vert }\times \Torus^{d-\vert J\vert}\\
	\tilde T_{I'\supset J}^{cal} \arrow[r, "F_{I'}"]\arrow[rru,bend left=20,start anchor =60, end anchor=real west,"m_{I'}^{cal}"]\arrow[d]          &\tilde S_{I'\supset J}^{cal}  \arrow[d]\arrow[ru,bend right=10,"n_{I'}^{cal}"']      &                       \\
	T_J        \arrow[r,"f"]         &S_J       &          
	\end{tikzcd}
	\end{equation*}
\end{definition}

Let us compare with \eqref{QaffineTV}. Since $h$ and thus $h_I$ is an epimorphism, the difference sits in the kernel of these maps. Of special interest is therefore the short exact sequence \eqref{Kernelexactseq}
and its variation
\begin{equation}
\label{KernelexactseqI}
\begin{tikzcd}
0 \arrow[r] &\Xi_I:=\text{Ker }h_I \arrow[r] &\Z^n \arrow[r,"h_I"]  &A_I\Gamma\arrow[r] &0
\end{tikzcd}
\end{equation}
Let $v_I$ and $v_J$ be two maximal cones with non-empty intersection. From \eqref{CDcalibratedCone} and \eqref{KernelexactseqI}, we obtain the following commutative diagram
\begin{equation}
\label{bigCD}
\begin{tikzcd}
0 \arrow[r] &\Xi_I\arrow[r]\arrow[d,"K_{IJ}"] &\Z^n \arrow[r,"h_I"]\arrow[d,"H_{IJ}"] &A_I\Gamma \arrow[r]\arrow[d,"A_{IJ}"]& 0\\
0 \arrow[r] &\Xi_J \arrow[r] &\Z^{n} \arrow[r,"h_J"'] &A_J\Gamma \arrow[r] &0
\end{tikzcd}
\end{equation}
with $K_{IJ}$ uniquely defined by the pair $(A_{IJ},H_{IJ})$.

Using $K_{IJ}$ and \eqref{gluingmonomial}, we can define a calibrated toric variety more concretely using gluings of quasifolds. This is similar to what we did in Section \ref{Cgluings}, but this time the gluing \eqref{gluingmonomial} is defined on the calibrated torus $\nctorusgerbe{d}{h}\subset Q^{cal}_I$ through the diagram
\begin{equation}
\label{CDdecgluing}
\begin{tikzcd}
0 \arrow[r] &\Xi_I \arrow[r]\arrow[d,"K_{IJ}"]&\Z^n \arrow[d,"H_{IJ}"]\arrow[r,"h_I"] &\C^d \arrow[r]\arrow[d,"A_{IJ}"]&\nctorusgerbe{\Gamma}{h_I}\arrow[r]\arrow[d,"\mathscr{A}_{IJ}^{cal}"]&0\\
0 \arrow[r] &\Xi_J \arrow[r] &\Z^{n} \arrow[r,"h_J"]&\C^d \arrow[r] &\nctorusgerbe{\Gamma}{h_J}\arrow[r] &0
\end{tikzcd}
\end{equation}
We note that
\begin{equation}
\label{decgluing}
\begin{aligned}
A_{IJ}(z+ h_I(p))=&A_{IJ}z+A_{IJ}h_I(p)\\
&{A_{IJ}}z + h_J(H_{IJ}(p))
\end{aligned}
\end{equation}
hence is $\Z^n$-equivariant and descends as $\mathscr{A}_{IJ}^{cal}$. In other words, and passing in exponential form, the gluing
\begin{equation}
	\label{gluingmonomialdec}
	[z]\in Q^{cal}_{I\supset J}\longmapsto [z^{{A_{I'}}{A_I}^{-1}}]\in Q^{cal}_{I'\supset J}
\end{equation}
is well defined.
Then, it is extended to $Q_{I\supset I\cap J}$ by noting that the first $\vert I\cap J\vert$ rows of the matrix $A_J^{-1}A_I$ correspond to the first $\vert I\cap J\vert$ rows of the identity.

Hence, we define alternatively

\begin{definition}
\label{decQTVdef}
We call calibrated Quantum Toric Variety associated to the simplicial calibrated Quantum Fan $(\Delta,{\Gamma}, h)$ the stack obtained by gluing all the $Q^{cal}_I$ associated to maximal cones $C_I$ along their intersections through \eqref{CDdecgluing}. We denote it by $\nctoricgerbe{\Delta}{h}{J}$.
\end{definition}

Using \eqref{gluingmonomial}, we may also construct the quantum toric variety $\nctoric{\Delta}{\Gamma}{v}$. 	We have
\begin{proposition}
	\label{propQTVetQTVdec}
	The functor $\text{\sl f}$ defined on the affine Quantum calibrated cones extends as a morphism from $\nctoricgerbe{\Delta}{h}{J}$ to $\nctoric{\Delta}{\Gamma}{v}$.
\end{proposition}

\begin{proof}
	This is a direct consequence of Lemma \ref{lemmadecoublicom}.
\end{proof}
Now, observe that from diagram \eqref{CDdecgluing}, we can also glue together the affine classical toric varieties $\C^{\vert I\vert}\times\Torus^{n-\vert I\vert}$ using the cocyle $(H_{IJ})$. In this way, we obtain a classical toric variety $\mathscr{S}$. Moreover, we have
\begin{proposition}
	\label{propglobalh}
	The functor $\text{\sl h}$, resp.  $\text{\sl h}^{cal}$, extends as a morphism from $\mathscr{S}$ to $\nctoric{\Delta}{\Gamma}{v}$,  resp. to $\nctoricgerbe{\Delta}{h}{J}$.
\end{proposition}

\begin{proof}
	This is a direct consequence of diagram \eqref{cdLHhetdec}.
\end{proof}

As a consequence, we obtain a triple describing $\nctoricgerbe{\Delta}{h}{J}$ as in \eqref{cdtriple}
\begin{equation}
\label{cdtriplebis}
\begin{tikzcd}[row sep=tiny]
\mathscr{S}\arrow[dr,bend left=10,"\text{\sl h}"]\arrow[rr,bend left=20,"\text{\sl h}^{cal}"]&&\nctoricgerbe{\Delta}{h}{J}\arrow[ld,bend right=10,"\text{\sl f}"']\\
&\nctoric{\Delta}{\Gamma}{v}&\\
\end{tikzcd}
\end{equation} 
We shall see in Section \ref{QGIT} that a calibrated Quantum Toric Variety can be realized as the quotient space of $\mathscr{S}$ by a global holomorphic $\C^{n-d}$-action; and, under some additional hypotheses, as the quotient space of a non-K\"ahler complex manifold by a global holomorphic action.

\subsection{Gerbe Structure of a calibrated Quantum Toric Variety}
\label{AdditionalS}

Recall \eqref{Kernelexactseq}. Let $a$ such that $\Xi$ is isomorphic to $\Z^a$. Observe that $a$ is reduced to zero as soon as there is no $\Z$-linear relation between the $h(e_i)$, which is the generic case. In that case $\Xi$ is reduced to zero and \eqref{Kernelexactseq} reduces to $h$ being an isomorphism from $\Z^n$ to $\Gamma$. As a consequence $\nctoricgerbe{\Delta}{h}{J}$ and $\nctoric{\Delta}{\Gamma}{v}$ are the same stack. 

Assume now that $a$ is maximal, that is equal to $n-d$. Then, the matrices $A_I$ have integer entries and action \eqref{Zaction} is trivial so that $Q_I^{cal}$ is a gerbe over the classical affine toric variety $Q_I=\C^k\times\Torus^{d-k}$ with band $\Z^{n-d}$. 

In the intermediate case $0<a<n-d$, then $Q_I^{cal}$ is still a gerbe over the classical affine toric variety $Q_I=\C^k\times\Torus^{d-k}$ but with band $\Z^a$.  Globally we thus have by functoriality of the construction

\begin{proposition}
	\label{propgerbe}
	The calibrated Quantum Torics $\nctoricgerbe{\Delta}{h}{J}$ is a gerbe over $\nctoric{\Delta}{\Gamma}{v}$ with band $\Z^{a}$. In particular, if $a=0$, $\nctoricgerbe{\Delta}{h}{J}$ and $\nctoric{\Delta}{\Gamma}{v}$ are isomorphic.
\end{proposition}

Homotopically, such a gerbe is given by a map in the classifying space $\mathcal B\mathcal B (\Z^a)$, which is nothing else than $(\mathbb P^{\infty})^{a}$. In other words, such a gerbe defines a principal $\Torus^{a}$ bundle over $\nctoric{\Delta}{\Gamma}{v}$ up to homotopy. We can describe this bundle as $\text{\sl h}$.

We first assert that there exists a natural morphism $\mathscr{P}$ from $\C^k\times\Torus^{n-k}$ to $Q^{cal}_{k,d,h}$. To see this,
let $\Z^{n-d}$ acts on $\C^k\times\Torus^{d-k}\times \C^{n-d}$ as
\begin{equation}
\label{actionconeV}
(p,(z,w))\in\Z^{n-d}\times (\C^k\times\Torus^{d-k})\times\C^{n-d}\longmapsto (E(\hbar(p))\cdot z,w+p)
\end{equation}
Then the projection map $P:(\C^k\times\Torus^{d-k})\times\C^{n-d}\to \C^k\times\Torus^{d-k}$ is equivariant for action \eqref{actionconeV} on the left and for action \eqref{Zaction}, hence it descends as a morphism
\begin{equation}
\label{morphismVdec}
\begin{tikzcd}
(\C^k\times\Torus^{d-k})\times\C^{n-d}\arrow[r,"P"]\arrow[d] &\C^k\times\Torus^{d-k}\arrow[d]\\
(\C^k\times\Torus^{d-k}\times\C^{n-d})/\Z^{n-d}\arrow[r,"{\mathscr{P}}"] &Q^{cal}_{k,d,h}
\end{tikzcd}
\end{equation}
We think of \eqref{morphismVdec} as a trivialization diagram for the morphism $\mathscr{P}$, making of $\C^k\times\Torus^{n-k}$ a trivial principal $\C^{n-d}$-bundle above $Q^{cal}_{k,d,h}$.
We then notice that $P$ is related to $\hat h$ through
\begin{equation}
\label{trivCmbundle}
\begin{tikzcd}
\C^k\times\Torus^{d-k}\times\C^{n-d}\arrow[dd,"(\hat h^{-1}\text{,}Id)"']\arrow[dr,"P"]&\\
&\C^k\times\Torus^{d-k}\\
(\C^k\times\Torus^{d-k})\times\C^{n-d}\arrow[ur,"\hat h"]
\end{tikzcd}
\end{equation}
Observe that $(\hat h^{-1},Id)$ is an isomorphism which is moreover equivariant for the action of $\Z^{n-d}$ on the source by \eqref{actionconeV} and on the target by integer translations on the $\C^{n-d}$-factor. So it descends as
\begin{equation}
\label{trivCmbundlebis}
\begin{tikzcd}
(\C^k\times\Torus^{d-k}\times\C^{n-d})/\Z^{n-d}\arrow[dd,"(\hat h^{-1}\text{,}Id)"']\arrow[dr,"\mathscr{P}"]&\\
&Q^{cal}_{k,d,h}\\
(\C^k\times\Torus^{d-k})\times\Torus^{n-d}\arrow[ur,"\text{\sl h}^{cal}"]
\end{tikzcd}
\end{equation}
By functoriality,
 \begin{proposition}
	\label{propliftPdec}
	The morphism $\text{\sl h}^{cal}$ from $\mathscr{S}$ to $\nctoric{\Delta}{\Gamma}{v}$ makes of $\mathscr{S}$ a principal $\C^{n-d}$ principal bundle over $\nctoricgerbe{\Delta}{h}{J}$.
\end{proposition}
 Because of Proposition \ref{propgerbe}, when considering $\text{\slshape{h}}$ and no more $\text{\slshape{h}}^{cal}$, we have immediatly
 
 \begin{proposition}
 	\label{propliftP}
 	The morphism $\text{\sl h}$ from $\mathscr{S}$ to $\nctoric{\Delta}{\Gamma}{v}$ makes of $\mathscr{S}$ a principal $\Torus^{a}\times \C^{n-d-a}$ principal bundle over $\nctoric{\Delta}{\Gamma}{v}$.
 \end{proposition}

This bundle is a concrete realization of the gerbe structure described in Proposition \ref{propgerbe}.

 \begin{example}
 \label{QP1dec}
 Let us go back to the calibrated Quantum Fan of Example \ref{P1QFans}. If $a$ is an integer, then $\Gamma$ is $\Z$ and it encodes the standard projective line $\mathbb P^1$. The calibrated Quantum Toric Variety is a $\Z^2$-gerbe over $\mathbb P^1$. The toric variety $\mathscr{S}$ is $\C^2\setminus\{(0,0)\}\times\Torus$, associated to the fan $(\R^+\cdot e_1,\R^+\cdot e_2)$ in $\R^3$ equipped with the standard integer lattice. Thus, the $\Torus^2$-bundle of Proposition \ref{propliftP} is just $(x,y,z)\in \mathscr{S}\mapsto [x,y]\in\mathbb P^1$, that is $\mathcal O(-1)\setminus\{0\}\oplus\mathcal O\setminus\{0\}$. It becomes a $\C^2$-bundle when replacing $\mathbb P^1$ with a $\Z^2$-gerbe over $\mathbb P^1$. For $a$ irrational, we still have a gerbe, but it has now band $\Z$.
 \end{example}
 
 \subsection{Morphisms of (calibrated) Quantum Toric Varieties}
 \label{QTVmorphismsection}
 A morphism between two simplicial calibrated Quantum Toric Varieties $\nctoricgerbe{\Delta}{h}{J}$ and $\nctoricgerbe{\Delta'}{h'}{J'}$ is simply a collection of calibrated affine toric morphisms $\mathscr{L}^{cal}_{II'}$ between the affine toric varieties $Q^{cal}_I$ and $(Q')^{cal}_{I'}$ that are compatible with the gluings. 
 Compatibility means that $\mathscr{A'}^{cal}_{I'J'}\circ \mathscr{L}^{cal}_{II'}$ and $\mathscr{L}^{cal}_{JJ'}\circ \mathscr{A}^{cal}_{IJ}$ are equal on $Q^{cal}_{J\supset{I}}$.
 Recalling that $\mathscr{L}_{II'}$ is encoded in $(L_{II'},H_{II'})$ and recalling \eqref{CDdecgluing}, we get easily that compatibility means
 \begin{equation*}
 A'_{I'J'}L_{II'}=L_{JJ'}A_{IJ}\qquad\text{ and }\qquad H'_{I'J'}H_{II'}=H_{JJ'}H_{IJ}
 \end{equation*}
 In other words, the linear maps $(A'_{I'}L_{II'}A_{I},H'_{I'}H_{II'}H_I)$ glue together in a Quantum Fan morphism $(L,H)$.
 
 Conversely, it is straightforward to check that a morphism of calibrated Quantum Fans induces a morphism of the corresponding calibrated Quantum Toric Varieties. From this, we obtain (compare with Theorem \ref{theoremtoricmorphismcondition})
 
 \begin{theorem}
 	\label{prisodecQTV}\textcolor{white}{fff}
 	\begin{enumerate}[\rm i)]\kern-5pt
 		\item A stack morphism between calibrated Quantum Toric Varieties is a toric morphism if and only if its restriction to the underlying Quantum tori is a calibrated torus morphism.
 		\item Let $\mathscr Q^{cal}$ be the category of simplicial calibrated Quantum Toric Varieties. Then $\mathscr Q^{cal}$ is equivalent to the category of quantum toric fans $Q^{cal}$.
 	\end{enumerate}
 \end{theorem}
 
  Finally, we may sum up the properties of calibrated toric morphisms in the following two theorems. They are obtained respectively from Theorem \ref{theoremmorphismaffdec} and Lemma \ref{lemmaLHhetdec}.
 
 \begin{theorem}
 	\label{thmdecmorphisms2}
 	There is a $1:1$ correspondence between morphisms $\mathscr{L}^{cal}$ from $\nctoricgerbe{\Delta}{h}{J}$ to $\nctoricgerbe{\Delta'}{h'}{J'}$ and pairs $(\mathscr{L},\mathscr{H})$ such that
 	\begin{enumerate}
 		\item [\rm i)] $\mathscr L$ is a morphism of affine Quantum varieties from $\nctoric{\Delta}{\Gamma}{v}$ to $\nctoric{\Delta'}{\Gamma'}{v'}$.
 		\item[\rm ii)] $\mathscr H$ is a morphism of classical toric varieties from $\mathscr{S}$ to $\mathscr{S}'$.
 		\item [\rm iii)] The following diagram is commutative
 		\begin{equation}
 		\label{cddectoricmorphismfull}
 		\begin{tikzcd}
 		\mathscr{S}\arrow[d,"{\text{\sl h}}"']\arrow[r,"\mathscr H"]&\mathscr{S}'\arrow[d,"{\text{\sl h}}'"]\\
 		\nctoric{\Delta}{\Gamma}{v}\arrow[r,"\mathscr L"]&\nctoric{\Delta'}{\Gamma'}{v'}
 		\end{tikzcd}
 		\end{equation}
 		\item[\rm iv)] For $j\not \in J'$, the $j$-th component $\mathscr{H}_j$ is a monomial with integer coefficients in the variables $(w_i)_{i\not\in J}$.
 		\item[\rm v)] There exists a map $s$ from $J$ to $J'$ such that, for $j\in J'$, the $j$-th component $\mathscr{H}_{j}$ satisfies 
 		\begin{equation*}
 		\mathscr{H}_{j}(w)=\left \{
 		\begin{aligned}
 		\prod_{i\in s^{-1}(j)}&w_i\qquad\text{ if }s^{-1}(j)\not =\emptyset\cr
 		&\ 1 \qquad \text{ if }s^{-1}(j)=\emptyset
 		\end{aligned}
 		\right .
 		\end{equation*}
 	\end{enumerate}
 \end{theorem}
 and
 \begin{theorem}
 	\label{thmdecmorphisms3}
 	The following diagram is commutative
 	\begin{equation}
 	\label{cdLHhetdecglob}
 	\begin{tikzcd}[row sep=scriptsize, column sep=scriptsize]
 	& \mathscr{S} \arrow[dl, "Id"'] \arrow[rr,"\mathscr{H}"] \arrow[dd,"\text{\sl h}^{cal}"' near start] & & \mathscr{S}' \arrow[dl,"Id"'] \arrow[dd,"(\text{\sl h'})^{cal}"] \\
 	\mathscr{S} \arrow[rr, crossing over,"\mathscr{H}" near end] \arrow[dd,"{\text{\sl h}}"'] & & \mathscr{S}'\\
 	& \nctoricgerbe{\Delta}{h}{J} \arrow[dl,"{\text{\sl f}}"'] \arrow[rr,"\mathscr{L}^{cal}" near start] & & \nctoricgerbe{\Delta'}{h'}{J'} \arrow[dl,"{\text{\sl f}}"'] \\
 	\nctoric{\Delta}{\Gamma}{v}\arrow[rr,"\mathscr{L}"] & & \nctoric{\Delta'}{\Gamma'}{v'} \arrow[from=uu, crossing over,"{\text{\sl h'}}" near start]\\
 	\end{tikzcd}
 	\end{equation}
 \end{theorem}
 Also, the following straight adaptations of Corollary \ref{corcaliso} and Definitions \ref{defisomarked} are valid.
 \begin{corollary}
 	\label{corcalisofull}
 	The following points are equivalent.
 	\begin{enumerate}[i)]
 		\item The stack isomorphism $\mathscr{L}^{cal}$ is an toric isomorphism $\nctoricgerbe{\Delta}{h}{J}$ to $\nctoricgerbe{\Delta'}{h'}{J'}$ (that is both $\mathscr{L}^{cal}$ and its inverse are toric morphisms).
 		\item The associated Quantum Fan morphism $(L,H)$ is an isomorphism.
 		\item The restriction of the stack isomorphism $\mathscr{L}^{cal}$ to the Quantum tori is an isomorphism.
 	\end{enumerate}
 \end{corollary}
 From this Corollary, we define
 \begin{definition}
 	\label{defisomarkedfull}
 	A stack isomorphism $\mathscr{L}^{cal}$ is a {\it marked isomorphism} if the associated Quantum Fan morphism $(L,H)$ is a marked isomorphism, or, equivalently, if its restriction to the Quantum tori is a marked isomorphism.
 \end{definition}

\begin{example}
We go back to Example \ref{P1P2morphisms} and we want to describe the associate morphism between calibrated Quantum Toric Varieties. We assume $x=a\alpha$, $y=b\alpha$, with $\alpha\geq\beta\geq 0$ integers. The calibrated toric variety $\nctoricgerbe{\Delta}{\Gamma}{h_a}$ is described in Examples \ref{quantumP1} and \ref{QP1dec} and is a calibrated version of a quantum projective line. It has two charts corresponding to the two $1$-cones.
In the same way, $\nctoricgerbe{\Delta_2}{\Gamma}{h}$ is a calibrated version of a Quantum projective space. 
The morphism splits as two morphisms 
\begin{equation*}
[z]\in [\C/\exp(2i\pi\Gamma_a)]\longmapsto [z^\alpha,z^\beta]\in[\C^2/\exp(2i\pi\Gamma)]
\end{equation*}
and
\begin{equation*}
[w]\in [\C/\exp(2i\pi(-\Gamma_a))]\longmapsto [z^{\alpha-\beta},z^\alpha]\in[\C^2/\exp(2i\pi A_{23}^{-1}\Gamma)]
\end{equation*}
with
\begin{equation*}
A_{23}^{-1}=
\begin{pmatrix}
-1 &1\\
-1 &0
\end{pmatrix}
\end{equation*}
with corresponding homomorphisms $H$ described in Example \ref{P1P2morphisms} from the one hand and
\begin{equation*}
\begin{pmatrix}
\alpha-\beta &\beta &0\\
\alpha &0 &0\\
0 &\alpha &0\\
0 &0 &1
\end{pmatrix}
\end{equation*}
from the other hand.
\end{example}

\begin{example}
	\label{exqdP2dec}
	We construct the trivially calibrated $\Gamma$-complete quantum deformation of $\mathbb P^2$ whose calibrated fan is given in Example \ref{exP2dec}. It is important to compare with the construction of Example \ref{exqdP2}. 
	
	The three matrices $H_{12}$, respectively $H_{23}$ and $H_{31}$ corresponding to \eqref{AIP2} are just the permutation matrices of the permutation $(12)$, resp. $(23)$ and $(31)$. The associated $\hbar$ functions can be read off from \eqref{chartsQDP2}. It gives
	\begin{equation}
	\label{hslashP2}
	\left\{
	\begin{aligned}
	&\hbar_{12}(x)=(xa,xb)\\
	&\hbar_{23}(x)=(-xb/a,x/a)\\
	&\hbar_{31}(x)=(x/b,-xa/b) 
	\end{aligned}
	\right .
	\end{equation}
	The main difference with Example \ref{exqdP2} is that the three charts are now modelled on the quotient of $\C^2$ by a $\Z$-action, regardless of the rationality or irrationality of $a$ and $b$.
	
	However, there still exists a difference between the rational and the irrational case. If $a$ and/or $b$ is irrational, then all $\hbar$ functions in \eqref{hslashP2} are injective and the calibrated quantum deformation of $\mathbb P^2$ is isomorphic to the corresponding non-calibrated quantum deformation described in Example \ref{exqdP2}. But if $a$ and $b$ are rational, the calibrated quantum deformation of $\mathbb P^2$ is a $\Z$-gerbe over the corresponding non-calibrated quantum deformation. In particular, $a=b=-1$ does not correspond here to $\mathbb P^2$ but to a $\Z$-gerbe over $\mathbb P^2$.
	
	Finally, note that $\mathscr S$ is $\C^3\setminus\{0\}\}$, so the $\Z$-gerbe corresponds to the bundle $\mathcal O(-1)$.
\end{example}

\section{Quantum GIT}
\label{QGIT}
In this section, we will generalize the classical construction of toric varieties as GIT quotients to the quantum case; in particular, we show that Quantum toric stacks can be constructed as global quotients.

\subsection{Quantum calibrated GIT}
\label{QGITdec}
We first deal with the calibrated case. So start now from a simplicial calibrated Quantum Fan $(\Delta,h)$ in $\Gamma$. Set $v_i=h(e_i)$ for $i=1,\hdots,n$. We will show in this section how to construct the calibrated quantum toric variety $\nctoricgerbe{\Delta}{h}{J}$ as a quotient stack.
 
To do this, we make use of Gale transforms, a classical tool in convex geometry. We thus perform a Gale transform of $v=(v_1,\hdots,v_{n})$, that is we choose some vectors $A=(A_1,\hdots,A_{n})$ of $\R^{n-d}$ such that
\begin{equation}
\label{GTdec}
h(x)=\sum_{i=1}^{n}x_iv_i=0
\iff \left\{
\begin{aligned}
x_1=&\langle A_1,t\rangle\\
\vdots&\\
x_{n}=&\langle A_{n},t\rangle
\end{aligned}
\right .
\quad\text{ for some }t\in\R^{n-d}
\end{equation}
This leads to a short exact sequence
\begin{equation}
\label{sesGale}
\begin{tikzcd}
0\arrow[r]&\R^{n-d}\arrow[r,"k"]&\R^n\arrow[r,"h"]&\R^d\arrow[r]&0
\end{tikzcd}
\end{equation}
for
\begin{equation}
\label{k}
k(t_1,\hdots, t_{n-d})=k(t)=(\langle A_1,t\rangle, \hdots, \langle A_n,t\rangle)
\end{equation}
As usual, we assume that the fan is standard and decompose $h$ as in \eqref{h}. We note the following easy fact

\begin{lemma}
	\label{lemmah}
	Set $\hbar=(\hbar_1,\hdots,\hbar_d)$. Then, for all $i$ between $1$ and $d$ and forall $t\in\R^{n-d}$, one has
	\begin{equation}
	\label{Aandh}
	0=\langle A_i,t\rangle +\hbar_i(\langle A_{d+1},t\rangle,\hdots,\langle A_n,t\rangle)
	\end{equation}
\end{lemma}

\begin{proof}
	Just compute $h\circ k$ using \eqref{h}.
\end{proof}
\noindent and its corollary
\begin{corollary}
		\label{corh}
		Assume that $\langle A_i,t\rangle$ belongs to $\Z$ for all $i$ between $d+1$ and $n$. Then $(\langle A_{1},t\rangle,\hdots,\langle A_d,t\rangle)$ belongs to $\Gamma$.
\end{corollary}
 \noindent Define
 \begin{equation}
 \label{Iz}
 i\in I_z\iff z_i\not =0.
 \end{equation}
 and set
\begin{equation}
\label{S}
\mathscr S=\{z\in\C^{n}\mid \{1,\hdots,n+1\}\setminus I_z\text{ is a cone of }\Delta\}
\end{equation}
Observe the
\begin{lemma}
	\label{lemmaSaffine}
	The set $\mathscr{S}$ is an affine toric variety.
\end{lemma}

\begin{proof}
	The set $\mathscr{S}$ is the complement in $\C^n$ of a union of coordinate vector subspaces and the classical torus $\Torus^n$ acts multiplicatively on it with a Zariski dense orbit, that of $(1,\hdots,1)$. 
\end{proof}
\noindent Define finally a (classical) fan $\Delta_H$ in $\R^n$ with the lattice of integer points as follows
\begin{equation}
\label{deltaH}
\R^+\cdot e_I\in\Delta_H\iff I\text{ is a cone of }\Delta
\end{equation}
We have
\begin{lemma}
	\label{lemmaDHetS}
	The fan $\Delta_H$ is the fan of $\mathscr{S}$.
\end{lemma}

\begin{proof}
	By definition, $\R^+\cdot e_I$ is a cone of $\Delta_H$ if and only if 
	\begin{equation}
	\label{CI}
	\C^I\times\Torus^{\bar I}:=\{z\in\C^n\mid z_i\not =0\text{ for } i\not\in I\}
	\end{equation}
	is included in $X(\Delta_H)$, the toric variety associated to $\Delta_H$. But this occurs if and only if $I$ is a cone of $\Delta$ by \eqref{deltaH}, hence $X(\Delta_H)$ is $\mathscr{S}$ by \eqref{S}.
\end{proof}

We let now $\C^{n-d}\ni T$ act holomorphically on $\C^n\ni z$ through
\begin{equation}
\label{completeactionnp}
T\cdot z:=\left (
z_iE({\langle A_i,T\rangle})
\right )_{i=1} ^n
\end{equation}
Note that
\begin{lemma}
	\label{lemmacommutes}
	Action {\rm \eqref{completeactionnp}} preserves $\mathscr{S}$ and commutes with the action of $\Torus^n$ on $\mathscr{S}$.
\end{lemma}
We denote by $\mathcal{A}$ this action and we form the global quotient $[\mathscr{S}/\mathcal{A}]$. We consider it as a stack over the site $\mathfrak{A}$. This is possible and pertinent thanks to Lemmas \ref{lemmaSaffine} and \ref{lemmacommutes} and proposition \ref{propstacks}.

\medskip

 We have
 \begin{theorem}
 	\label{TheoremGITdec}
 	The stacks $[\mathscr{S}/\mathcal{A}]$ and $\nctoricgerbe{\Delta}{h}{J}$ are isomorphic.
 \end{theorem}

\begin{proof}
	We decompose $\mathscr{S}$ as the union of affine toric varieties $\C^I\times \Torus^{\bar I}$
	for every maximal cone $I$ of $\Delta$, cf. \eqref{CI}. We will show that 
	\begin{enumerate}[a)]
		\item every $[(\C^I\times \Torus^{\bar I})/\mathcal{A}]$ is isomorphic to $Q_I^{cal}$, 
		\item and that the natural gluings between such pieces, say encoded by $I$ and $I'$ correspond through these isomorphisms to gluings of $Q_I^{cal}$ and $Q_{I'}^{cal}$ through $\mathscr{A}^{cal}_{II'}$
	\end{enumerate}
Hence, up to isomorphism, both stacks are obtained by descent of the same collection $(Q_I^{cal},\mathscr{A}^{cal}_{II'})$, proving Theorem \ref{TheoremGITdec}.

Let $I=\{i_1,\hdots, i_{\vert I\vert}\}$ be a maximal cone of $\Delta$. If $\vert I\vert$ is strictly less than $d$, we choose some subset $J=\{j_1,\hdots,j_{\vert J\vert}\}$ of $\{1,\hdots,n\}\setminus I$ of complementary cardinal $d-\vert I\vert$. 
\\
Let
\begin{equation}
\label{AIcal}
\mathcal A_I=\{T\in\C^{n-d}\mid \langle A_i,T\rangle\in \Z\text{ for } i\in \bar I\setminus J\}
\end{equation}
and consider the action of $\mathcal A_I\ni T$ on $\C^{\vert I\vert}\times\Torus^{d-\vert I\vert}\ni z$ given by
\begin{equation}
\label{AIaction}
\begin{aligned}
zE(\langle A_{I\cup J},T\rangle)&:=(z_1E(\langle A_{i_1},T\rangle),\hdots,z_{\vert I\vert}E(\langle A_{\vert I\vert},T\rangle),\\
& z_{\vert I\vert+1}E(\langle A_{j_1},T\rangle),\hdots,z_dE(\langle A_{j_{\vert J\vert}},T\rangle))\end{aligned}
\end{equation}
The first key remark is the following
\begin{lemma}
	\label{lemmaGITfirstiso}
	The stacks $[(\C^I\times \Torus^{\bar I})/\mathcal{A}]$ and $[(\C^{\vert I\vert}\times \Torus^{d-\vert I\vert})/\mathcal{A}_I]$ are isomorphic.
\end{lemma}

\begin{proof}[Proof of Lemma \ref{lemmaGITfirstiso}]
We consider the transverse section 
\begin{equation}
T_I:=\{z\in\C^{I}\times\Torus^{\bar I}\mid z_i=1\text{ for }i\in\bar I\setminus J\}
\end{equation}
and will compute its intersection with an orbit of $\C^{I}\times\Torus^{\bar I}$ through the action \eqref{AIaction}. Hence we write
\begin{equation*}
z_iE(\langle A_{i},T\rangle)=1\leqno{i\in\bar I\setminus J}
\end{equation*}
This gives a solution which is unique up to addition of an element of $\mathcal{A}_I$, i.e.
the map 
\begin{equation}
\label{transversaliso}
z\in T_I\longmapsto (z_{i_1},\hdots, z_{i_{\vert I\vert}}, \hdots, z_{j_{\vert J\vert}})\in \C^I\times \Torus^{\bar I}
\end{equation}
defines an isomorphism between the stack $[(\C^I\times \Torus^{\bar I})/\mathcal{A}]$ and the stack $[(\C^{\vert I\vert}\times \Torus^{d-\vert I\vert})/\mathcal{A}_I]$ as wanted.
\end{proof}
 Let $A_I$ send $v_{i_k}$ onto $e_k$ and $v_{j_k}$ onto $e_{\vert I\vert +k}$. Let $H_I$ be an isomorphism of $\R^n$ sending $e_{i_k}$ onto $e_k$ and $e_{j_k}$ onto $e_{\vert I\vert +k}$ and making the diagram \eqref{CDcalibratedCone} commutative. Then, defining $L_I$ and $Q_I^{cal}$ as in \eqref{calibratedConeAction} and \eqref{QaffineTVdec},
\begin{lemma}
\label{lemmaGITsecondiso}
The stacks $Q_I^{cal}$ and $[(\C^{\vert I\vert}\times \Torus^{d-\vert I\vert})/\mathcal{A}_I]$ are equal.
\end{lemma}

\begin{proof}[Proof of Lemma \ref{lemmaGITsecondiso}]
	Setting $h_I=Id+\hbar_I$, it follows from \eqref{CDcalibratedCone} and Lemma \ref{lemmah} that
	\begin{equation}
	\label{hIetGamma}
	\langle A_i,T\rangle=-(\hbar_I)_i(\langle A_{\bar I\setminus J},T\rangle )\qquad \text{ for }i\in I\cup J
	\end{equation}
	hence, action $L_I$ and \eqref{AIaction} are exactly the same action. 
\end{proof}
Lemmas \ref{lemmaGITfirstiso} and \ref{lemmaGITsecondiso} proves a). \vspace{3pt}\\
Let now $I$ and $I'$ two maximal cones (associated to $J$ and $J'$ if necessary) with non-empty intersection. We need to understand the gluing between $[(\C^{\vert I\cap I'\vert}\times \Torus^{d-\vert I\cap I'\vert})/\mathcal{A}_I]$ and $[(\C^{\vert I\cap I'\vert}\times \Torus^{d-\vert I\cap I'\vert})/\mathcal{A}_{I'}]$. In other words, starting from a point $z\in T_I$, we need to understand the intersection of its orbit with the transverse section $T_{I'}$. \\
This is similar to the computation we made in the proof of Lemma \ref{lemmaGITfirstiso}. The arrow is a composition of several maps. Firstly,
\begin{equation}
\label{compo1}
z\in \C^{\vert I\cap I'\vert}\times \Torus^{d-\vert I\cap I'\vert}\longmapsto w:=H_I^{-1}(z,1)\in T_I
\end{equation}
is the inverse map to \eqref{transversaliso}. Secondly,
\begin{equation}
\label{compo2}
w\longmapsto w':=wE(\langle A,T_0\rangle)\in T_{I'}
\end{equation}
with $T_0$ determined by the equations
\begin{equation}
\label{wT0}
w_iE(\langle A_{i},T\rangle)=1\qquad {i\in\bar I'\setminus J'}
\end{equation}
Thirdly,
\begin{equation}
\label{compo3}
w'\longmapsto H_{I'}(w')=:(z',1)
\end{equation}
and finally projects onto $z'\in \C^{\vert I\cap I'\vert}\times \Torus^{d-\vert I\cap I'\vert}$.\\
Let us rewrite this composition in restriction to $\Torus^d$. It firstly maps $z=E(Z)$ to $H_I^{-1}(E(Z,0))=E(H^{-1}_I(Z,0))$ ; secondly to $E(H^{-1}_I(Z,0)+\langle A,T_0\rangle)$ ; and thirdly to
\begin{equation}
\label{Ecompo}
H_{I'}(E(H^{-1}_I(Z,0)+\langle A,T_0\rangle))=E(H_{II'}(Z,0)+H_{I'}(\langle A,T_0\rangle ))
\end{equation}
Now, using \eqref{wT0} and \eqref{hIetGamma}, we have
\begin{equation*}
H_{I'}(\langle A,T_0\rangle )=(\hbar_{I'}(W_{\bar I'\setminus J'},-W_{\bar I'\setminus J'}))
\end{equation*}
hence
\begin{equation*}
\begin{aligned}
E(H_{II'}(Z,0)+H_{I'}(\langle A,T_0\rangle ))&=E((W_{I'\cup J'},W_{\bar I'\setminus J'}))+(\hbar_{I'}(W_{\bar I'\setminus J'},-W_{\bar I'\setminus J'}))\\
&=E(h_{I'}(W_{I'\cup J'},W_{\bar I'\setminus J'}),0)\\
&=E(h_{I'}H_{II'}(Z,0),0)\\
&=E(A_{II'}h_I(Z,0))\text{ by \eqref{CDcalibratedCone}}\\
&=E(A_{II'}Z)\\
&=z^{A_{II'}}
\end{aligned}
\end{equation*}
exactly as for the gluings defined in \eqref{gluingmonomialdec}. This proves b).
\end{proof}

\begin{corollary}
	\label{coratlasdec}
	The stack morphism $\text{\sl h}\ :\ \mathscr{S}\to\nctoricgerbe{\Delta}{h}{J}$ is an atlas for $\nctoricgerbe{\Delta}{h}{J}$, that is the diagram
	\begin{equation}
	\label{CDatlasdec}
	\begin{tikzcd}
	\C^m\times\mathscr{S}\arrow[d,"\text{\rm action} "']\arrow[r,"pr_2"]\arrow[dr,phantom,"\scriptstyle{\square}",very near start, shift left=0ex]&\mathscr{S}\arrow[d,"\text{\sl h}"]\\
	\mathscr{S}\arrow[r,"\text{\sl h}"']&\nctoricgerbe{\Delta}{h}{J}
	\end{tikzcd}
	\end{equation}
	2-commutes.
\end{corollary}
In particular, for every object $T$ of $\mathfrak A$, we have
\begin{equation}
\label{CDatlasdecexpl}
\begin{tikzcd}
T
\arrow[drr, bend left, "u"]
\arrow[ddr, bend right, "v"']
\arrow[dr, dotted, "{(a,u)}" description] & & \\
&\C^m\times\mathscr{S}\arrow[d,"\text{action}"']\arrow[r,"pr_2"]\arrow[dr,phantom,"\scriptstyle{\square}",very near start, shift left=0ex]&\mathscr{S}\arrow[d,"\text{\sl h}"]\\
&\mathscr{S}\arrow[r,"\text{\sl h}"']&\nctoricgerbe{\Delta}{h}{J}
\end{tikzcd}
\end{equation}
where the action satisfies $a\cdot u=v$.
\subsection{Quantum GIT and holonomy groupoid}
Secondly, \eqref{completeactionnp} defines also a holomorphic foliation on $\mathscr{S}$. The leaves are the orbits of this action. We may thus form its \'etale holonomy groupoid, see \cite[\S 5.2]{MM} for the definition of the \'etale holonomy groupoid.\\
To do this, we need a complete set of transverse sections to the foliation. We take the set $(T_I)_{I \text{ maximal cone of } \Delta}$. We identify each $T_I$ with $\C^{\vert I\vert}\times \Torus^{d-\vert I\vert}$ through \eqref{transversaliso}.\\
The set of morphisms encode the holonomy morphisms, that is local morphisms between transverse sections which identifies a point $z$ in some transversal $T_I$ with the intersection of the leaf through $z$ and some other transversal $T_{I'}$.\\
Taking into account Section \ref{QGITdec}, we dispose of a complete description of these holonomy morphisms. From the one hand, the holonomy morphisms fixing a transversal $\C^{\vert I\vert}\times \Torus^{d-\vert I\vert}$ are holomorphic maps
\begin{equation}
\label{holonomyI}
z\in\C^{\vert I\vert}\times \Torus^{d-\vert I\vert}\longmapsto zE(\langle A_{I\cup J},T\rangle)\in \C^{\vert I\vert}\times \Torus^{d-\vert I\vert}
\end{equation}
for $T$ belonging to $\mathcal A_I$, cf. \eqref{AIcal} and \eqref{AIaction}.\\
However, one must take care of the fact that different $T$ and $T'$ in $\mathcal A_I$ may give the same holonomy morphism \eqref{holonomyI}. They must be identified in the holonomy groupoid. Now, we know from Lemma \ref{lemmaGITsecondiso} that it is the case if and only if 
\begin{equation*}
\hbar_I(\langle A_{\bar I\setminus J},T\rangle )=\hbar_I(\langle A_{\bar I\setminus J},T'\rangle )
\end{equation*}
In other words, such a holonomy morphism is uniquely encoded by 
\begin{equation*}
\hbar_I(\langle A_{\bar I\setminus J})\in\hbar_I(\Z^{n-d})
\end{equation*}
But $\hbar_I(\Z^{n-d})$ is $\Gamma_I$ and the conclusion of all that preceeds is that the holonomy morphisms of $\C^{\vert I\vert}\times \Torus^{d-\vert I\vert}$ are transformations of the group action
\begin{equation}
\label{holonomyIbis}
(p,z)\in\tilde\Gamma_I\times \C^{\vert I\vert}\times \Torus^{d-\vert I\vert}\longmapsto z\cdot E(p)\in \C^{\vert I\vert}\times \Torus^{d-\vert I\vert}
\end{equation}
 From the other hand, the holonomy morphisms between $T_I$ and $T_{I'}$ give rise to holomorphic maps defined on 
\begin{equation*}
\C^{\vert I\cap I'\vert}\times \Torus^{d-\vert I\cap I'\vert}\subset \C^{\vert I\vert}\times \Torus^{d-\vert I\vert}
\end{equation*}
and obtained as the compositions of \eqref{compo1}, \eqref{compo2} and \eqref{compo3}.\\ 
However, these compositions are only locally well defined since they correspond to local determinations of the stack morphism $z\mapsto z^{A_{II'}}$. And locally here means locally for the euclidean topology, not locally for the Zariski topology, which is the topology we use in the definition of our site $\mathfrak A$. This is a real problem as we cannot replace $T_I$ with an adequate covering.
So to avoid it, we modify the construction of the holonomy groupoid in a different manner and replace the transversal $\C^{\vert I\vert}\times \Torus^{d-\vert I\vert}$ with $\C^{\vert I\vert}\times \C^{d-\vert I\vert}$ through the map $(z,w)\mapsto (z,E(w))$.\\
Hence we replace \eqref{holonomyI} with
\begin{equation}
\label{holonomyIexp}
(p,z,w)\in\tilde\Gamma_I\times \C^{\vert I\vert}\times \C^{d-\vert I\vert}\longmapsto (z\cdot E(p_1),w+p_2)\in \C^{\vert I\vert}\times \C^{d-\vert I\vert}
\end{equation}
where $p=(p_1,p_2)\in\C^{\vert I\vert}\times \C^{d-\vert I\vert}$. We denote the corresponding holonomy morphisms by $\text{hol}_{I,p}$. And we consider the induced holonomy morphisms
\begin{equation}
\label{holonomyII}
(z,w)\in\C^{\vert I\cap I'\vert}\times (\Torus^{\vert I\setminus I'\vert }\times\C^{d-\vert I\vert})\longmapsto (zE(B_{II'}w),\tilde A_{II'}w)
\end{equation}
where we set
\begin{equation*}
A_{II'}=\begin{pmatrix}
I_{\vert I\cap I'\vert} &B_{II'}\\
0 &\tilde A_{II'}
\end{pmatrix}
\end{equation*}
We denote them by $\text{hol}_{II'}$. Then we may compose \eqref{holonomyIexp} and \eqref{holonomyII} and so on. Indeed, given any path of maximal cones $\mathcal I=(I_1,\hdots, I_k)$ with $I_1\cap \hdots\cap I_k$ non-empty, we may consider the composition
\begin{equation}
\label{holonomycomp}
\text{hol}_{I_k,p_k}\circ \text{hol}_{I_{k-1}I_k}\circ\hdots\circ \text{hol}_{I_1I_2}\circ \text{hol}_{I_1,p_1}
\end{equation}
defined on $\C^{\vert I_1\cap \hdots\cap I_k\vert}\times\Torus^{\vert I_1\setminus (I_2\cup\hdots\cup I_{k})\vert}\times \C^{d-\vert I\vert}$.

We are in position to describe the groupoid. 
The set of objects of our groupoid, say $G_1\rightrightarrows G_0$, is given by the disjoint union 
\begin{equation}
\label{G0}
\uniondisjointe_{I\text{ maximal cone}}\C^{\vert I\vert}\times \Torus^{d-\vert I\vert}
\end{equation}
For any path $\mathcal I$ as above, we define
\begin{equation}
\label{G1component}
\tilde \Gamma_{\mathcal I}\times \C^{\vert I_1\cap \hdots\cap I_k\vert}\times\Torus^{\vert I_1\setminus (I_2\cup\hdots\cup I_{k})\vert}\times \C^{d-\vert I\vert}\rightrightarrows G_0
\end{equation}
The source map is the projection onto the $I_1$ component of $G_0$ and the target map is the compositions of maps \eqref{holonomyIexp} and \eqref{holonomyII} corresponding to $\mathcal I$.

Finally, take the disjoint union of all components \eqref{G1component} and divide by the following equivalence relation. Two points are equivalent if they have same source and target and correspond to two holonomy morphisms \eqref{holonomycomp} which are equal as holomorphic self-maps of $\C^{\vert I_1\cap \hdots\cap I_k\vert}\times\Torus^{\vert I_1\setminus (I_2\cup\hdots\cup I_{k})\vert}\times \C^{d-\vert I\vert}$. This gives $G_1$.

We have
\begin{theorem}
	\label{theoremholgroupoid}
	The stackification of the holonomy groupoid $G_1\rightrightarrows G_0$ over $\mathfrak A$ is $\nctoric{\Delta}{\Gamma}{v}$.
\end{theorem} 
In other words, $\nctoric{\Delta}{\Gamma}{v}$ is the leaf stack of the foliation defined by \eqref{completeactionnp}. This can be reformulated as follows. For each maximal cone $\mathcal I$, the object
	\begin{equation*}
\begin{tikzcd}
\C^{\vert I\vert}\times \Torus^{d-\vert I\vert}\times E(\Gamma_I) \arrow[r,"m_I"]\arrow[d,"pr"'] &\C^{\vert I\vert}\times \Torus^{d-\vert I\vert}\\
\C^{\vert I\vert}\times \Torus^{d-\vert I\vert} & &
\end{tikzcd}
\end{equation*}
defines, via Yoneda's lemma, a morphism from $\C^{\vert I\vert}\times \Torus^{d-\vert I\vert}$ to $Q_I$. And all these morphisms combine into a morphism, say $\sl m$ from $G_0$ to $\nctoric{\Delta}{\Gamma}{v}$.
\begin{corollary}
	\label{coratlas}
	The stack morphism $\text{\sl m}\ :\ G_0\to\nctoric{\Delta}{\Gamma}{v}$ is an atlas for $\nctoric{\Delta}{\Gamma}{v}$, that is the diagram
	\begin{equation}
	\label{CDatlas}
	\begin{tikzcd}
	G_1\arrow[d,"t"']\arrow[r,"s"]\arrow[dr,phantom,"\scriptstyle{\square}",very near start, shift left=0ex]&G_0\arrow[d,"\text{\sl m}"]\\
	G_0\arrow[r,"\text{\sl m}"']&\nctoric{\Delta}{\Gamma}{v}
	\end{tikzcd}
	\end{equation}
	2-commutes.
\end{corollary}
%
%
%
\begin{example}
	\label{exampleGIT}
	Let 
	\begin{equation*}
	v_1=e_1\qquad v_2=e_2\qquad v_3=-e_1-e_2\qquad v_4=-e_1\qquad v_5=-e_2
	\end{equation*}
	be the fan of the blow-up of $\mathbb P^2$ at two invariant points (or equivalently that of $\mathbb P^1\times\mathbb P^1$ at one invariant point) and let $h$ sending $x\in\Z^5$ to $\sum x_iv_i$ in $\Z^2$. Straightforward computations show that 
	\begin{equation*}
	\begin{aligned}
	A_1=(1,1,0)\qquad A_2=(1,0,1)\qquad &A_3=(1,0,0)\\
	\text{ and }\quad &A_4=(0,1,0)\qquad A_5=(0,0,1)
	\end{aligned}
	\end{equation*}
	is a Gale transform of $(v_1,\hdots,v_5)$ and 
	\begin{equation*}
	\begin{aligned}
	\mathscr{S}=\C^5\setminus\big (\{z_1=&z_3=0\}\cup\{z_1=z_4=0\}\\
	&\cup \{z_2=z_3=0\}\cup\{z_2=z_5=0\}\cup\{z_3=z_5=0\}\big )
	\end{aligned}
	\end{equation*}
	The (classical) quotient of $\mathscr{S}$ by the action \eqref{completeactionnp} is the blow-up of $\mathbb P^2$ at two invariant points, that is $\nctoric{\Delta}{\Gamma}{v}$. We recover the GIT construction of classical toric varieties, see also \cite{meersseman2004holomorphic}. The quotient stack of $\mathscr{S}$ by \eqref{completeactionnp}, that is $\nctoricgerbe{\Delta}{h}{J}$, is a gerbe with band $\Z^3$ over the blow-up of $\mathbb P^2$ at two invariant points. Observe that we cannot obtain a gerbe with smaller band.
	
	Taking $-\alpha e_1-\beta e_2$ with $\alpha$ and $\beta$ positive as new definition of $v_3$ but letting $h$ unchanged, a new Gale transform is now
	\begin{equation*}
	\begin{aligned}
	A_1=(\alpha,1,0)\qquad A_2=(\beta,0,1)\qquad &A_3=(1,0,0)\\
	\text{ and }\quad &A_4=(0,1,0)\qquad A_5=(0,0,1)
	\end{aligned}
	\end{equation*}
	and $\mathscr{S}$ is unchanged. We obtain Quantum versions of the blow-up of $\mathbb P^2$ at two invariant points, resp. the gerbe with band $\Z^3$ over the blow-up of $\mathbb P^2$ at two invariant points as $\nctoric{\Delta}{\Gamma}{v}$, resp. $\nctoricgerbe{\Delta}{h}{J}$.
\end{example}
\section{LVMB Theory}
\label{lvmtheory}

We give a short survey of LVM manifolds, see \cite{de1997new}, \cite{meerssemanthesis}, \cite{meersseman2000} and \cite{meersseman2004holomorphic}, and LVMB manifolds, see \cite{Bosio}. 
\subsection{LVM manifolds}
\label{lvm}
Given $\Lambda=(\Lambda_1,\hdots ,\Lambda_n)$ a $n$-uple of vectors of $\mathbb C^m$ fulfilling
\begin{enumerate}
\item the {\it Siegel condition}: $0$ belongs to the convex hull $\mathscr H(\Lambda)$ of the vectors $\Lambda_i$ in $\mathbb C^n$.
\item the {\it Weak Hyperbolicity condition}: take $I$ a subset of $\{1,\hdots,n\}$ and let $\Lambda_I$ be the corresponding set. Then, if $0\in\mathscr H(\Lambda_I)$, we must have $\text{Card }I\geq 2m+1$.
\end{enumerate}
We associate to it a LVM manifold $N_\Lambda$ constructed as follows. Define 
\begin{equation}
\label{SLambda}
\mathcal S:=\{z\in\mathbb C^n \quad\vert\quad 0\in\mathscr H(\Lambda_{I_z})\}
\end{equation}
where $I_z$ is defined as in \eqref{Iz}.

Then $N_\lambda$ is the quotient of the projectivization $\mathbb P(\mathcal S)$ by the holomorphic action
\begin{equation}
\label{actionLVM}
(T,[z])\in\mathbb C^m\times \mathbb P(\mathcal S)\longmapsto \big [z_i\exp \langle \Lambda_i,T\rangle\big ]_{i=1,\hdots,n}
\end{equation}
 It is a compact complex manifold of dimension $n-m-1$, which is either a $m$-dimensional compact complex torus (for $n=2m+1$) or a non K\"ahler manifold (for $n>2m+1$).

\subsection{LVMB manifolds}
\label{LVMB}

In \cite{Bosio}, F. Bosio gives a generalization of the previous construction. The idea is to relax the weak hyperbolicity and Siegel conditions on $\Lambda$ and to look for all the subsets $\mathcal S$ of $\C^n$ such that action \eqref{actionLVM} is free and proper. 

To be more precise, let $n\geq 2m+1$ and let $\Lambda=(\Lambda_1,\hdots,\Lambda_n)$ be a configuration of $n$ vectors in $\C^m$. 
Let also $\mathscr E$ be a non-empty set of subsets of $\{1,\hdots,n\}$ of cardinal $2m+1$ and set
\begin{equation}
\label{SB}
\mathcal S=\{z\in\C^n\mid  I_z\supset E\text{ for some }E\in\mathscr E\}
\end{equation}
Assume that 
\begin{enumerate}
\item[(i)] For all $E\in\mathscr E$, the real affine hull of $(\Lambda_i)_{i\in E}$ is the whole $\C^m$.
\item[(ii)] For all couples $(E,E')\in\mathscr E\times\mathscr E$, the convex hulls $\mathscr H((\Lambda_i)_{i\in E})$ and $\mathscr H((\Lambda_i)_{i\in E'})$ have non-empty interior with non-empty intersection.
\item[(iii)] For all $E\in\mathscr E$ and for every $k\in\{1,\hdots,n\}$, there exists some $k'\in E$ such that $E\setminus\{k'\}\cup\{k\}$ belongs to $\mathscr E$.
\end{enumerate}
Then, action \eqref{actionLVM} is free and proper \cite[Th\'eor\`eme 1.4]{Bosio}. We still denote it by $N_\Lambda$ althought it also depends on the choice of $\mathcal S$. As in the LVM case, it is a compact complex manifold of dimension $n-m-1$, which is either a $m$-dimensional compact complex torus (for $n=2m+1$) or a non k\"ahler manifold (for $n>2m+1$).

Assume now that $(\Lambda_1,\hdots,\Lambda_n)$ is an admissible configuration. Let
\begin{equation}
\label{ELVM}
\mathscr E=\{I\subset\{1,\hdots, n\}\mid 0\in\mathscr H((\Lambda_i)_{i\in I}\},
\end{equation}
then \eqref{SB} and \eqref{SLambda} are equal, the previous three properties are satisfied and the LVMB manifold is exactly the LVM manifold of section \ref{lvm}.

We say that $\Lambda_i$, or simply $i$, is an {\it indispensable point} if every point $z$ of $\mathcal S$ satisfies $z_i\not =0$. We denote by $k$ the number of indispensable points.

In the sequel, a {\it LVMB datum} or {\it LVMB configuration} is a couple $(\mathcal S,\Lambda)$ satisfying the above hypotheses. It is a  {\it LVM datum} or {\it LVM configuration} if $\mathcal S$ satisfies \eqref{ELVM} in addition. In that case, $\mathcal S$ is completely determined by $\Lambda$.

\subsection{LVMB manifolds as equivariant compactifications}
\label{LVMBec}
Observe that  $(\mathbb C^*)^n$ acts by multiplication on $\mathcal S$ with an open and dense orbit, making it a toric variety. This action commutes with projectivization and with \eqref{actionLVM}, making of $N_\Lambda$ an equivariant compactification of an abelian Lie group, say $G_\Lambda$ so are objects of the category $\mathfrak{G}$. A straightforward computation shows the following \cite[p.27]{meerssemanthesis}

\begin{proposition}
\label{Glattice}
Assume that 
\begin{equation}
\label{firstmcondition}
\text{\rm rank}_{\mathbb C}
\begin{pmatrix}
\Lambda_1 &\hdots &\Lambda_{m+1}\cr
1 &\hdots &1
\end{pmatrix}
=m+1.
\end{equation}
Then $G_\Lambda$ is isomorphic to the quotient of $\mathbb C^{n-m-1}$ by the $\mathbb Z^{n-1}$ abelian subgroup generated by
$(Id, B_\Lambda A_\Lambda^{-1})$ where
\begin{equation}
\label{Alattice}
A_\Lambda=^t\kern-4pt(\Lambda_2-\Lambda_1,\hdots,\Lambda_{m+1}-\Lambda_1)
\end{equation}
and
\begin{equation}
\label{Blattice}
B_\Lambda= ^t\kern-4pt(\Lambda_{m+2}-\Lambda_1,\hdots,\Lambda_{n-1}-\Lambda_1).
\end{equation}
\end{proposition}

\begin{remark}
It is easy to prove that 
\begin{equation*}
\text{\rm rank}_{\mathbb C}
\begin{pmatrix}
\Lambda_1 &\hdots &\Lambda_n\cr
1 &\hdots &1
\end{pmatrix}
=m+1.
\end{equation*}
(cf. {\rm \cite[Lemma 1.1]{meersseman2004holomorphic}} in the LVM case). Hence, up to a permutation, condition \eqref{firstmcondition} is always fulfilled.
\end{remark}

\begin{definition}
	\label{defGbiholo}
	We say that $N_\Lambda$ and $N_{\Lambda'}$ are {\it $G$-biholomorphic} if there exists a Lie group isomorphism between $G_\Lambda$ and $G_{\Lambda'}$ which extends as a $(G_\Lambda,G_{\Lambda'})$-equivariant biholomorphism between $N_\Lambda$ and $N_{\Lambda'}$.
	
	More generally, a {\it $G$-morphism} is a holomorphic map from $N_\Lambda$ to $N_{\Lambda'}$ which restricts to a Lie morphism from $G_\Lambda$ to $G_{\Lambda'}$.
\end{definition}

We notice the following important fact.

\begin{lemma}
	\label{lemmaaffindLVMB}
	Let $(\mathcal S,\Lambda)$ be a LVMB data. Then, for any $A\in\text{GL}_m(\C)$ and $B\in\C^m$, the couple $(\mathcal S,A\Lambda+B)$ is also a LVMB datum. Moreover, $N_\Lambda$ and $N_{A\Lambda+B}$ are $G$-biholomorphic.
\end{lemma}

\begin{proof}
	The conditions on $\mathscr{E}$ are affine conditions, hence invariant when replacing $\Lambda$ with $A\Lambda+B$. Then observe that 
	\begin{equation*}
	\big [z_i\exp \langle \Lambda_i,AT\rangle\big ]_{i=1,\hdots,n}=\big [z_i\exp \langle \kern 2pt ^t\kern -2ptA\Lambda_i+B,T\rangle\big ]_{i=1,\hdots,n}
	\end{equation*}
	hence the identity descends as a biholomorphism between $N_\Lambda$ and $N_{A\Lambda+B}$. This is obviously a $G$-biholomorphism.
\end{proof}

\subsection{The associated polytope of a LVM manifold}
\label{asspol}

In this section, $N_\Lambda$ is a LVM manifold.
The manifold $N_\Lambda$ embeds in $\mathbb P^{n-1}$ as the $C^\infty$ submanifold
\begin{equation}
\label{Nsmooth}
\mathscr N=\{[z]\in\mathbb P^{n-1}\quad\vert\quad \sum_{i=1}^n\Lambda\vert z_i\vert ^2=0\}.
\end{equation}
It is crucial to notice that this embedding is not arbitrary but has a clear geometric meaning. Indeed, it is proven in \cite{meerssemanthesis} that action \eqref{actionLVM} induces a foliation of $\mathcal S$; that every leaf admits a unique point closest to the origin (for the euclidean metric); and finally that \eqref{Nsmooth} is the projectivization of the set of all these minima\footnote{This is a sort of non-algebraic Kempf-Ness Theorem.}. So we may say that this embedding is canonical. 
\vspace{5pt}\\
The maximal compact subgroup $(\mathbb S^1)^n\subset (\mathbb C^*)^n$ acts on $\mathcal S$, and thus on $N_\Lambda$. This action is clear on the smooth model \eqref{Nsmooth}. Notice that it reduces to a $(\mathbb S^1)^{n-1}$ since we projectivized everything.
\vspace{5pt}\\
The quotient of $N_\Lambda$ by this action is easily seen to be a simple convex polytope of dimension $n-2m-1$, cf. \cite{meersseman2000} and \cite{meersseman2004holomorphic}. Up to scaling, it is canonically identified to 
\begin{equation}
\label{KLambda}
K_\Lambda:=\{r\in(\mathbb R^+)^n\quad\vert\quad \sum_{i=1}^n\Lambda r_i=0,\ \sum_{i=1}^nr_i=1\}.
\end{equation}
It is important to have a description of $K_\Lambda$ as a convex polytope in $\mathbb R^{n-2m-1}$. This can be done as follows. Take a Gale diagram of $\Lambda$, that is a basis of solutions $(v_1,\hdots, v_n)$ over $\mathbb R$ of the system
\begin{equation}
\label{systemS}
\left\{
\begin{aligned}
\sum_{i=1}^n\Lambda_ix_i&=0\cr
\sum_{i=1}^nx_i&=0
\end{aligned}
\right .
\end{equation}
Take also a point $\epsilon$ in $K_\Lambda$. This gives a presentation of $K_\Lambda$ as
\begin{equation}
\label{KLambdaproj}
\{x\in\mathbb R^{n-2m-1}\quad\vert\quad \langle x,v_i\rangle\geq -\epsilon_i\text{ for }i=1,\hdots, n\}
\end{equation}
This presentation is not unique. Indeed, taking into account that $K_\Lambda$ is unique only up to scaling, we have
\begin{lemma}
\label{Kupto}
The projection \eqref{KLambdaproj} is unique up to action of the affine group of $\mathbb R^{n-2m-1}$.
\end{lemma}
On the combinatorial side, $K_\Lambda$ has the following property. A point $r\in K_\Lambda$ is a vertex if and only if the set $I$ of indices $i$ for which $r_i$ is zero is maximal, that is has $n-2m-1$ elements. Moreover, we have
\begin{equation}
\label{Kcombinatoire}
r\text{ is a vertex }\iff \mathcal S\cap \{z_i=0\text{ for }i\in I\}\not =\emptyset\iff 0\in\mathscr H(\Lambda_{I^c})
\end{equation}
for $I^c$ the complementary subset to $I$ in $\{1,\hdots,n\}$.  This gives a numbering of the faces of $K_\Lambda$ by the corresponding set of indices of zero coordinates. To be more precise, we have
\begin{equation}
\label{Knumbering}
\begin{aligned}
&J\subset\{1,\hdots,n\}\text{ is a face of codimension Card }J\cr
&\iff \mathcal S\cap \{z_i=0\text{ for }i\in J\}\not =\emptyset\iff 0\in\mathscr H(\Lambda_{J^c})
\end{aligned}
\end{equation}
In particular, $K_\Lambda$ has $n-k$ facets. Observe, moreover, that the action \eqref{actionLVM} fixes $\mathcal S\cap \{z_i=0\text{ for }i\in J\}$, hence its quotient defines a submanifold $N_J$ of $N_\Lambda$ of codimension $\text{Card J}$.

Also, \eqref{Knumbering} implies that
\begin{equation}
\label{Sdelta}
\mathcal S=\{z\in\C^n\mid I_z^c\text{ is a face of } K_\Lambda\}
\end{equation}

\subsection{The canonical foliation of a LVMB manifold}
\label{F} 
Going back to $G_\Lambda$, we see that its Lie algebra is generated by the linear vector fields $z_i\partial/\partial z_i$ for $i=1,\hdots, n$. Due to the quotient by \eqref{actionLVM}, they only generate a vector space of dimension $n-m-1$ as needed. Amongst these $n-m-1$ linearly independent vector fields, we can find $m$ of them which extend to $\mathcal S$ without zeros and which generates a locally free action of $\mathbb C^m$ onto $\mathcal S$. For example, we can take the vector fields
\begin{equation}
\label{etavf}
\eta_i(z)=\left\langle \text{\rm Re }\Lambda_i,\sum_{j=1}^nz_j\dfrac{\partial}{\partial z_j}\right\rangle,
\end{equation} 
cf. \cite[Theorem A]{meersseman2004holomorphic} (this is proved only in the LVM case, but the generalization is obvious).

We denote by $\mathscr F_\Lambda$ the foliation induced by this action. It is easy to check that $\mathscr F_\Lambda$ is independent of the choice of vector fields. Indeed, changing the vector fields just means changing the parametrization of $\mathscr F_\Lambda$, that is changing the $\mathbb C^m$-action by taking a different basis of $\mathbb C^m$. The leaf space of $\F_\Lambda$ is the quotient space of $\mathbb P(\mathcal S)$ by the action
\begin{equation}
\label{totalaction}
 [z]\cdot (T,S):=\left [z_i\exp (\langle \text{\rm Re }\Lambda_i,T\rangle+\langle\text{\rm Im }\Lambda_i, S\rangle)\right ]_{i=1}^n
\in \mathbb P(\mathcal S)
\end{equation}

Assume now that we are in the LVM case. Pull back the Fubini-Study form of $\mathbb P^{n-1}$ to the embedding \eqref{Nsmooth}. This is the {\it canonical Euler form} $\omega$ of $\Lambda$, as defined in \cite{meersseman2004holomorphic}. It is a representative of the Euler class of a particular $\mathbb S^1$-bundle associated to $\mathscr N_\Lambda$, hence the name. Then $\mathscr F_\Lambda$ is transversely k\"ahler with transverse k\"ahler form $\omega$. For our purposes, we will not focus on $\omega$ but on the ray $\mathbb R^{>0}\omega$ it generates (this is also the case in \cite{meersseman2004holomorphic}).

Going back to the general case, recall that $\Lambda$ {\it fulfills  condition} (K) if \eqref{systemS} admits a basis of solutions with integer coordinates; and that $\Lambda$ {\it fulfills condition} (H) if \eqref{systemS} does not admit any solution with integer coordinates. If condition (K) is fulfilled, then $\mathscr F_\Lambda$ is a foliation by compact complex tori and the quotient space is a toric orbifold, see \cite{meersseman2004holomorphic} which contains a thorough study of this in the LVM case and see \cite{CFZ} for the LVMB case.
\vspace{5pt}\\
We just note here that, in the LVM case, even if condition (K) is not satisfied, the foliation $\mathscr F_\Lambda$ has some compact orbits. Indeed, let $I$ be a vertex of $K_\Lambda$. Then, by \eqref{Kcombinatoire}, $0$ belongs to $\mathscr H(\Lambda_{I^c})$, so by \cite[Lemma 1.1]{meersseman2004holomorphic},
\begin{equation}
\label{secondmcondition}
\text{\rm rank}_{\mathbb C}
\begin{pmatrix}
\Lambda_{i^c_1} &\hdots &\Lambda_{i^c_{2m+1}}\cr
1 &\hdots &1
\end{pmatrix}
=m+1.
\end{equation}
Hence, up to performing a permutation, we may assume at the same time \eqref{firstmcondition} and 
\begin{equation}
\label{inter}
I\cap \{1,\hdots, m+1\}=\emptyset.
\end{equation}
We have
\begin{proposition}
\label{verticesorbits}
For each vertex $I$ of $K_\Lambda$, the corresponding submanifold $N_I$ is a compact complex torus of dimension $m$ and is a leaf of $\mathscr F$. Moreover, assume that $\Lambda$ satisfies \eqref{firstmcondition} and \eqref{inter}. Then, letting $B_I$ denote the matrix obtained from \eqref{Blattice} by erasing the rows $\Lambda_i-\Lambda_1$ for $i\in I$, the torus $N_I$ is isomorphic to the torus of lattice $(Id,B_IA_\Lambda^{-1})$. 
\end{proposition}

\section{Quantum Torics and LVMB manifolds}
\label{LVMBdonnedec}
In this section, we investigate the relationship between LVMB manifolds and complete simplicial calibrated Quantum Torics. It culminates in an equivalence category which is precisely stated and proved in Theorem \ref{thmLVMBdec}.

\subsection{LVMB manifolds from Quantum Torics}
\label{LVMBQT}
We slightly modify the construction of Section \ref{QGIT} to construct LVMB manifolds from some special {\sl complete} simplicial Quantum varieties.

\begin{warning}
	\label{warLVMB}
	In this Section, all LVMB configurations count $n+1$ vectors. 
\end{warning}

We start with the following definition.

\begin{definition}
	\label{defcentered}
	Let $(\Delta,h)$ be a calibrated Quantum Fan. We say that $(\Delta,h)$ is {\it even} if  
	\begin{enumerate}
	 \item[i)] It is a complete simplicial calibrated Quantum Fan of maximal length.
	 \item[ii)] $n-d$ is even.
	\end{enumerate}
\end{definition}

As usual, set $v_i=h(e_i)$ for $i=1,\hdots,n$. We {\it add} the vector
\begin{equation}
\label{addedvector}
v_{n+1}=-v_1-\hdots -v_n
\end{equation}
and consider the extended configuration $\underline{v}=(v_1,\hdots,v_{n+1})$.

\begin{remark}
	\label{rkBZ}
	The couple $(\underline{v},\Delta)$ is exactly a balanced triangulated vector configuration as defined in \cite{BZ}. 
\end{remark} 

As in Section \ref{QGIT} and as in \cite{meerssemanthesis}, \cite{meersseman2000}, \cite{meersseman2004holomorphic} and \cite{BZ}, we perform a Gale transform of $\underline{v}=(v_1,\hdots,v_{n+1})$. Note however that this is in a slightly different sense as in Section \ref{QGIT}. We choose some vectors $A=(A_1,\hdots,A_{n+1})$ of $\R^{n-d}$ such that
\begin{equation}
\label{GT}
\left .
\begin{aligned}
&\sum_{i=1}^{n+1}x_iv_i=0\\
&\sum_{i=1}^{n+1}x_i=0
\end{aligned}
\right \}
\iff \left\{
\begin{aligned}
x_1=&\langle A_1,t\rangle\\
\vdots&\\
x_{n+1}=&\langle A_{n+1},t\rangle
\end{aligned}
\right .
\quad\text{ for some }t\in\R^{n-d}
\end{equation}

\begin{remark}
	\label{Galemodulo}
	The choice of $A$ is unique up to action of $\text{GL}_{n-d}(\R)$. Also, the same $A$ is a Gale transform for any {\sl affine} transform of $v$, thanks to the balanced condition. Comparing with the original Gale transform used in Section \ref{QGIT}, this one could be thought of as an affine version.
\end{remark}

\begin{remark}
	\label{BZGale}
	A slightly different definition of Gale transform is used in \cite{BZ}. 
\end{remark}
The relationship between the two notions of Gale transforms goes as follows.
\begin{lemma}
	\label{lemmaGaleetGaleaff}
	A set $(A_1,\hdots,A_{n+1})$ of $n+1$ vectors in $\R^{n-d}$ satisfies {\rm \eqref{GT}} if and only if the set $(A_1-A_{n+1},\hdots ,A_n-A_{n+1})$ of $n$ vectors in $\R^{n-d}$ satisfies {\rm \eqref{GTdec}}.
\end{lemma}

\begin{proof}
	The proof is a straightforward calculation.
	\begin{equation*}
	\begin{aligned}
	\sum_{i=1}^n \langle A_i-A_{n+1}, T\rangle v_i=&\sum_{i=1}^{n+1} \langle A_i-A_{n+1}, T\rangle v_i\\
	=&\sum_{i=1}^{n+1} \langle A_i, T\rangle v_i-\langle A_{n+1},T\rangle \sum_{i=1}^{n+1}v_i\\
	=&\sum_{i=1}^{n+1} \langle A_i, T\rangle v_i
	\end{aligned}
	\end{equation*}
\end{proof}

We define 
\begin{equation}
\label{Sbis}
\mathcal S_\Delta=\{z\in\C^{n+1}\mid \{1,\hdots,n+1\}\setminus I_z\text{ is a cone of }\Delta\}
\end{equation}
Let $v=(v_1,\hdots,v_n)$ and define $\mathscr{S}$ as in \eqref{S}. Observe that \eqref{Sbis} is the same definition than that of $\mathscr{S}$ except that here we are in $\C^{n+1}$, not in $\C^n$. 
We have indeed
\begin{lemma}
	\label{lemmaSetS'}
	The affine toric variety $\mathcal{S}_\Delta$ is isomorphic to $\mathscr{S}\times\C^*$.
\end{lemma}
\begin{proof}
	The identity mapping from $\C^{n+1}$ to $\C^n\times\C$ does the job.
\end{proof}
Let $\C^{n-d}\ni T$ act holomorphically on $\mathbb P(\mathcal{S}_\Delta)\ni [z]$ through
\begin{equation}
\label{completeaction}
T\cdot [z]:=\left [ 
z_ie^{\langle A,T\rangle}
\right ]_{i=1} ^{n+1}
\end{equation}
(compare with \eqref{completeactionnp}).

Since the length $n-d$ is even, we set $m=(n-d)/2$ and we consider $\R^{n-d}$ as the complex $\C^m$ through
\begin{equation}
\label{isoRC}
(x_1,\hdots, x_{2m})\in\R^{n-d}\longmapsto (x_1+ix_2,\hdots,x_{2m-1}+ix_{2m})\in\C^m
\end{equation}
In particular, we denote by $\Lambda_i$, resp. $\Lambda$, the image of $A_i$, resp. $A$, through this isomorphism.

We define the $\Lambda$-action onto the projectivization of $\mathcal{S}_\Delta$ as in \eqref{actionLVM}.
 
\begin{lemma}
\label{LVM+B}
The couple $(\mathcal{S}_\Delta,\Lambda)$ is a LVMB datum. 
\end{lemma}

\begin{proof}
This is exactly \cite[Proposition 2.1]{BZ}. 
\end{proof}

As usual, we let $N_\Lambda$ be the LVMB manifold and $\F_\Lambda$ be the associated foliation, cf. section \ref{lvmtheory}. We notice the following fact
\begin{lemma}
	\label{lemmaindetvir}
	Let $1\leq i\leq n$. The vector $v_i$ is a virtual generator if and only if $\Lambda_i$ is indispensable. Moreover, $\Lambda_{n+1}$ is also an indispensable point.
\end{lemma}

\begin{proof}
	Since $(\Delta,h)$ is maximal, $v_i$ is a virtual generator if and only if it is not a $1$-cone generator. By \eqref{Sbis}, this is the case if and only if $\mathcal S_\Delta$ does not intersect $\{z_i=0\}$, i.e. if and only if $\Lambda_i$ is indispensable. Analogously, since $\mathcal S_\Delta$ does not intersect $\{z_{n+1}=0\}$, $\Lambda_{n+1}$ is also an indispensable point.
\end{proof}

In particular, we obtain a LVMB manifold with at least $\Lambda_{n+1}$ indispensable.

\subsection{Quantum Torics from LVMB manifolds}
\label{QTLVMB}
Let $(\mathcal S,\Lambda)$ be a LVMB datum. We assume that $\Lambda$ is balanced, i.e. that
\begin{enumerate}[i)]
\item $\sum_{i=1}^{n+1}\Lambda_i=0$.
\item $\Lambda_{n+1}$ is indispensable.
\end{enumerate}
Since $\mathcal S$ is an affine toric variety as well as its projectivization thanks to the existence of an indispensable point; and since the action \eqref{completeaction} preserves the toric structure, we may define two stacks over $\mathfrak A$ from this data. Firstly, the quotient stack $[N_\Lambda/\C^m]$, where $\C^m$ is the composition of the flows of the vector fields \eqref{etavf}. Secondly, the leaf stack $[N_\Lambda/\mathscr F_\Lambda]$, that is the stackification by torsors of the holonomy groupoid of $\mathcal F_\Lambda$.

There is a natural morphism from  $[N_\Lambda/\mathbb C^m]$ onto $[N_\Lambda/\mathscr F_\Lambda]$ with discrete fibers. Hence both stacks admit a presentation by an \'etale complex Lie groupoid.
\begin{remark}
The set of morphisms of a holonomy groupoid may be non-Hausdorff. However, this is never the case here, because of the Hausdorff presentation $\mathbb C^m\times N_\Lambda\rightrightarrows N_\Lambda$.
\end{remark}

We claim that both stacks are Quantum Torics. To see this, first replace the complex configuration $\Lambda$ with the real configuration $A$ using the isomorphism \eqref{isoRC}. Then, we need to define some configuration $\overline{v}$ satisfying \eqref{GT}. But it is a classical fact that the Gale transform is invertible. In other words,

\begin{lemma}
	\label{lemmaGalereverse}
	The configuration $A$ is a Gale transform of $\overline{v}$, i.e. satisfies {\rm \eqref{GT}}, if and only if $\overline{v}$ is a Gale transform of $A$, i.e; satisfies
	\begin{equation}
	\label{GT2}
	\left .
	\begin{aligned}
	&\sum_{i=1}^{n+1}x_iA_i=0\\
	&\sum_{i=1}^{n+1}x_i=0
	\end{aligned}
	\right \}
	\iff \left\{
	\begin{aligned}
	x_1=&\langle v_1,s\rangle\\
	\vdots&\\
	x_{n+1}=&\langle v_{n+1},s\rangle
	\end{aligned}
	\right .
	\quad\text{ for some }s\in\R^{d}
	\end{equation}
\end{lemma}

\begin{proof}
	This is a straightforward computation. Both \eqref{GT} and \eqref{GT2} are equivalent to
	\begin{equation*}
	\sum_{i=1}^{n+1}\langle A_i,t\rangle\langle v_i,s\rangle=0
	\quad\text{ for all }t\in\R^{n-d},\ s\in\R^{d}
	\end{equation*}
	and
	\begin{equation*}
	\sum_{i=1}^{n+1}A_i=\sum_{i=1}^{n+1}v_i=0
	\end{equation*}
\end{proof}

Hence, using Lemma \ref{lemmaGalereverse}, let $\overline{v}$ be a Gale transform of $A$. Set 

\begin{equation}
\label{vfromLVMB}
\left\{
\begin{aligned}
	&v=(v_1,\hdots,v_n)\\
	&\Gamma=\Z v_1+\hdots+\Z v_{n+1}=\Z v_1+\hdots+\Z v_{n}\\
	&x\in\Z^{n}\longmapsto h(x):=\sum_{i=1}^{n}x_iv_i\in\Gamma
\end{aligned}
\right.
\end{equation}

Define also $\Delta$ by the equivalence
\begin{equation*}
\sigma_I\text{ is a cone of }\Delta \iff I\not\in \mathscr{E}
\end{equation*}
and finally let $J$ be the set of indispensable points of $\Lambda$ (minus $\Lambda_{n+1}$, since $J$ is a subset of $\{1,\hdots,n\}$). The following lemma is straightforward

\begin{lemma}
	\label{lemmagooddec}
	The calibrated Quantum fan $(\Delta,h)$ is even.  
\end{lemma}

As usual, let $\nctoricgerbe{\Delta}{h}{J}$ be the corresponding calibrated Quantum Torics and let $\nctoric{\Delta}{\Gamma}{v}$ its image through $\text{\slshape f}$.

\begin{theorem}
\label{mainGIT}
Let $(\mathcal S,\Lambda)$ and $(\Delta,h)$ as above. Then,
\begin{itemize}
\item [i)] The stack $[N_\Lambda/\mathscr F_\Lambda]$ is isomorphic to the Quantum Toric Variety $\nctoric{\Delta}{\Gamma}{v}$.
\item [ii)] The stack $[N_\Lambda/\mathbb C^m]$ is isomorphic to the calibrated Quantum Toric Variety $\nctoricgerbe{\Delta}{h}{J}$.
\end{itemize}
\end{theorem}

\begin{proof}
	We note that $[N_\Lambda/\mathbb C^m]$ is also given as the quotient stack $[\mathbb P(\mathcal S_\Delta)/\C^{2m}]$, where $\C^{2m}$ denotes the action \eqref{totalaction}. Using the isomorphism
	$$
	z\in\mathscr S\longmapsto [z; 1] \in\mathbb P(\mathcal S_\Delta)
	$$ 
	which is a reformulation of Lemma \ref{lemmaSetS'}, we may rewrite action \eqref{completeaction} on $\mathscr{S}$ as 
	\begin{equation}
	\label{actionnew}
	T\cdot z:=\left (
	z_iE({\langle A_i-A_{n+1},T\rangle})
	\right )_{i=1} ^n
	\end{equation}
	This is exactly action \eqref{completeactionnp} for the set $(A_2-A_{n+1},\hdots, A_n-A_{n+1})$. Thanks to Lemma \ref{lemmaGaleetGaleaff} and Theorem \ref{TheoremGITdec}, we obtain that this stack is isomorphic to $\nctoricgerbe{\Delta}{h}{J}$. This proves ii).\\
	As for i), rerunning the previous line of arguments for the leaf stack $[N_\Lambda/\F_\Lambda]$, we see that it is isomorphic, firstly, to the leaf stack associated to action \eqref{completeaction} on $\mathcal S_\Delta$, and, secondly, to the leaf stack of action \eqref{actionnew} on $\mathscr{S}$, and, finally, to $\nctoricgerbe{\Delta}{h}{J}$ by use of Theorem \ref{theoremholgroupoid}.
\end{proof}
\begin{remark}
	\label{rkindispoint}
	In the GIT construction of this section, the projective sets $\mathbb P(\mathcal S_\Delta)$ are always isomorphic to affine ones by Lemma \ref{lemmaSetS'}. Indeed this lemma tells us that the corresponding LVMB has always an indispensable point. This is due to the fact that at the beginning of our construction we {\it add} the vector $v_{n+1}$ before making a Gale transform, creating thus such an indispensable point $\Lambda_{n+1}$.
	
	Now, {\it starting from} any LVMB manifold $N_\Lambda$ and its canonical foliation $\mathscr F_\Lambda$, we may consider the leaf stack $[N_\Lambda/\mathscr F_\Lambda]$ and the quotient stack $[N_\Lambda/\mathbb C^m]$. If $\Lambda$ has at least an indispensable point, we obtain a quantum toric variety, resp. a calibrated quantum toric variety. But if it has no, we are in a slightly more general situation. For example, we may obtain a $\Z^2$-gerbe over the blow up of $\mathbb{P}^2$ at two invariant points, cf. \cite{meersseman2004holomorphic} and not a $\Z^3$-gerbe as in Example \ref{exampleGIT}.
	
	However, this does not produce really different stacks. As shown in the proof of Theorem \ref{mainGIT}, a stack $[N_\Lambda/\C^m]$ can also be obtained as $[\mathbb P(\mathcal S_\Delta)/\C^{2m}]$. But we may form the stack $[\mathcal S_\Delta/\C^{2m+1}]$ where the additional factor $t\in\C$ acts through $z\mapsto zE(t)$ gives a $\Z$-gerbe over $[\mathbb P(\mathcal S_\Delta)/\C^{2m}]$. And this stack is a calibrated quantum toric. The situation is even better for $[N_\Lambda/\mathscr{F}_\Lambda]$ since it will always give a quantum toric.
	
	Still it highlights a difference of nature between Quantum Torics' GIT construction of Section \ref{QGIT} and LVMB construction of Section \ref{lvmtheory}. The first one starts from an open set of the {\it affine} space, whereas the second uses an open set of the {\it projective} space.
\end{remark}

\begin{remark}
	\label{Battisti}
	In \cite{Battisti}, L. Battisti shows that a LVMB datum $(\mathscr E,\Lambda)$ is completely encoded in a pair $(E,\Delta)$, where $E$ is a $2m$ linear subspace of $\R^n$ and $\Delta$ a subfan of the fan of $\mathbb P^n$ such that the projection map $\pi : \R^n\to \R^n/E$ is injective on the support of $\Delta$, and such that the fan $\pi(\Delta)$ is complete in $\R^n/E$. This results can be recovered from Theorem \ref{mainGIT} and Remark \ref{rkindispoint} identifying $\Delta$ with the fan of $\mathbb P(\mathcal S)$, and $\pi(\Delta)$ with the associated Quantum Fan. 
\end{remark}

\begin{remark}
	\label{Ishida}
	In \cite{IshidaMax}, H. Ishida undertakes a thorough study of compact complex manifolds $X$ with a maximal torus action. Here the torus $T$ is a real torus - a product of circles - acting by biholomorphisms on $X$. Maximal means that the action is effective and that the standard inequality 
	\begin{equation*}
	\dim T_x+\dim T\leq \dim X
	\end{equation*}
	is an equality for some $x\in X$ ($T_x$ is the isotropy group at $x$). He proves that the category $\mathscr{C}_{max}$ of compact complex manifolds $X$ with a maximal torus action (with equivariant mappings as morphisms) is equivalent to a category $\mathscr{F}_{max}$ of rational fans $\Delta$ together with a projection $\pi$ onto a complete, maybe irrational, fan.  Complete toric varieties and LVMB manifolds belong to $\mathscr{C}_{max}$ and form the building blocks of compact complex manifolds $X$ with a maximal torus action. The combinatorial datum in $\mathscr{F}_{max}$ associated to a LVMB manifold is exactly Battisti's encoding described in Remark \ref{Battisti}. The difference between the LVMB case and the general case is that the rational fan is no more supposed to be a subfan of a fan of $\mathbb P^n$.
	
	Like LVMB manifolds, compact complex manifolds $X$ with a maximal torus action are endowed with a canonical holomorphic foliation and one can look at the leaf stack. It follows from Theorems 11.1 and 11.2 of \cite{IshidaMax} and Theorem 3.3 of \cite{IshidaTowards} that all these leaf stacks are Quantum Torics. The marked fan of \cite{IshidaTowards} is nothing else than the associated Quantum Fan and is given by the complete fan $\pi(\Delta)$. 
	
	In other words, leaf stacks of compact complex manifolds $X$ with a maximal torus action give atlases for Quantum Torics, even if they are not LVMB; but do not give quotients different from Quantum Torics.
\end{remark}

\subsection{Equivalence of category between LVMB and Quantum Torics}
\label{EqCat}
We begin with some definitions.
\begin{definition}
	\label{defmarked}
	A {\it marked LVMB datum} or {\it marked LVMB configuration} is a triple $(\mathcal S,\Lambda,i)$ where $(\mathcal S,\Lambda)$ is a LVMB datum and $i\in\{1,\hdots,n+1\}$ is an indispensable point.
\end{definition}
Given a marked LVMB datum $(\mathcal S,\Lambda,i)$, then the map
\begin{equation}
\label{Si}
z=(z_1,\hdots,z_{n+1})\in\mathcal S\longmapsto (z_1/z_i,\hdots,\widehat{z_i/z_i},\hdots z_{n+1}/z_i)\in\C^n
\end{equation}
descends as a toric isomorphism between $\mathbb P(\mathcal S)$ and a toric affine open subset of $\C^n$. Call this open set $\mathscr{S}_i$.

We denote by $J$ the set of indispensable points of $\mathscr{S}_i$, that is $j\in\{1,\hdots,n\}$ is indispensable if $\{w_j=0\}$ does not intersect $\mathscr S_i$.
\begin{definition}
	\label{defSmorphism}
	A {\it $\mathcal S$-morphism} between two marked LVMB data $(\mathcal S,\Lambda,i)$ and $(\mathcal S',\Lambda',i')$ is a a toric morphism $\Phi$ from $\mathscr{S}_i$ to $\mathscr{S'}_{i'}$ such that
	\begin{enumerate}[\rm i)]
		\item For $j\not \in J'$, the $j$-th component $\Phi_j$ is a monomial with integer coefficients in the variables $(w_i)_{i\not\in J}$.
		\item There exists a map $s$ from $J$ to $J'$ such that, for $j\in J'$, the $j$-th component $\Phi_{j}$ satisfies 
		\begin{equation*}
		\Phi_{j}(w)=\left \{
		\begin{aligned}
		\prod_{i\in s^{-1}(j)}&w_i\qquad\text{ if }s^{-1}(j)\not =\emptyset\cr
		&\ 1 \qquad \text{ if }s^{-1}(j)=\emptyset
		\end{aligned}
		\right .
		\end{equation*}
	\end{enumerate}
\end{definition}
Let $\mathscr{V}^*_{LVMB}$ be the category whose objects are marked LVMB data and whose morphisms are $\mathcal S$-morphisms. The main result of this section is the following theorem.

\begin{theorem}
	\label{thmLVMBdec}
	The category $\mathscr{V}^*_{LVMB}$ is equivalent to the full subcategory $\mathscr{Q}^{cal}_{even}$ of $\mathscr{Q}^{cal}$ formed by calibrated Quantum Torics associated to an even calibrated Quantum Fan.
\end{theorem}

\begin{proof}
	Thanks to Lemma \ref{lemmaaffindLVMB}, every marked LVMB datum is $\mathcal S$-isomorphic to a marked LVMB datum $(\mathcal S,\Lambda,n+1)$ such that $\Lambda$ is balanced. Moreover, identifying $\Lambda$ with $A$ as usual using \eqref{isoRC}, we may assume, performing another $\mathcal S$-isomorphism if necessary that
	\begin{equation}
	\label{Snormal}
	\det (A_1,\hdots,A_{n-d})\not =0
	\end{equation}
	Call {\it $\mathcal S$-normal} a marked LVMB datum $(\mathcal S,\Lambda,n+1)$ satisfying \eqref{Snormal} and with $\Lambda$ balanced. The previous remark can be reformulated by saying that the natural injection of the set of $\mathcal S$-normal configurations into the set of marked LVMB configurations yields an equivalence category between $\mathscr{V}^*_{LVMB}$ and the full subcategory of $\mathscr{V}^*_{LVMB}$ formed by the $\mathcal S$-normal LVMB data. Call $\mathscr{V}^{\mathcal S}_{LVMB}$ this subcategory.
	
	Define a functor $\text{\slshape g}$ from $\mathscr{V}^{\mathcal S}_{LVMB}$ to $\mathscr{Q}^{cal}_{even}$ as follows. Given $(\mathcal S,\Lambda,n+1)$ a $\mathcal S$-normal LVMB configuration, we associate to it the unique Gale transform $\underline{v}$ satisfying
	\begin{equation}
	\label{vnorm}
	v_{n-d+1}=e_1,\hdots, v_n=e_d
	\end{equation}
	in addition to \eqref{GT}. Then define $(v,\Gamma,h,\Delta)$ as in \eqref{vfromLVMB} and the lines below. We set 
	\begin{equation}
	\label{gobjects}
	\text{\slshape g}(\mathcal S,\Lambda,n+1)=\nctoricgerbe{\Delta}{h}{J}
	\end{equation}
	As for the morphisms, let $\Phi$ be a $\mathcal S$-morphism from $(\mathcal S,\Lambda,n+1)$ to $(\mathcal S',\Lambda',n+1)$. Since $\Phi$ is given by monomials (see condition ii) of Definition \ref{defSmorphism}), it is equivariant with respect to the total action \eqref{completeactionnp} and descends as a morphism $\mathcal L$ between the Quantum Torics $\nctoric{\Delta}{\Gamma}{v}=\text{\slshape fg}(\mathcal S,\Lambda,n+1)$ and $\nctoric{\Delta'}{\Gamma'}{v'}$ making the following diagram commutative
	\begin{equation}
	\label{cdLVMBcateqbis}
	\begin{tikzcd}[row sep=normal, column sep=scriptsize]
	\mathscr{S}_{n+1} \arrow[rr,"\mathscr{H}:=\Phi"] \arrow[dd,"{\text{\sl h}}"' near start] \arrow[dr]& & \mathscr{S'}_{n+1}\arrow[dr]\arrow[dd, crossing over,"{\text{\sl h'}}" near start]\\
	& N_\Lambda\arrow[dl,"{\text{\sl fg}}"']  & & N_{\Lambda'} \arrow[dl,"{\text{\sl fg}}"'] \\
	\nctoric{\Delta}{\Gamma}{v}\arrow[rr,"\mathscr{L}"] & & \nctoric{\Delta'}{\Gamma'}{v'} \\
	\end{tikzcd}
	\end{equation}
	Since $\mathscr{S}_{n+1}$, respectively $\mathscr{S'}_{n+1}$, is nothing else than the open set $\mathscr{S}$ of \eqref{S}, resp. $\mathscr{S'}$, and since the set of indispensable points of $\mathscr{S}_{n+1}$, respectively $\mathscr{S'}_{n+1}$, is the same that the set of virtual generators by Lemma \ref{lemmaindetvir}, then the couple $(\mathscr L,\mathscr{H})$ satisfies the conditions of Theorem \ref{thmdecmorphisms2} and thus defines a calibrated toric morphism $\mathscr{L}^{cal}$ between $\nctoricgerbe{\Delta}{h}{J}$ and $\nctoricgerbe{\Delta'}{h'}{J'}$. Set
	\begin{equation}
	\label{gmorphisms}
	\text{\slshape g}(\Phi,\phi)=\mathscr{L}^{cal}
	\end{equation}
	Now, starting with $\mathscr{L}^{cal}$ between $\nctoricgerbe{\Delta}{h}{J}$ and $\nctoricgerbe{\Delta'}{h'}{J'}$, we associate to it a unique couple $(\mathscr L,\mathscr{H})$ by Theorem \ref{thmdecmorphisms2}. Comparing Theorem \ref{thmdecmorphisms2} and Definition \ref{defSmorphism} shows that $\Phi$ is a $\mathcal S$-morphism.
	
	As a consequence, the functor $\text{\slshape g}$ is fully faithful. But it is also essentially surjective by Theorem \ref{mainGIT}. Hence it is an equivalence of categories.
\end{proof}

We emphasize that, in diagram \eqref{cdLVMBcateqbis}, there is no arrow between $N_\Lambda$ and $N_{\Lambda'}$. In other words, a $\mathcal S$ morphism does not descend as a $G$-morphism and two $\mathcal S$-isomorphic LVMB data may be associated to two non-biholomorphic LVMB manifolds.

For later use in Section \ref{Qmoduli}, we also define the more natural category $\mathscr{N}_{LVMB}$ of LVMB data and $G$-morphisms, and the full subcategory  $\mathscr{N}^*_{LVMB}$ of balanced LVMB data with at least one indispensable point. They correspond to LVMB manifolds which are either 2-connected or non simply-connected.  

\section{K\"ahlerianity}
\label{Qproj}
The Quantum Torics of this paper are constructed as stacks over the category of affine toric varieties. However, they are not algebraic stacks, recall Warning \ref{warningstacks}. This is due to the existence of isotropy groups isomorphic to some power of $\Z$. Of course, when the fan  $\Delta$ is rational, the isotropy groups are finite and the quantum torics $\nctoric{\Delta}{\Gamma}{v}$ is equivalent to a Toric Deligne-Mumford stack of \cite{BCS}, hence to an algebraic stack. But note that the atlas construction of Section \ref{Qatlas} does not recover the structure of algebraic stack because of the extensive use of the exponential map we make throughout the whole construction.

As an alternative to algebraicity, it is suggested in \cite[\S 4.6]{KLMV} that Quantum Torics are K\"ahler. This can be made precise as follows. In the complete case, thanks to Section \ref{LVMBdonnedec}, we can realize a Quantum Torics as the leaf stack of a LVMB manifold by choosing a calibration satisfying the extra-conditions detailed at the beginning of this section. 

When $\nctoric{\Delta}{\Gamma}{v}$ is equivalent to a classical complete simplicial toric variety, it has a natural K\"ahler structure, to wit K\"ahler metrics and $2$-forms. This structure can be pulled-back to the above LVMB manifold as a transversely K\"ahler metric/form. This motivates the following definition.

\begin{definition}
	\label{defkahler}
	A Quantum toric $\nctoric{\Delta}{\Gamma}{v}$ with $\Delta$ complete is {\it K\"ahler} if it is isomorphic to the leaf stack of a LVMB manifold with transversely K\"ahler canonical foliation, that is to say, it is the leaf stack of a LVM manifold.
\end{definition}

By a Theorem of Ishida \cite{Ishida}, the canonical foliation of a LVMB manifold $N$ is transversely k\"ahler if and only if $N$ is in fact a LVM manifold. Hence, we obtain 

\begin{theorem}
	\label{thkahler}
	A Quantum torics $\nctoric{\Delta}{\Gamma}{v}$ with $\Delta$ complete is {K\"ahler} if and only if $\Delta$ is polytopal.
\end{theorem}

\begin{remark}
	\label{rkprojective}
	It may seem natural to say that a Quantum torics $\nctoric{\Delta}{\Gamma}{v}$ with $\Delta$ complete is {\it projective} if there exists a toric embedding of $\nctoric{\Delta}{\Gamma}{v}$ in a Quantum projective space. Nevertheless, this definition is quite restrictive. Indeed, let $\mathscr{L}$ be an injective  toric morphism of a Quantum torics $\nctoric{\Delta}{\Gamma}{v}$ in some Quantum projective space of dimension $N$. Because of point iii) in Definition \ref{QFmorphismdef}, the associated linear mapping $L$ must send the generators $v_1,\hdots,v_p$ onto vectors of $\R^N$ with integer coordinates. In other words, $\Z v_1+\hdots+\Z v_n$ must be discrete in $\R^d$. 
	
	Therefore this definition only works for fans that become rational after replacing $\Gamma$ with the smallest additive subgroup in which the generators $v_i$ live. So, thinking of the identity map as a fan morphism from $(\Delta,\Z v_1+\ldots+\Z v_n,v)$ to $(\Delta,\Gamma,v)$ - a fan cover, this would be a cover in the classical case -, we obtain than a projective Quantum Torics is rational up to a fan cover.  
	
	All this shows that a more general setting is needed to understand projectivity.
\end{remark}

\section{Moduli Spaces}
\label{Qmoduli}
As a reward for all the efforts done to establish the foundations of Quantum Torics, we derive now several moduli spaces of Quantum Torics. Thanks to Theorems \ref{prisoQTV} and \ref{prisodecQTV}, they can be computed using only elementary linear algebra. It should be emphasized, however, that such results cannot be obtained without the develpment of an entirely functorial point of view. This is perhaps the most prominent reason for us to have defined quantum toric varieties as stacks rather than quasifolds or diffeological spaces (even though the diffeological space can recover the full stack in this case).

\subsection{Moduli Spaces of Calibrated Quantum Tori}
\label{QmoduliQT}
Let $n>d>0$. We are interested in the moduli space $\modulitorusgerbe{d,n}{\emptyset}$ of calibrated Quantum Tori $\nctorusgerbe{h}{\emptyset}$ with fixed $d$ and $n$ and empty set of virtual generators. Putting $h$ in normal form \eqref{h}, we see that it is completely determined by fixing the values of $\hbar(e_{d+1}),\hdots,\hbar(e_n)$. Hence it is given by
\begin{equation}
\label{moduliDQT}
\modulitorusgerbe{d,n}{\emptyset}=\{(v_{d+1},\hdots,v_n)\in (\R^{d})^{n-d}\}/\sim
\end{equation}
where two $(n-d)$-uples of vectors are equivalent if and only if there exists a calibrated Quantum Fan isomorphism $(L,H)$ between the corresponding $h$ functions. In particular, we endow $\modulitorusgerbe{d,n}{\emptyset}$ with the quotient topology coming from the euclidean topology of $(\R^{d})^{n-d}$.

Now, pick some $h=Id+\hbar$ and let $H\in \text{GL}_n(\Z)$. We look for some $h'$ in normal form \eqref{h} and some $L$ such that the diagram commutes
\begin{equation}
\label{CDLHmod}
\begin{tikzcd}
\Gamma\subset\R^d \arrow[r,"L"] &\Gamma'\subset \R^d\\
\Z^n \arrow[u,"h"]\arrow[r,"H"] &\Z^n \arrow[u,"h'"']
\end{tikzcd}
\end{equation} 
Observe that $L$ is determined by $H$, $h$ and $h'$, since
\begin{equation*}
Le_i=Lhe_i=h'He_i
\eqno{1\leq i\leq d}
\end{equation*}
Moreover,
\begin{equation*}
Lv_i=Lhe_i=h'He_i
\eqno{d+1\leq i\leq n}
\end{equation*}
Decomposing
\begin{equation}
\label{Hdecompo}
H=\begin{pmatrix}
H_1 &H_2 \\
H_3 &H_4
\end{pmatrix}
\end{equation}
and considering $\hbar$, respectively $\hbar'$, as the matrix $(\hbar(e_{d+1}),\hdots,\hbar(e_n))$, resp. $(\hbar'(e_{d+1}),\hdots,\hbar'(e_n))$, the previous equations yield
\begin{equation*}
L=H_1+\hbar'H_3
\end{equation*}
and, assuming $H_1+\hbar' H_3$ is invertible
\begin{equation*}
\hbar=L^{-1}(H_2+\hbar'H_4)=(H_1+\hbar' H_3)^{-1}(H_2+\hbar'H_4)
\end{equation*}
We thus have
\begin{proposition}
	\label{propmoduliQt}
	The moduli space $\modulitorusgerbe{d,n}{\emptyset}$ of calibrated Quantum Tori $\nctorusgerbe{h}{\emptyset}$ with fixed $d$ and $n$ and empty set of virtual generators is homeomorphic to the quotient of $(\R^d)^{n-d}$ by the following equivalence relation: $v$ and $v'$ are equivalent if there exists some $H\in\text{GL}_n(\Z)$ with $H_1+\hbar' H_3$ invertible such that
	\begin{equation}
	\label{actionmoduliQt}
	v'=v\cdot H=(H_1+(v_{d+1},\hdots, v_n) H_3)^{-1}(H_2+(v_{d+1},\hdots, v_n)H_4)
	\end{equation}
	where we use decomposition {\rm \eqref{Hdecompo}}.
\end{proposition}

\begin{example}
	\label{exmoduliQt}
	Let $n=2$ and $d=1$. In normal form, we write $\hbar (y)=ay$ for some $a\in\R$. Hence $\Gamma$ is equal to $\Z+a\Z$. Let
	\begin{equation*}
	H=\begin{pmatrix}
	p &r\\
	q &s
	\end{pmatrix}
	\in\text{GL}_2(\Z)
	\end{equation*}
	and observe that the condition $H_1+\hbar' H_3$ invertible translates into $p+aq$ is different from zero.
	
	Assume that $a$ is irrational. Then $p+aq$ is different from zero for all $(p,q)\in\Z^2\setminus\{0\}$, so \eqref{actionmoduliQt} defines an action of $\text{GL}_2(\Z)$ on $a$ through the equality
	\begin{equation}
	\label{actionmoduliQt2}
	a\cdot H=\dfrac{r+sa}{p+qa}
	\end{equation}
	Assume that $a$ is a non-zero rational number. Set $a=p/q$ with $\text{GCD}(p,q)$ equal to one. By Bezout, there exist some integers $r$ and $s$ such that $ps+rq=1$. Then
	\begin{equation*}
	\dfrac{p}{q}\cdot 
	\begin{pmatrix}
	r &-p\\
	s &q
	\end{pmatrix}=\dfrac{-p+q(p/q)}{r+s(p/q)}=0
	\end{equation*}
	that is, $a$ is equivalent to zero.
	To sum up, $\modulitorusgerbe{1,2}{\emptyset}$ contains the quotient of $\R\setminus \mathbb Q$ by the action \eqref{actionmoduliQt2} of $\text{GL}_2(\Z)$ plus $0$ whose equivalence class contains all rational numbers.
\end{example}

We want to interpretate the moduli space $\nctorusgerbe{d,n}{\emptyset}$ as a moduli space of linear $\R^{n-d}$-actions on a real torus $\R^n/\Z^n$. To do that, we make use of Gale transforms as in Section \ref{QGITdec}. So let $(A_1,\hdots,A_n)$ satisfy \eqref{GTdec}. By Theorem \ref{TheoremGITdec}, the calibrated torus $\nctorusgerbe{h}{\emptyset}$ is isomorphic to the quotient stack $[\Torus^n/\mathcal A]$.

The next step is to tropicalize the situation. Consider the log map
\begin{equation*}
w\in\Torus ^n \longmapsto \log w:=(\log \vert w_1\vert,\hdots, \log \vert w_n\vert)\in\R^n
\end{equation*}
and observe that 
\begin{equation*}
\log (wE(\langle A,T\rangle))=\log w+\langle A,-2\pi\Im T\rangle
\end{equation*}
since all the $A_i$ belong to $\R^{n-d}$. Let $t\in\R^{n-d}$ act on $x\in \R^n$ by
\begin{equation*}
\mathcal B(t,x)=x+\langle A,t\rangle
\end{equation*}
and set
\begin{equation*}
T\in\C^{n-d}\longmapsto \mu (T)=-2\pi\Im T
\end{equation*}
Then the log map is $(\mathcal A,\mathcal B)$ equivariant. More precisely
\begin{equation*}
\log (\mathcal A(w,T))=\mathcal B(\mu (T),x)
\end{equation*}
where $x=\log  w $. So it associates to a calibrated torus a unique linear $\R^{n-d}$-action on $\R^n$. Such an action descends as an action on the real torus $\R^{n}/\Z^n$. We denote the real torus endowed with this action by $\log \nctorusgerbe{h}{\emptyset}$.

Moreover, given a torus morphism $(L,H)$ between two calibrated Quantum Tori $\nctorusgerbe{h}{\emptyset}$ and $\nctorusgerbe{h'}{\emptyset}$, we consider the $\R$-version of \eqref{completemorphismCD}
\begin{equation}
\label{CDRLHK}
\begin{tikzcd}
0 \arrow[r] &\R^{n-d}\arrow[d,"K"]\arrow[r,"i"]&\R^n\arrow[d,"H"]\arrow[r,"h"]&\R^d\arrow[d,"L"]\arrow[r] &0\\
0 \arrow[r] &\R^{n-d}\arrow[r,"i'"]&\R^n\arrow[r,"h'"]&\R^d\arrow[r] &0
\end{tikzcd}
\end{equation}
Then $(H,K)$ defines a linear equivariant map between $\log \nctorusgerbe{h}{\emptyset}$ and $\log \nctorusgerbe{h'}{\emptyset}$, to wit
\begin{equation*}
H(x+\langle A,t\rangle)=H(x+i(t))=Hx+i'(K(t))=Hx+\langle A',K(t)\rangle
\end{equation*}
Hence the log map associates to a calibrated torus morphism between $\nctorusgerbe{h}{\emptyset}$ and $\nctorusgerbe{h'}{\emptyset}$ a linear equivariant map between $\log \nctorusgerbe{h}{\emptyset}$ and $\log \nctorusgerbe{h'}{\emptyset}$. Composition is obviously preserved; and isomorphisms go to isomorphisms. Let $\mathscr{Q}\mathscr{T}^{cal}_{n,d,\emptyset}$ be  the category of calibrated Quantum Tori $\nctorusgerbe{h}{\emptyset}$ with fixed dimensions $n$ and $d$. And let $\mathscr{N}\mathscr{C}_{n,d}$ be the category whose objects are $\R^n/\Z^n$ endowed with a $\R^{n-d}$-linear action and whose morphisms are linear equivariant maps. We just proved
\begin{proposition}
	\label{propCElog}
	The log map induces an equivalence of category between $\mathscr{Q}\mathscr{T}^{cal}_{n,d,\emptyset}$ and $\mathscr{N}\mathscr{C}_{n,d}$.
\end{proposition}
And we have
\begin{theorem}
	\label{thmmoduliQt}
	The moduli space $\modulitorusgerbe{d,n}{\emptyset}$ is homeomorphic to
	\begin{enumerate}[\rm i)]
		\item the space of $\R^{n-d}$-linear actions on $\R^n/\Z^n$ up to linear equivariant isomorphisms.
		\item the space of real Quantum Tori $\TT_{\R^{n-d},\hbar}^n$ up to Morita equivalence.
	\end{enumerate}
\end{theorem}
	
\begin{proof}
	The first item is a direct consequence of Propositions \ref{propmoduliQt} and \ref{propCElog}. The second item reformulates the first one using results of \cite{Rieffel} which state that Morita equivalence of real Quantum Tori is given by equivalence \eqref{actionmoduliQt} in the two-dimensional case (in (ii) the $\hbar$ refers to the `slope-parameter' for the Kronecker foliations described in the introduction).
\end{proof}
	
As a variant, we may also compute the moduli space $\modulitorusgerbe{d,n}{J}$ of calibrated Quantum Tori with fixed dimensions $d$ and $n$ and fixed non-empty set $J$. We just treat the most interesting case, that is the maximal case. For the rest of the subsection, we are thus interested in calibrated Quantum Tori with $J_{\text{max}}=\{d+1,\hdots, n\}$. As before, it is determined by the value of $\hbar$, that is by a point in $(\R^d)^{n-d}$. But the morphisms are different.  The previous computations are still valid, but recall that a morphism $H$ must preserve the set of virtual generators, hence is associated to a permutation $s$ of the set $J_{\text{max}}$. Hence, in \eqref{Hdecompo}, we have $H_2=H_3=0$ and $H_4$ is the matrix of $s$, that is
\begin{equation}
\label{smatrix}
s=(s(e_{d+1},\hdots,e_n))=(e_{s(d+1)},\hdots,e_{s(n)})
\end{equation}
so that we have 
\begin{proposition}
	\label{propmoduliQtbis}
	The moduli space $\modulitorusgerbe{d,n}{J_{\text{max}}}$ of calibrated Quantum Tori $\nctorusgerbe{h}{J_{\text{max}}}$ with fixed $d$ and $n$ and set of virtual generators $J_{\text{max}}$ is homeomorphic to the quotient of $(\R^d)^{n-d}$ by the following action of $\text{GL}_d(\Z)\times\mathfrak{S}_{n-d}$: given $v$ and $(H_1,s)$ we set
	\begin{equation}
	\label{actionmoduliQtbis}
	v\cdot (H_1,s)=H_1^{-1}v s
	\end{equation}
\end{proposition}

Thinking of real Quantum Tori, then we see that $\modulitorusgerbe{d,n}{J_{\text{max}}}$ is homeomorphic to the space of real Quantum Tori $\TT_{\R^{n-d},\theta}^n$ up to linear equivalence of $\R^n$ preserving both the $\R^{n-d}$-action and a fixed basis of $\R^{n-d}$ up to permutation.

Finally,
\begin{proposition}
	\label{propmoduliQtter}
	The moduli space $\modulitorusgerbe{d,n}{J_{\text{max}},*}$ of calibrated Quantum Tori $\nctorusgerbe{h}{J_{\text{max}}}$ with fixed $d$ and $n$ and set of virtual generators $J_{\text{max}}$ up to marked isomorphisms is homeomorphic to the quotient of $(\R^d)^{n-d}$ by the following action of $\text{GL}_d(\Z)$: given $v$ and $H_1$ we set
	\begin{equation}
	\label{actionmoduliQtter}
	v\cdot H_1=H_1^{-1}v 
	\end{equation}
\end{proposition}

Thinking of real Quantum Tori, the difference with the previous interpretation is that $\modulitorusgerbe{d,n}{J_{\text{max}}}$ is homeomorphic to the space of real Quantum Tori $\TT_{\R^{n-d},\theta}^n$ up to linear equivalence of $\R^n$ preserving both the $\R^{n-d}$-action and a fixed basis of $\R^{n-d}$; no permutation of this fixed basis is allowed.
\begin{remark}
	\label{rkmodulitorus}
	To define moduli spaces of Quantum Tori, one is faced with the problem that the rank, that is the minimal number of generators of some $\Gamma$ in $\R^d$ is not stable when moving slightly $\Gamma$. For example $\Z+a\Z$ has rank $1$ if $a$ is rational and rank $2$ if $a$ is irrational. So if we let $\modulitorus{d}{n}$ be the moduli space of Quantum Tori of dimension $d$ and rank $n$, then the same type of computations show that $\modulitorus{d}{n}$ naturally injects in $\modulitorusgerbe{d,n}{\emptyset}$ but it misses the dense subset of $\Gamma$ with smaller rank. We will recover and explain this fact in Section \ref{Qmodulidec}.
	
	Of course, this problem disappears when calibrating the tori. Calibrations are useful in moduli problems.  
\end{remark}

\begin{remark}
	\label{rkmodulitorusvirtual}
	Allowing a non-empty set of virtual generators adds automorphisms acting as a permutation of this set. This explains the $\mathfrak S_{n-d}$ factor in the statement of Proposition \ref{propmoduliQtbis}.
	
	Of course, this problem disappears when using marked isomorphisms. Marking kill these extra automorphisms as shown by Proposition \ref{propmoduliQtter}. This will be crucial in Section \ref{QmoduliPn}.
\end{remark}

\subsection{Moduli Spaces of $\Gamma$-complete Quantum Torics}
\label{QmoduliGC}

We consider now moduli spaces of $\Gamma$-complete Quantum Torics. For that purpose, we need to define the combinatorial type of a Quantum Fan and combinatorial equivalence of Quantum Fans.

\begin{definition}
	\label{defcombtype}
	The {\it combinatorial type} of a Quantum Fan $(\Delta,v)$ is the poset $\text{comb}(\Delta)$ of subsets $I$ of $\{1,\hdots,p\}$ such that $\sigma_I$ is a cone of $\Delta$.
\end{definition} 
Moreover,
\begin{definition}
	\label{defcombequiv}
	Two Quantum Fans $(\Delta,v)$ and $(\Delta',v')$ are {\it combinatorially equivalent} if
	\begin{enumerate}[\rm i)]
		\item They have same number of $1$-cones, say $p$.
		\item There exists a bijection of $\{1,\hdots,p\}$ which sends bijectively $\text{comb}(\Delta)$ onto $\text{comb}(\Delta')$.
	\end{enumerate}
\end{definition}
Let $D$ be the combinatorial type of some $\Gamma$-complete simplicial Quantum Fan. We are interested in the moduli space $\modulitoric{D}{d}$ of $\Gamma$-complete simplicial Quantum Torics $\nctoric{\Delta}{\Gamma}{v}$ such that $\text{comb}(\Delta)$ is equivalent to $D$. Thanks to Theorem \ref{prisoQTV}, it can alternatively be defined as
\begin{equation}
\label{moduliGcompdef}
\modulitoric{D}{d}=\{(\Delta,v)\text{ Quantum Fan in }\R^d\mid \text{comb}(\Delta)\simeq D\}/\sim
\end{equation}
where $\simeq$ means combinatorially equivalent in the sense of Definition \ref{defcombequiv} and $\sim$ is the isomorphism of Quantum Fans.

The permutation group $\mathfrak{S}_{p}$ acts on $\{1,\hdots,p\}$ hence on $D$ and we have
\begin{equation}
\text{comb}(\Delta)\simeq D\iff \exists s\in\mathfrak{S}_{p}\quad\text{ s.t. }\text{comb}(\Delta)=s\cdot D
\end{equation} 
We may assume that all Quantum Fans are standard, hence $(\Delta,v)$ is entirely described by $v_{d+1},\hdots,v_p$ vectors of $\R^d$. 
So making the following definition
\begin{definition}
	\label{defrealiz}
	We say that the set of vectors $\underline{v}=(v_{d+1},\hdots,v_p)$ of $\R^d$ is {$D$-realizable} if there exists a standard fan $\Delta$ in $\Gamma=\Z^d+\Z v_{d+1}+\hdots +\Z v_p$ with $1$-cone generators $v_1=e_1,\hdots,v_d=e_d,v_{d+1},\hdots,v_p$ whose combinatorial type is $D$.
\end{definition}
\noindent we obtain that \eqref{moduliGcompdef} can be rewritten as
\begin{equation}
\label{moduliGcompdef2}
\modulitoric{D}{d}=
\left \{\underline{v}\in(\R^d)^{p}\mid \underline{v}\text{ is }  D\text{-realizable}\right \}
\big /\sim
\end{equation}
Since $D$ is simplicial, being $D$-realizable is an open condition. Thus the set between parenthesis in \eqref{moduliGcompdef2} is an open subset of $(\R^d)^{p-d}$.

As in Section \ref{QmoduliQT}, this description allows us to endow $\modulitoric{D}{d}$ with the quotient topology coming from the euclidean topology of $(\R^{d})^{p-d}$.

The computations are similar to those of Section \ref{QmoduliQT}. Indeed, $\underline{v}\sim\underline{v'}$, that is $(\Delta,v)$ and $(\Delta',v')$ are isomorphic Quantum Fans if and only if there exists a permutation $s\in\mathfrak{S}_p$ such that 
\begin{enumerate}[i)]
	\item The matrix $(v_{s(1)},\hdots,v_{s(d)})$ is invertible.
	\item The linear isomorphism
	\begin{equation}
	\label{L}
	L:=(v_{s(1)},\hdots,v_{s(d)})^{-1}
	\end{equation}
	sends $v_{s(i)}$ onto $v'_i$ for all $i$ between $1$ and $n$.
	\item $\text{comb}(\Delta)=\text{comb}(\Delta')$.
\end{enumerate}
Observe that, by definition, $L$ sends $(v_{s(1)},\hdots,v_{s(d)})$ onto $(e_1,\hdots,e_d)$ which is the same as $(v'_1,\hdots,v'_d)$. Hence point ii) is really meaningful for $i>d$. 

We conclude
\begin{theorem}
	\label{thmmoduliGC}
	The moduli space $\modulitoric{D}{d}$ is homeomorphic to the quotient {\rm \eqref{moduliGcompdef2}} of an open set of $(\R^d)^{n-d}$ by the equivalence relation defined in {\rm\eqref{L}}.
\end{theorem}
and
\begin{corollary}
	\label{cormoduliGC}
	The moduli space $\modulitoric{D}{d}$ can be endowed with a structure of a real orbifold of dimension $d(n-d)$.
\end{corollary}

\begin{remark}
	\label{rkorbifold}
	By a real, respectively complex, orbifold, we mean a stack over the category of $C^\infty$-manifolds, resp. complex manifolds (both with usual euclidean coverings), with only finite isotropy groups.
\end{remark}

\begin{proof}
	The isotropy group of any $\underline{v}$ is a subgroup of $\mathfrak{S}_p$.
\end{proof}

\begin{example}
	\label{P2def}
	We treat the case of the Quantum projective planes. Hence we take
	\begin{equation}
	\label{DP2}
	D=\{1,2,3,12,23,31\}
	\end{equation}
	and consider
	\begin{equation}
	\label{uvP2}
	\underline{v}=(a,b)\in\R^{<0}\times\R^{<0}
	\end{equation}
	and
	\begin{equation}
	\label{GammaP2}
	\Gamma=\Z^2\oplus \Z\underline{v}
	\end{equation}
	A standard Quantum Fan of projective plane has combinatorial type $D$ and $1$-cone generators $(e_1,e_2,\underline{v})$. This implies that \eqref{moduliGcompdef2} takes the form
	\begin{equation}
	\modulitoric{D}{2}=(\R^{<0}\times\R^{<0})/\mathfrak{S}_3
	\end{equation}
	Observe that $D$ is $\mathfrak{S}_3$-invariant, hence there is only one component. Also, condition i) in \eqref{L} is automatic, hence we really have a group action here. We are left with computing it. Consider the generators $(12)$ and $(123)$ of $\mathfrak{S}_3$. They correspond to
	\begin{equation}
	L_1=\begin{pmatrix}
	0 &1\\
	1 &0
	\end{pmatrix}
	\qquad\text{ and }\qquad
	L_2=\begin{pmatrix}
	0 &1/b\\
	1 &-a/b
	\end{pmatrix}
	\end{equation}
	The isomorphism $L_1$ exchanges $e_1$ and $e_2$ and sends $\underline{v}=(a,b)$ onto $(b,a)$. The isomorphism $L_2$ sends $e_1$ onto $e_2$, and $\underline{v}$ onto $e_1$ and finally $e_2$ onto $(1/b,-a/b)$. In other words, setting
	\begin{equation}
	\label{sigmatau}
	\sigma(a,b)=(b,a)\qquad\text{ and }\qquad \tau(a,b)=(1/b,-a/b)
	\end{equation}
	then we obtain
	\begin{theorem}
		\label{thmmoduliP2}
		The moduli stack of Quantum $\mathbb P^2$s is the real orbifold $[(\R^{<0}\times\R^{<0})/\mathfrak{S}_3]$ where the $\mathfrak{S}_3$-action is generated by the involution $\sigma$ and the trivolution $\tau$ defined in \eqref{sigmatau}.
	\end{theorem}
	Points with isotropy can easily be computed. Firstly, only $a=b=-1$ is $\tau$-invariant. Recall from Example \ref{exqdP2} that this is the only point which encodes the $\mathbb P^2$. Of course, it is also $\sigma$-invariant, so its isotropy group is the full $\mathfrak{S}_3$. Secondly, we obtain that
	\begin{enumerate}[a)]
		\item Points $(a,a)$ with $a\not =-1$ have isotropy group equal to $\Z_2$, generated by the involution $\sigma$.
		\item Points $(a,-1)$ with $a\not =-1$ have isotropy group equal to $\Z_2$, generated by the involution $\sigma\circ\tau$.
		\item Points $(-1,b)$ with $b\not =-1$ have isotropy group equal to $\Z_2$, generated by the involution $\tau\circ\sigma$.
	\end{enumerate}
	All other points have trivial isotropy group.
	
	The computations can easily be generalized to the case of Quantum $\mathbb P^n$s for any $n>0$. One finds that its moduli stack is the real orbifold $[(\R^{<0})^n/\mathfrak{S}_n]$ and the generators of the action can be computed. Here again, the only point with isotropy group equal to $\mathfrak{S}_n$ is $(-1,\hdots,-1)$ corresponding to the classical $\mathbb P^n$.
\end{example}

\subsection{Moduli Spaces of calibrated Quantum Torics of maximal length}
\label{Qmodulidec}
We finally consider moduli spaces $\modulitoricgerbe{D}{n,d}$ of calibrated Quantum Torics of maximal length and fixed combinatorial type $D$. We assume that $D$ contains (at least) a cone of dimension $d$, hence the set of $1$-cones of any Quantum Fan with combinatorial type equivalent to $D$ contains a basis of $\R^d$. The computations are similar to those of Section \ref{QmoduliQT}. Firstly, thanks to Theorem \ref{prisodecQTV}, we have
\begin{equation}
\label{modulidecQTmax}
\modulitoricgerbe{D}{n,d}=
\left\{
\begin{aligned}
&\hbar : \Z^{n-d}\to\R^d\mid \hbar(e_{1},\hdots,e_{p})\text{ is }\\
&D\text{-realizable}
\end{aligned}
\right\}
\bigg/\sim
\end{equation}
where $\sim$ means isomorphism of calibrated Quantum Fans of maximal length. We denote by $U$ the open set of $(\R^d)^{n-d}$ appearing in \eqref{modulidecQTmax} and endow $\modulitoricgerbe{D}{n,d}$ with the quotient topology.

 Then two calibrations $\hbar$ and $\hbar'$ define isomorphic calibrated Quantum Torics if and only if there exists $H$ in $\text{GL}_n(\Z)$ which acts as a permutation both on the set $\{v_1,\hdots,v_p\}$ of $1$-cones and on the set $\{v_{p+1},\hdots,v_n\}$ of virtual generators. Hence $H$ decomposes as a diagonal block matrix
\begin{equation}
\label{Hdecompo2}
H=\begin{pmatrix}
s &0\\
0 &t
\end{pmatrix}^{-1}
\end{equation}
with $s\in\mathfrak{S}_p$ and $t\in\mathfrak{S}_{n-p}$. Note that $s$ is not arbitrary but must satisfy that $(v_{s(1)},\hdots,v_{s(d)})$ is a basis of $\R^d$. As usual, we take as convention that $H$ sends $(v_{s(1)},\hdots,v_{s(d)})$ onto the canonical basis, hence the $(-1)$-exponent in \eqref{Hdecompo2}.

Starting with some $\hbar$ in $U$, then there exists a unique couple $(L,h')$ making \eqref{CDLHmod} commutative. To wit, the linear isomorphism $L$ is given by \eqref{L} and $h'$ is just $LhH^{-1}$.

We thus obtain the following complete description of $\modulitoricgerbe{D}{n,d}$.

\begin{theorem}
	\label{thmmodulidec}
	Fix $0<d<n$. Let $D$ be the combinatorial type of a Fan in $\R^d$.  Assume that $D$ contains (at least) a cone of dimension $d$.
	
	Then, the moduli space $\modulitoricgerbe{D}{n,d}$ of calibrated Quantum Torics of maximal length and fixed combinatorial type $D$ is homeomorphic to the quotient of the open set
	\begin{equation}
	\label{U}
	\{\hbar : \Z^{n-d}\to\R^d\mid \hbar(e_{d+1},\hdots,e_{n})\text{ is }
	D\text{-realizable}\}
	\end{equation}
	by the equivalence relation
	\begin{equation}
	\label{simmodulidec}
	\hbar\sim\hbar'\iff \exists (s,t)\in\mathfrak{S}_p\times\mathfrak{S}_{n-p}\text{ s.t. }L\text{ exists and }h'=LhH^{-1}
	\end{equation}
	for $L$, respectively $H$, defined in {\rm \eqref{L}}, resp. {\rm \eqref{Hdecompo2}}.
\end{theorem}
and 
\begin{corollary}
	\label{cormodulidec}
	The moduli space $\modulitoricgerbe{D}{n,d}$ can be endowed with a structure of a real orbifold of dimension $d(n-d)$.
\end{corollary}

\begin{proof}
	The isotropy group of any $\hbar$ is a subgroup of $\mathfrak{S}_p\times\mathfrak{S}_{n-p}$.
\end{proof}

Consider the special case where $p=n$, hence $J=\emptyset$. Forgetting the calibration, it corresponds to the $\Gamma$-complete case treated in Section \ref{QmoduliGC}. Indeed, we have
\begin{corollary}
	\label{cormodulidecetGCweak}
	The orbifold $\modulitoricgerbe{D}{n,d}$ and the orbifold $\modulitoric{D}{d}$ are isomorphic.
\end{corollary}

\begin{proof}
	Since $p=n$, the group occurring in \eqref{simmodulidec} is just $\mathfrak{S}_p$. Then, comparing the statements of Theorems \ref{thmmoduliGC} and \ref{thmmodulidec} yields the result.
\end{proof}
As an application of Corollary \ref{cormodulidecetGCweak}, we obtain directly the moduli stack of trivially calibrated $\Gamma$-complete quantum deformations of $\mathbb P^2$ described in Example \ref{exqdP2dec}.
\begin{corollary}
	\label{cormoduliP2dec}
	 The moduli stack of trivially calibrated Quantum $\mathbb P^2$s is the real orbifold $[(\R^{<0}\times\R^{<0})/\mathfrak{S}_3]$ where the $\mathfrak{S}_3$-action is generated by the involution $\sigma$ and the trivolution $\tau$ defined in \eqref{sigmatau}.
\end{corollary}

\begin{proof}
	Combine Theorem \ref{thmmoduliP2} and Corollary \ref{cormodulidecetGCweak}.
\end{proof}

It is important to notice from Example \ref{exqdP2dec} that the moduli stacks of  Corollary \ref{cormoduliP2dec} (and more generally those of Theorem \ref{thmmoduliP2}) consist of {\it different} Quantum Torics. Irrational points (that is points in $\R^{<0}\times\R^{<0}$ with at least one irrationnal coordinate) are isomorphic but rational points are not. In particular, the first one contains $\mathbb P^2$ whereas the second one contains only a $\Z$-gerbe over $\mathbb P^2$.

Hence, this isomorphism of moduli stacks needs an explanation. It comes indeed from the following equivalence of category, which has its own interest. We first need a definition.
\begin{definition}
	\label{definjectivecat}
	A full subcategory $\mathscr{C}$ of $\mathscr{Q}^{cal}$ is said to be {\it $\hbar$-injective} if 
	\begin{enumerate}[\rm i)]
		\item Every object of $\mathscr{C}$ has empty set of virtual generators.
		\item Every object of $\mathscr{C}$ has injective $\hbar$ function from $\Z^{n-d}$ to $\Gamma$.
	\end{enumerate}
\end{definition}
Then we state
\begin{theorem}
	\label{thmeqcatmoduli}
	Let $\mathscr{C}$ be a $\hbar$-injective full subcategory of $\mathscr{C}$. Then, the forgetful  functor $\text{\slshape f}$ induces an equivalence of categories between $\mathscr{C}$ and the subcategory $\text{\slshape f}(\mathscr C)$ of $\mathscr{Q}$.
\end{theorem}

\begin{proof}
	We just restrict $\text{\slshape f}$ to $\mathscr{C}$. It is (essentially) surjective by definition so we just need to check it is fully faithful. Now, given $\mathscr L$ between $\text{\slshape f}(\nctoricgerbe{\Delta}{h}{J})$ and $\text{\slshape f}(\nctoricgerbe{\Delta'}{h'}{J'})$, take the corresponding $L$ and observe that $h$, respectively $h'$ is a bijection from $\Z^n$ to $\Gamma$, resp. from $\Z^{n'}$ to $\Gamma'$, since $\mathscr{C}$ is assumed to be $\hbar$-injective. This implies that $H:=h'Lh^{-1}$ is well defined as linear map from $\Z^n$ to $\Z^{n'}$ making commutative the diagram \eqref{CDLHmod}. Since there is no virtual generator, hence no further constraint to check on $H$, then $(L,H)$ satisfies the conditions of Theorem \ref{thmdecmorphisms2}. This gives a well-defined morphism $\mathscr{L}^{cal}$ such that $\text{\slshape f}(\mathscr{L}^{cal})$ is equal to $\mathscr{L}$. This finishes the proof.
\end{proof}
\noindent Examples \ref{P1QFans} and \ref{P1P2morphisms} show that the condition of having empty set of virtual generators is necessary.
\vspace{5pt}\\
From Theorem \ref{thmeqcatmoduli}, we obtain a more precise version of Corollary \ref{cormodulidecetGCweak}.
\begin{corollary}
	\label{cormodulidecetGC}
	The functor $\text{\slshape f}$ induces a stack isomorphism between the orbifold $\modulitoricgerbe{D}{n,d}$ and the orbifold $\modulitoric{D}{d}$.
\end{corollary}

\begin{proof}
	Apply Theorem \ref{thmeqcatmoduli} to the full subcategory $\mathscr{C}$ of $\mathscr{Q}^{cal}$ formed by the trivially calibrated $\Gamma$-complete Quantum Torics.
\end{proof}
We also obtain a direct proof of the injection of $\modulitorus{d}{n}$ in $\modulitorusgerbe{d,n}{\emptyset}$, which is asserted in Remark \ref{rkmodulitorus}.
\begin{corollary}
	\label{cormodulitori}
	The functor $\text{\slshape f}$ induces a homeomorphism between the subset of $\modulitorusgerbe{d,n}{\emptyset}$ consisting of calibrated Quantum Tori with $\text{rk}(\Gamma)$ equal to $n$ and $\modulitorus{d}{n}$.
\end{corollary}

\begin{proof}
	Apply Theorem \ref{thmeqcatmoduli} to the full subcategory $\mathscr{C}$ of $\mathscr{Q}^{cal}$ formed by calibrated Quantum Tori $\nctorusgerbe{h}{\emptyset}$ with $\text{rk}(\Gamma)$ equal to $n$.
\end{proof}

\subsection{The case of the projective space}
\label{QmoduliPn}
In Corollary \ref{cormoduliP2dec}, we describe the moduli space of trivially calibrated Quantum $\mathbb P^2$s as a real orbifold and the case of Quantum $\mathbb P^d$s follow easily, using also Example \ref{P2def}. The goal of this section is to use LVMB manifolds and configurations to show that, adding an even number of virtual generators, and using marked isomorphisms, the corresponding moduli space of Quantum $\mathbb P^d$s becomes a {\sl complex} orbifold. 

So we start with an even calibrated Quantum Fan $(\Delta,h)$ such that the combinatorial type $\Delta_0$ of $\Delta$ is that of a $d$-simplex and $h$ sends the first $d+1$ vectors canonical basis of $\R^n$ onto
\begin{equation*}
v_1=e_1\qquad \hdots \qquad v_d=e_d\qquad v_{d+1}=-a_1e_1-\hdots-a_de_d
\end{equation*}
with $a_i>0$ and the other ones onto some virtual generators $v_{d+2}$, ..., $v_n$.

We add $v_{n+1}$ such that the sum of the $v_i$'s is zero. We associate to it the unique (affine) Gale transform $(A_1,\hdots, A_{n+1})$ satisfying
\begin{equation*}
A_{d+2}=e_1\qquad \hdots \qquad A_{n+1}=e_{n-d}
\end{equation*}
Since $n-d$ is even, we transform it through \eqref{isoRC} into a configuration $(\Lambda_1,\hdots,\Lambda_{n+1})$ with
\begin{equation}
\label{Lambdanorm}
\Lambda_{d+2}=e_1\qquad \hdots \qquad \Lambda_{n+1}=ie_{m}
\end{equation}
with $m$ equal to $(n-d)/2$. The $(d+1)$-uple $(\Lambda_1,\hdots,\Lambda_{d+1})$ lives in the open set
\begin{equation}
\label{Omega}
\Omega\subset\{Z\in (\C^m)^{d+1}\mid Z_1+\hdots+Z_{d+1}=(-1-i,\hdots,-1-i)\}
\end{equation}
The couple $(\mathbb C^{d+1}\setminus\{0\}\times \Torus^{n-d},\Lambda)$ is a LVMB datum by Lemma \ref{LVM+B} with indispensable points $\Lambda_{d+2},\hdots,\Lambda_{n+1}$ by Lemma \ref{lemmaindetvir}.
We prove
\begin{theorem}
	\label{thmoduliPn}
	The moduli stack $\modulitoricgerbe{\Delta_0}{d+1,d,*}$ of Quantum $\mathbb P^d$s up to mark\-ed isomorphisms is the complex orbifold $[\Omega/\mathfrak{S}_{d+1}]$ where $\mathfrak{S}_{d+1}$ acts on a configuration $(\Lambda_1,\hdots, \Lambda_{d+1})\in \Omega$ by permutation.
\end{theorem}
	
\begin{proof}
	Let $(L,H)$ be an isomorphism between $(\Delta,h)$ and $(\Delta,h')$. We assume that both are standard. Then $H$ consists of a permutation $s$ of $\{1,\hdots,d+1\}$ and a permutation $t$ of $\{d+2,\hdots,n+1\}$. And $L$ sends the $1$-cone generators $v_1,\hdots,v_{d+1}$ onto $v'_{s(1)},\hdots, v'_{s(d+1)}$ and the virtual generators $v_{d+2},\hdots , v_{n+1}$ onto $v'_{t(d+2)},\hdots , v'_{t(n+1)}$. Since we only use marked isomorphisms, this forces $t$ to be the identity.
	
	Through the Gale transform, $H$ induces the permutation $s$ on $A_1,\hdots, A_{d+1}$ and thus on $\Lambda_1,\hdots,\Lambda_{d+1}$. Note that \eqref{Lambdanorm} is still verified, hence there is no renormalization to do and we are done.
\end{proof}	
	Here the use of marked isomorphisms is crucial. Let $(L,H)$ be an isomorphism between $(\Delta,h)$ and $(\Delta,h')$, not necessarily a marked one. Recall the isomorphism $K$ of \eqref{CDRLHK}. We have
	\begin{equation}
	\label{calculmodPn}
	\begin{aligned}
	H(\langle A_1,T\rangle,\hdots,\langle A_{n+1},T\rangle)&=(A_{s(1)},\hdots,A_{s(d+1)},\\
	&\qquad\qquad A_{t(d+2)},\hdots,A_{t(n+1)})\\
	&=(\langle A'_1,K(T)\rangle,\hdots, \langle A'_{n+1},K(T)\rangle) \\
	&=(\langle\kern 2pt^t\kern -2ptK A'_1,T\rangle,\hdots, \langle \kern 2pt^t\kern -2ptKA'_{n+1},T\rangle)
	\end{aligned}
	\end{equation}
	If $H$ is the permutation matrix of $(s,t)$ as assumed in \eqref{calculmodPn}, then $K$ is nothing else than the permutation matrix of $t$. So, using marked isomorphisms means keeping $K$ equal to the identity and we obtain that $H$ induces the permutation $s$ on $A_1,\hdots, A_{d+1}$ and thus on $\Lambda_1,\hdots,\Lambda_{d+1}$ as wanted.
	
	If we now use arbitrary isomorphism, then \eqref{calculmodPn} implies that $H$ induces the permutation $s$ on (the columns of) $A_1,\hdots, A_{d+1}$ and the permutation $t$ on the rows of $A_1,\hdots, A_{d+1}$. Hence it induces the permutation $s$ on (the columns of) $\Lambda_1,\hdots,\Lambda_{d+1}$ and the permutation $t$ on the {\sl real} rows of $\Lambda_1,\hdots,\Lambda_{d+1}$. But this does not produce always holomorphic transformations of $\Lambda_1,\hdots,\Lambda_{d+1}$.
	
	In other words, denoting by $\Omega^{\R}$ the set in $(\R^{n-d})^{d+1}$ obtained by rewriting $\Omega$ through \eqref{isoRC}, we have
	\begin{corollary}
		\label{thmoduliPnnotmark}
		The moduli stack $\modulitoricgerbe{\Delta_0}{d+1,d}$ of Quantum $\mathbb P^d$s is the real orbifold $[\Omega^{\R}/\mathfrak{S}_{d+1}\times\mathfrak S_{n-d}]$ where $\mathfrak{S}_{d+1}$, respectively $\mathfrak S_{n-d}$ acts on a configuration $(A_1,\hdots, A_{d+1})\in \Omega^{\R}$ by permuting its columns, resp. its rows.
	\end{corollary}
Unfortunately, the case of $\mathbb P^n$ is very special. There seems to be no Theorem analogous to Theorem \ref{thmoduliPn} when $\Delta$ is a more complicated complete and $\Gamma$-complete fan. The technical reason for that is that such a fan would have at least $d+2$ cones of dimension $1$. Performing the same transformations as above, we encode every associated Quantum Torics as a $(d+1)$-uple $(\Lambda_1,\hdots,\Lambda_{d+1})$ living in the set $\Omega$ defined in \eqref{Omega}. Now, if $(L,H)$ is an isomorphism between $(\Delta,h)$ and $(\Delta,h')$, it acts as a couple of permutations $(s,t)\in\mathfrak{S}_p\times \mathfrak{S}_{n+1-p}$. Marking forces $t$ to be the identity and we are left with $(s,Id)$. However, since $p$ is at least $d+2$, when applying $(s,t)$ to a configuration, we do not always obtain a configuration satisfying \eqref{Lambdanorm}. Of course, this just means that we have to renormalize the image configuration. But renormalizing means applying a {\sl real} linear transformation to $\Lambda$ since we want to fix the values of $2m$ of the $\Lambda_i$s. Hence such a $H$ defines a {\sl real, a priori non holomorphic} transformation of $\Omega$, hence the associated moduli stack is only a real orbifold.

\subsection{Moduli of LVMB manifolds}
\label{Gmoduli}
To obtain complex orbifolds as moduli stacks, we must consider moduli stacks of the categories $\mathscr{N}_{LVMB}$ and $\mathscr{N^*}_{LVMB}$, that is of LVMB manifolds up to $G$-biholomorphisms.

Fix $\mathcal S$ and let $(\mathcal S,\Lambda)$ be a LVMB datum. We want to compute the moduli space $\mathcal M^{\mathcal S}_{m,n}$ of $G$-biholomorphism classes of LVMB manifolds obtained from a datum $(\mathcal S,\Lambda')$ with $m'=m$ and $n'=n$. We assume that 
\begin{equation}
\mathcal S=\Torus^k\times\mathcal S_0\qquad\text{ and }\qquad \mathcal S_0\text{ is simply-connected}
\end{equation} which implies that
\begin{equation}
\label{km}
\Lambda_i\text{ is indispensable}\iff i\leq k
\end{equation}
We also assume that
\begin{equation}
\Torus^{m+1}\times\{0\}\subset \mathcal S
\end{equation}
hence $(\Lambda_{1},\hdots, \Lambda_{m+1})$ are affinely independent, i.e. \eqref{firstmcondition} is satisfied.

It follows that there exists an affine transformation of $\mathbb C^m$ sending $\Lambda$ onto a configuration 
(which we still denote by $\Lambda$) satisfying 
\begin{equation}
\label{LambdaAffineNorm}
\Lambda_1=ie_1,\qquad \Lambda_2-\Lambda_1=e_1,\qquad\hdots,\qquad\Lambda_{m+1}-\Lambda_1=e_m
\end{equation}
for $(e_1,\hdots,e_m)$ the canonical basis of $\mathbb C^m$.
 By Lemma \ref{lemmaaffindLVMB},
this does not change $N_{\Lambda}$ up to $G$-biholomorphism. 

In the same way, every $G$-biholomorphism class $[N_{\Lambda'}]$ of $\mathcal M^{\mathcal S}_{m,n}$ can be represented by a 
configuration $\Lambda'$ satisfying \eqref{firstmcondition}, \eqref{km} and \eqref{LambdaAffineNorm}. Observe that since we fix $\mathcal S$, the number of indispensable points $k$ is fixed and equal to the first Betti number of $\mathcal S$.

So we may encode the LVMB data we need in an open set $\mathcal T_{\mathcal S}$ in $(\mathbb C^m)^{n-m-1}$.
Assume now that $N_{\Lambda'}$ is $G$-biholomorphic to $N_\Lambda$ with $\Lambda$ and $\Lambda'$ belonging to $\mathcal T_{\mathcal S}$. Then, $G_\Lambda$ and $G_{\Lambda'}$ are isomorphic Lie groups. Hence, their universal cover are isomorphic as Lie groups. Using the presentation given in 
Proposition \ref{Glattice}, this shows that there exists a matrix $M$ in $\text{GL}_{n-m-1}(\mathbb C)$ which sends the lattice of $G_\Lambda$ bijectively onto that of 
$G_{\Lambda'}$. 
Using notations \eqref{Alattice} and \eqref{Blattice}, this means that there exists a matrix $P$ in $\text{GL}_{n-1}(\mathbb Z)$ such that 
\begin{equation}
\label{isoAB}
M(Id, B_\Lambda A_\Lambda^{-1})=(Id,B_{\Lambda'}A_{\Lambda'}^{-1})P.
\end{equation}
Decomposing $P$ as
\begin{equation}
\label{Pdec}
P=
\begin{pmatrix}
P_1 &P_2\cr
Q_1 &Q_2
\end{pmatrix}
\end{equation}
with $P_1$ a square matrix of size $n-m-1$ and $Q_2$ a square matrix of size $m$, we obtain
\begin{equation}
\label{equa1}
MB_\Lambda A_\Lambda^{-1}=(P_1+B_{\Lambda'}A_{\Lambda'}^{-1}Q_1)B_\Lambda A_\Lambda^{-1}=P_2+B_{\Lambda'}A_{\Lambda'}^{-1}Q_2.
\end{equation}
Because of \eqref{LambdaAffineNorm}, this means that 
\begin{equation}
\label{isoLambda}
^t \kern-2ptB_\Lambda=(\,^t\kern-2pt P_2+\,^t\kern-2pt Q_2\,^t\kern-2pt B_{\Lambda'})(\,^t\kern-2pt P_1+\,^t\kern-2pt Q_1\,^t\kern-2pt B_{\Lambda'})^{-1}
\end{equation}
that is
\begin{proposition}
	\label{propiso}
	The moduli space $\mathcal M^{\mathcal S}_{m,n}$ is the quotient of $\mathcal T_{\mathcal S}$ by the action of $\text{\rm GL}_{n-1}(\mathbb Z)$ described in {\rm\eqref{isoLambda}}.
\end{proposition}

We claim

\begin{theorem}
	\label{thmorbifold}
	If the number $k$ of indispensable points is less than $m+1$, then the moduli space $\mathcal M^{\mathcal S}_{m,n}$ can be endowed with a structure of a complex orbifold.
\end{theorem}

\begin{proof}
	From the previous description and Remark \ref{rkorbifold}, it is enough to prove that the stabilizers of action \eqref{isoLambda} are finite. Let $f$ be a $G$-biholo\-morphism of $N_\Lambda$. Set
	\begin{equation}
	\label{S1}
	\mathcal S_1=\{w\in\mathbb C^{n-m-1}\quad\vert\quad (1,\hdots,1,w)\in\mathcal S\}.
	\end{equation}
	Observe that \eqref{S1} is a cover of the quotient $N_1$ of $\mathcal S\cap \{z_1\cdots z_{m+1}\not= 0\}$ by the action \eqref{actionLVM}. Indeed, we have a commutative diagram
	\begin{equation}
	\label{incl}
	\begin{tikzcd}
	(\mathbb C^*)^{n-m-1}\arrow[r]\arrow[d] &\mathcal S_1 \arrow[r]\arrow[d] &\mathcal S\arrow[d]\\
	G_\Lambda \arrow[r] &N_1 \arrow[r] &N_\Lambda
	\end{tikzcd}
	\end{equation}
	where the horizontal maps are inclusions and the first two vertical ones are covers. 
	Then, up to composing with a permutation of $\mathbb C^n$, we may assume that $f$ sends 
	$N_1$
	onto itself. 
	Because of assumption \eqref{km}, the set \eqref{S1} is simply-connected, since it is in fact equal to $\mathcal S_0$. Hence it is the universal cover of $N_1$ and the restriction of $f$ to $N_1$, say $f_1$, lifts to a biholomorphic map $F_1$ of $\mathcal S_1$. More precisely, $\mathcal S_1$ is equal to $\mathbb C^{n-m-1}$ minus a finite union of codimension $2$ vector subspaces, hence by Hartogs, $F_1$ extends as a biholomorphism of $\mathbb C^{n-m-1}$.
	\vspace{5pt}\\
	On the other hand, the restriction of $f$ to $G_\Lambda$ is a Lie isomorphism of $G_\Lambda$ and lifts as a Lie isomorphism $\tilde F$ of its universal cover $\mathbb C^{n-m-1}$. And we have a commutative diagram
	\begin{equation}
	\label{lifts}
	\begin{tikzcd}[column sep=huge]
	\mathbb C^{n-m-1}\arrow[r,"E"]\arrow[d,"\tilde F"']&(\mathbb C^*)^{n-m-1}\arrow[d,"F_1"]\\
	\mathbb C^{n-m-1}\arrow[r,"E"'] &(\mathbb C^*)^{n-m-1}
	\end{tikzcd}
	\end{equation}
	But, diagram \eqref{lifts} shows that the linear map $\tilde F=M$ must preserve the standard lattice of $\C^{n-m-1}$, so we have
	\begin{equation}
	\label{equ3}
	\tilde F(z+e_i)=\tilde F(z)+P_1e_i:=\tilde F(z)+\sum_{j=1}^{n-m-1}a_{ij}e_j
	\end{equation}
	that is, comparing with \eqref{isoAB}, then $Q_1$ is equal to $0$ in the decomposition \eqref{Pdec}. But through \eqref{lifts}, this implies that 
	\begin{equation}
	\label{equ4}
	F_1(w)=\Big (w_1^{a_{1j}}\cdots w_{n-m-1}^{a_{n-m-1j}}\Big )_{j=1}^{n-m-1}
	\end{equation}
	Now, recall that $F_1$ is a biholomorphism of the whole $\mathbb C^{n-m-1}$, so must send a coordinate hyperplane onto another one without ramifying. This shows that $P_1=(a_{ij})$ is a matrix of permutation. Hence every stabilizer is a subgroup of $\mathfrak{S}_{n-m-1}$, so is finite.
\end{proof}
\begin{example}
	\label{tori}
	{\bf Tori.} Let $n=2m+1$, then there are $2m+1$ indispensable points, $\mathcal S$ is $(\mathbb C^*)^n$ and $N$ is a compact complex torus of dimension $m$ \cite[Theorem 1]{meersseman2000}. The associate polytope $K$ is reduced to a 
	point and $N=G$. The moduli space $\mathcal M$ is equal to the moduli space of compact complex tori of dimension $m$, which is not an orbifold for $m>1$.
\end{example}

\begin{example}
	\label{Hopf}
	{\bf Hopf surfaces.} Let $n=4$ and $m=1$, then there are two indispensable points and $\mathcal S$ is $(\mathbb C^*)^2\times\mathbb C^2\setminus\{(0,0)\}$.
	We fix $\lambda_1=i$ and $\lambda_2=1+i$. They generate the real affine line $\Im z=1$. The complement of this line in $\C$ has two connected components, say $\mathscr{H}^+$ and $\mathscr{H}^-$. A LVMB datum $(\mathcal S,\lambda)$ is given by a couple of complex numbers $(\lambda_3,\lambda_4)$ belonging to
	\begin{equation}
	\label{zoneHopf}
	\mathscr{H}^+\times \mathscr{H}^+\cup \mathscr{H}^-\times \mathscr{H}^-.
	\end{equation}
	The manifold $N_\Lambda$ is equal to the diagonal Hopf surface obtained by taking the quotient of $\mathbb C^2\setminus\{(0,0)\}$ by the group generated by
	\begin{equation}
	\label{eqHopf}
	(z,w)\longmapsto (E(\lambda_3-\lambda_1)\cdot z, (E(\lambda_4-\lambda_1)\cdot w)
	\end{equation}
	Two points $(\lambda_3,\lambda_4)$ and $(\lambda'_3,\lambda'_4)$ with coordinates in \eqref{zoneHopf} are equivalent if and only if their difference $(\lambda'_3,\lambda'_4)-
	(\lambda_3,\lambda_4)$ or their switched difference  $(\lambda'_3,\lambda'_4)-
	(\lambda_4,\lambda_3)$ belongs to the lattice $\mathbb Z\oplus \mathbb Z$. The isotropy group of 
	a point is $\mathbb Z_2$ for the diagonal $\lambda_3=\lambda_4$ and more generally for all couples with $\lambda_4-\lambda_3\in\mathbb Z$; and is zero elsewhere. The moduli space is a complex orbifold.
	\\
	Observe that not all Hopf surfaces are obtained as LVMB-manifolds, but only the linear diagonal ones. Now, they coincide with the set of Hopf surfaces that are 
	equivariant compactifications of $(\mathbb C^*)^2$.
\end{example}

\subsection{Morphisms between moduli stacks}
\label{GQmoduyli}
We want to compare the complex moduli stack $\mathcal M^{\mathcal S}_{m,n}$ of Section \ref{Gmoduli} (assuming Theorem \ref{thmorbifold} holds) with the real moduli stack $\modulitoricgerbe{D}{n,d,*}$ of calibrated Quantum Torics of maximal length up to marked isomorphisms.

We first review and complete the construction of $\modulitoricgerbe{D}{n,d, *}$ following the case of Quantum $\mathbb P^n$ treated in Section \ref{QmoduliPn}. Here $D$ is the combinatorial type of a complete Quantum Fan and the calibration is assumed to be maximal. Theorem \ref{thmmodulidec} describes $\modulitoricgerbe{D}{n,d}$ (i.e. without using marked isomorphisms). The case of $\modulitoricgerbe{D}{n,d,*}$ follows easily.
\begin{theorem}
	\label{thmmodulidecmarked}
	The moduli stack $\modulitoricgerbe{D}{n,d,*}$ is the real orbifold obtained as quotient stack of the set {\rm \eqref{U}} by the equivalence relation
	\begin{equation}
	\label{simmodulidecmkd} 
		\hbar\sim\hbar'\iff \exists s\in\mathfrak{S}_p\text{ s.t. }L\text{ exists and }h'=LhH^{-1}
	\end{equation}
	for $L$, respectively $H$, defined in {\rm \eqref{L}}, resp. {\rm \eqref{Hdecompo2}} with $t$ representing the identity.
\end{theorem}

\begin{proof}
	This is a direct rephrasing of Theorem \ref{thmmodulidec} and Corollary \ref{cormodulidec}, noting that $t$ must be the identity since we use marked isomorphisms.
\end{proof}
Now, let $(\mathcal S,\Lambda)$ be a balanced LVMB configuration. As in Section \ref{LVMBdonnedec}, we assume that it is a configuration of $n+1$ points, cf. Warning \ref{warLVMB}. Fix $\mathcal S$, $m$ and $n+1$. We may associate to it the quotient stack $[N_\Lambda/\C^m]$, where $\C^m$ is the composition of the flows of the vector fields \eqref{etavf}. By Theorem \ref{mainGIT}, this is a calibrated Quantum Toric. In this way, we construct a map from the set of objects of $\mathscr{N}^*_{LVMB}$ with same $\mathcal S$, $m$ and $n+1$ to the corresponding moduli stack $\modulitoricgerbe{D}{n,d,*}$. Two $G$-isomorphic objects are sent to the same point $\modulitoricgerbe{D}{n,d,*}$. So we just defined a map
\begin{equation}
\label{mapmodst}
[N_\Lambda]\in \mathscr{M}^{\mathcal S}_{m,n+1}\longmapsto \mathscr{D}[N_\Lambda]:=[N_\Lambda/\C^m]\in\modulitoricgerbe{D}{n,d,*}
\end{equation}
We want to investigate the local structure of this map. To do that, we first rewrite $\mathscr{M}^{\mathcal S}_{m,n+1}$ locally in a slightly different way as that given in Proposition \ref{propiso}. Especially, we normalize the involved LVMB data in a slighty different way.

So, consider  a fixed balanced LVMB datum $(\mathcal S,\Lambda)$ with dimensions $m$ and $n+1$. Assume that $k$ is less than $m+1$. We assume that $\mathcal S$ is equal to $\mathcal S_0\times\Torus^k$, hence the $k$ indispensable points are $n-k+1,\hdots, n+1$. We also assume that $\{d+1,\hdots,n+1\}$ belongs to $\mathscr{E}$, hence for the real affine span of $(\Lambda_{d+1},\hdots ,\Lambda_{n+1})$ is the whole $\C^m \simeq\R^{n-d}$.

Up to $G$-isomorphisms, every LVMB datum $(\mathcal S,\Lambda')$ close to $(\mathcal S,\Lambda)$ can be normalized as follows 
\begin{enumerate}[(i)]
	\item $\Lambda'$ is balanced.
	\item $\Lambda'_i=\Lambda_i$ for all $i$ between $n-m+2$ and $n+1$.
\end{enumerate}
We perform an affine Gale transform \eqref{GT2} on the corresponding $A$, respectively $A'$, and define uniquely $\underline{v}$, respectively $\underline{v'}$ by the additional constraints
\begin{equation}
\label{constraints}
v_1=v'_1=e_1,\ \hdots,\ v_d=v'_d=e_d
\end{equation}
This is possible thanks to the fact that $\{d+1,\hdots,n+1\}$ belongs to $\mathscr{E}$.

Theorem \ref{mainGIT} shows that $v=(v_1,\hdots,v_n)$, resp. $v'=(v'_1,\hdots, v'_n)$ is a representant of $\mathscr{D}[N_\Lambda]$, resp. $\mathscr{D}[N_{\Lambda'}]$. 

Set 
\begin{equation}
\label{OmegaL}
	\left\{Z\in (\C^m)^{n-m+1}\mid \sum_{i=1}^{n-m+1}Z_i=-\sum_{i=n-m+2}^{n+1}\Lambda_{i}\right \}
\end{equation}
Then a neighborhood of $(\Lambda_1,\hdots,\Lambda_{n-m+1})$ in \eqref{OmegaL} encodes all $[N_{\Lambda'}]$ close to $[N_{\Lambda}]$ through the map 
\begin{equation}
Z\in\Omega_\Lambda\longmapsto \Lambda':=(Z_1,\hdots,Z_{n-m+1}, \Lambda_{n-m+2},\hdots,\Lambda_{n+1})
\end{equation}
More precisely, we may take the open subset $\Omega_\Lambda$ of \eqref{OmegaL} which consists of points $Z$ such that 
\begin{equation}
\text{rank }(Z_{d+1},\hdots,Z_{n-m-1},\Lambda_{n-m+2},\hdots,\Lambda_{n+1})=n-d+1
\end{equation}
that is maximal.

The previous Gale Transform with constraints \eqref{constraints} gives a well defined mapping, say $GT$, from $\Omega_\Lambda$ to 
\begin{equation}
\label{arrivaldomain}
\left\{X\in (\R^d)^{n-d+1}\mid\sum_{i=1}^{n-d+1}X_i=-e_1-\hdots-e_d\right\}
\end{equation}
sending $\Lambda$ (that is $(\Lambda_1,\hdots,\Lambda_{n-m+1})$) on $v$ (that is on $v_{d+1},\hdots,v_n$).	Denote by $\Pi_v$ the image $GT(\Omega_\Lambda)$.

A local model for the complex orbifold $\mathscr{M}^{\mathcal S}_{m,n+1}$ at $[N_\Lambda]$ is given by the complex orbifold chart $[\Omega_\Lambda/H_\Lambda]$ for $H_{\Lambda}$ a finite subgroup of $\text{GL}_n(\Z)$ fixing $\Lambda$. It follows from the proof of Theorem \ref{thmorbifold} that $H_\Lambda$ is indeed a subgroup of $\mathfrak S_{n-m}$ and acts by permutation on the first $n-m$ coordinates of $\Omega_\Lambda$ fixing $\Lambda$.

A local model for the real orbifold $\modulitoricgerbe{D}{n,d,*}$ at $v$ is given by the real orbifold chart $[\Pi_v/J_v]$ where $J_v$ is a subgroup of $\mathfrak S_p$ acting on $\Pi_v$ and fixing $v$. The precise action can be computed from the presentation given in Theorem \ref{thmmodulidecmarked}.

The mapping $GT$ is equivariant with respect to these actions and descends as the morphism $\mathscr{D}$ making the following diagram commutative
\begin{equation}
\begin{tikzcd}
\Omega_\Lambda \arrow[d]\arrow[r,"GT"] &\Pi_v\arrow[d]\\
\text{[}\Omega_\Lambda/H_\Lambda\text{]}\arrow[r,"\mathscr{D}"']&\text{[}\Pi_v/J_v\text{]}
\end{tikzcd}
\end{equation} 

Given $v'$ in $\Pi_\Lambda$, recall from Lemma \ref{lemmaGalereverse} that every $\Lambda'$ in the fiber $GT^{-1}(v')$ is a Gale transform of $\underline v'$. Recall also from Lemma \ref{lemmaGaleetGaleaff} and Remark \ref{Galemodulo} that such a Gale transform is unique up to action of the affine group $\text{Aff }(\R^{n-d})$. 

Now, every $\Lambda'$ in the fiber $GT^{-1}(v')$ is balanced with fixed last $m$ coordinates. Hence the fiber $GT^{-1}(v')$ naturally identifies with the group
\begin{equation}
\label{Gfiber}
\begin{aligned}
G:&=\{M\in\text{GL}_{n-d}(\R)\mid \\
&\qquad\qquad MA_i=A_i\text{ for all }n-m+2\leq i\leq n+1\}\\
\end{aligned}
\end{equation}
Observe that $G$ is isomorphic to 
\begin{equation}
\label{Gmodel}
\left\{\begin{pmatrix}
Id &B\\
0 &C
\end{pmatrix}\mid B,C\in \text{M}_m(\R)\right\}
\end{equation}
In particular, every fiber of $GT$ contains a configuration $\Lambda'$ which satisfies the stronger constraints
\begin{enumerate}[(i)]
	\item $\Lambda'$ is balanced.
	\item $\Lambda'_i=\Lambda_i$ for all $i$ between $d+2$ and $n+1$.
\end{enumerate}
Let $\tilde \Omega_\Lambda$ the subset of $\Omega_\Lambda$ consisting of those $\Lambda'$. Then $GT$ is a diffeomorphism between $\tilde \Omega_\Lambda$ and $\Pi_v$.
\begin{remark}
	The same argument was used to prove Theorem \ref{thmLVMBdec}.
\end{remark}
In other words, the map
\begin{equation}
\label{Gtriv}
(\Lambda',M)\in\tilde\Omega\times G\longmapsto M\Lambda'\in\Omega_\Lambda
\end{equation}
is a trivialization for $GT$ so the following diagram is commutative
\begin{equation}
\label{cdGtrivializations}
\begin{tikzcd}[row sep=normal, column sep=scriptsize]
\Omega_\Lambda \arrow[rr,"GT"] \arrow[dd,"\simeq"' near start] \arrow[dr]& & \Pi_v\arrow[dr]\\
& \text{[}\Omega_\Lambda/H_\Lambda\text{]}\arrow[rr,"\mathscr{D}" near start]  & & \text{[}\Pi_v/J_v\text{]} \\
\tilde\Omega_\Lambda\times G\arrow[ur]\arrow[rr,"\text{1st projection}"] & & \tilde\Omega_\Lambda\arrow[ur]\arrow[from=uu, crossing over,"\simeq" near start] \\
\end{tikzcd}
\end{equation}
We just proved that
\begin{theorem}
	The map $\mathscr{D}$ is an orbibundle with fiber $G$ isomorphic to $\R^{(n-d)^2/2}$.
\end{theorem}

\bibliographystyle{plain}
\bibliography{QTGBib}

\end{document}